\documentclass[a4]{amsart}
\usepackage{amsmath,amssymb,amsthm,bm,mathrsfs}
\usepackage{subfigure}
\usepackage[dvipdfmx]{graphicx}
\usepackage[monochrome]{color}
\usepackage{tikz,xy}
\xyoption{all}

\bmdefine{\ba}{a}
\bmdefine{\be}{e}
\bmdefine{\bg}{g}
\bmdefine{\bk}{k}
\bmdefine{\bm}{m}
\bmdefine{\bp}{p}
\bmdefine{\bs}{s}
\bmdefine{\bt}{t}
\bmdefine{\bv}{v}
\bmdefine{\bw}{w}
\bmdefine{\by}{y}
\bmdefine{\bz}{z}
\bmdefine{\i}{i}
\bmdefine{\j}{j}

\DeclareMathOperator{\supp}{supp}
\DeclareMathOperator{\Int}{int}
\DeclareMathOperator{\vol}{vol}

\DeclareMathOperator{\card}{card}
\DeclareMathOperator{\SP}{SP}
\DeclareMathOperator{\SR}{SR}

\newcommand{\sci}{\wedge}
\newcommand{\WP}{\mathcal{W}\mathcal{P}}

\newcommand{\R}{{\mathbb R}}

\newcommand{\Z}{{\mathbb Z}}

\newcommand{\N}{{\mathbb N}}

\newcommand{\T}{{\mathcal T}}

\newtheorem{theorem}{Theorem}[section]
\newtheorem{rem}{Remark}[section]
\newtheorem{lem}[theorem]{Lemma}
\newtheorem{cor}[theorem]{Corollary}
\newtheorem{prop}[theorem]{Proposition}
\theoremstyle{definition}
\newtheorem{defin}{Definition}%[section]
\newtheorem{ex}[theorem]{Example}
\makeatletter
  
  \@addtoreset{equation}{section}
\makeatother

\begin{document}
\title{Overlapping substitutions and tilings}
\author{Shigeki Akiyama, Yasushi Nagai, Shu-Qin Zhang*}
\address[S.~Akiyama]{Institute of Mathematics, University of Tsukuba, 1-1-1 Tennodai, Tsukuba, Ibaraki, zip:305-8571, Japan}
\email{akiyama@math.tsukuba.ac.jp}
\address[Y.~Nagai]{Center for General Education, Shinshu University, 3-1-1 Asahi, Matsumoto, Nagano, zip:390-8621, Japan}
\email{ynagai@shinshu-u.ac.jp}
\address[S.-Q.~Zhang]{School of Mathematics and Statistics, Zhengzhou University, 100 Science Avenue, Zhengzhou, Henan 45001, People's Republic of China}
\email{sqzhang@zzu.edu.cn}
\thanks{* The corresponding author}

\begin{abstract}
We generalize the notion of (geometric) substitution rule to obtain overlapping substitutions. 
%Such overlapping systems attract  attentions from researchers. 
%For example, Bernoulli convolutions is a self-similar measure defined by 
%heavy overlapping substitution.
%In this paper, we study a
%`light' but curious overlap. 
Our motivating example is the substitution presented in \cite{DBZ}, which features a substitution matrix with non-integer entries.
We give the meaning of such a matrix by showing that the right Perron--Frobenius eigenvector encodes the patch frequency of the resulting tiling. The patch frequencies are shown to be uniformly convergent, implying that the corresponding dynamical system is uniquely ergodic. Under mild assumptions, we further prove that the associated expansion constant is always an algebraic integer.
    In general, overlapping substitutions may yield a patch with illegal (partial) overlaps of tiles, even if it is locally consistent.
    We provide a sufficient condition for an overlapping substitution to be consistent, ensuring that no such illegal tiles emerge.  Finally, we construct many intriguing one-dimensional overlapping substitutions and present higher dimensional examples from Delone multi-sets with inflation symmetry.
\end{abstract}

\maketitle

%\begin{abstract}
%We extend the classical concept of geometric substitution rules by introducing a framework for overlapping substitutions. Such systems are of growing interest in symbolic dynamics and tiling theory, particularly in relation to self-similar measures like Bernoulli convolutions, which arise from heavily overlapping structures. In contrast, this work investigates a lighter class of overlaps that nevertheless exhibit nontrivial combinatorial and geometric behavior.

%Motivated by the substitution introduced in \cite{DBZ}, whose substitution matrix features non-integer entries, we provide a statistical interpretation by demonstrating that the right Perron--Frobenius eigenvector encodes the patch frequencies of the resulting tiling. Under mild conditions, we further show that the associated expansion constant is an algebraic integer.

%Overlapping substitutions may produce globally inconsistent patches despite local compatibility. To address this, we establish a sufficient condition ensuring global consistency and coherence of the generated tilings. Several illustrative one-dimensional examples are presented, and we conclude with a construction method for overlapping substitutions derived from Delone multisets with inflation symmetry.
%\end{abstract}

\section{Introduction}
%\textcolor{red}{Suggestion of introduction by Yasushi:}
The discovery of quasicrystals by Shechtman \cite{Shechtman:84} forces researchers to study the new stable phase of materials. Bragg peaks observed in the diffraction pattern of a quasicrystal show its internal long-range order, but they may have impossible symmetry with respect to the crystallographic restriction.
Along this line, its mathematical foundation has also been developed and coined as ``Mathematics of Aperiodic Order", see \cite{Baake-Grimm:13}. 
Quasicrystals are modeled by mathematical objects such as tilings and Delone sets, and the diffraction patterns are modeled by Fourier transforms for Radon measures and distributions.

Another context is the theory
of dynamical systems and ergodic theory. Words (sequences) and tilings offer
examples of ergodic transformations and group actions (\cite{Ledrappier:92},\cite{Queffelec:87}, etc.). General dynamical systems often have self-inducing structures, and words
and tilings give simple models
with such structures.
The dynamical spectra vary
for examples and various
peculiar examples have
been constructed by considering
words and tilings.

The self-inducing structure
for words and tilings is a useful foothold for the study
of their diffraction and dynamical spectra.
On the diffraction side,
it gives us the renormalization equations (\cite{Manibo}), from
which spectra are studied. On the dynamical side, the coincidence condition, discovered by Keane(\cite{Keane:71, Sirvent-Solomyak:02}), offers a sufficient condition for the spectra to be pure point.

The self-inducing structure is a kind of symmetry that is not described by groups. In the narrowest sense, it is described by the word substitution (the symbolic side) and the stone inflation (the geometric side). The latter is the so-called `inflation-and-subdivision rule'' which
is given by equations of the form
\begin{equation}
\label{SSE}
\beta A_j = \bigcup_{i=1}^{\kappa}(A_i+D_{ij})=\bigcup_{i=1}^{\kappa} \bigcup_{d\in D_{ij}} (A_i+d),
\end{equation}
where $A_i$'s are tiles and $D_{ij}$'s are finite subsets, called digits,  of the ambient Euclidean space.
We know that several classical examples of self-inducing structures are not in this realm. For example, Penrose
substitution in Figure \ref{PenroseSubst} and Ammann-Beenker
substitution in Figure \ref{AmmannSubst} have ``protruding'' or ``overlapping'' and do not
satisfy \eqref{SSE}.
The main target of this paper is to understand such a structure.
%\textcolor{red}{suggestion finished}
{\color{blue} The Penrose substitution and
the Ammann-Beenker substitution could be reduced
to stone inflations
by taking suitable sub-tiles }or considering its limit shapes having fractal boundaries, and in general,
by \cite{Solomyak_pseudo-self-similar},
{\color{blue}the induced structures that appear in this paper are always reduced to one described by stone inflations.
However, the reduction procedure is not always simple}: the associated set equation (\ref{SSE}) may become involved; the number of tiles generally increases; and the tiles often have fractal boundaries or are disconnected. It is advantageous to be able to handle the original substitution, which often has polygonal tiles.

\begin{figure}
\centering
\includegraphics[width= 4.5cm]{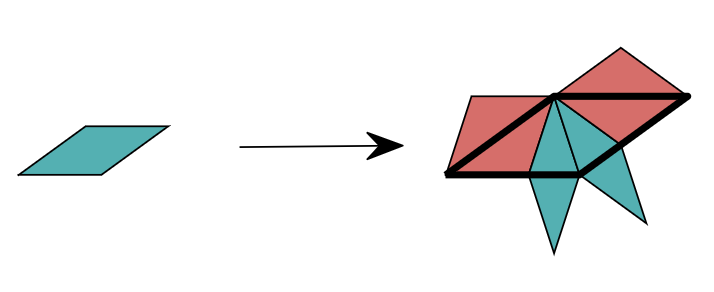}\quad\quad\includegraphics[width= 4.5cm]{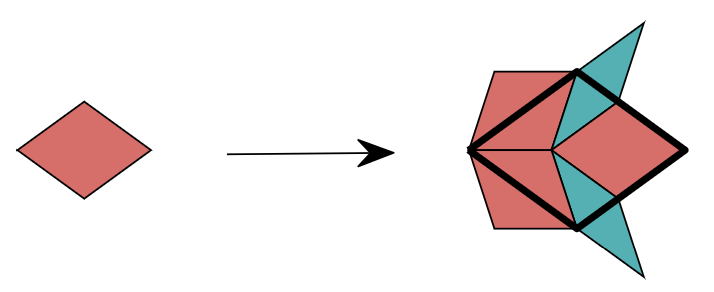}
\caption{Penrose tiling substitution\label{PenroseSubst}}
\end{figure}

\begin{figure}
\centering
\includegraphics[width= 4.5 cm]{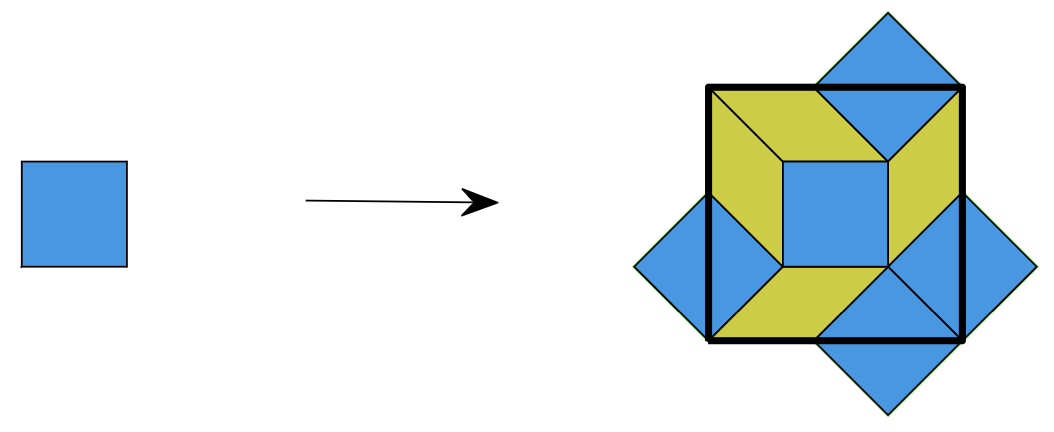}\quad\quad\includegraphics[width= 4.5 cm]{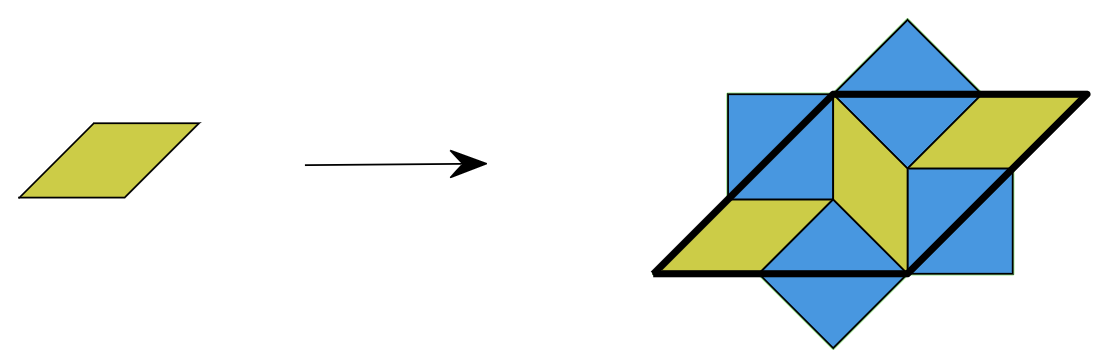}
\caption{Ammann-Beenker tiling substitution\label{AmmannSubst}}
\end{figure}

\begin{figure}
\centering
\includegraphics[width=8cm]{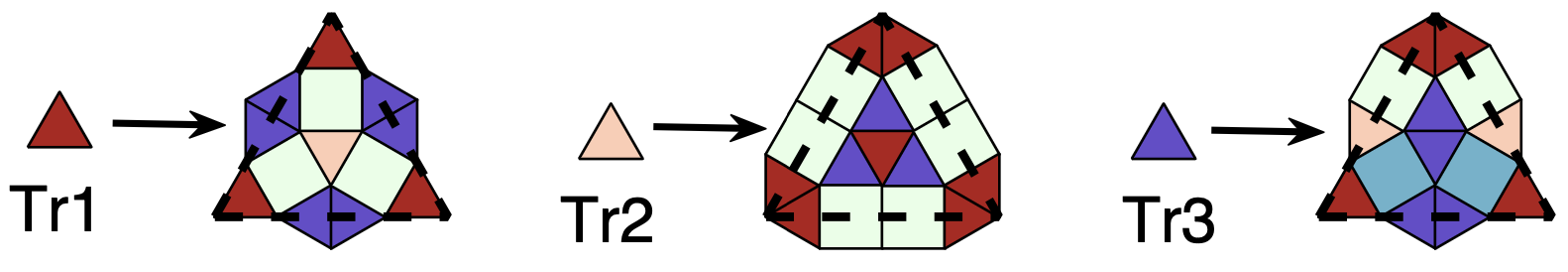}\quad\includegraphics[width=6cm]{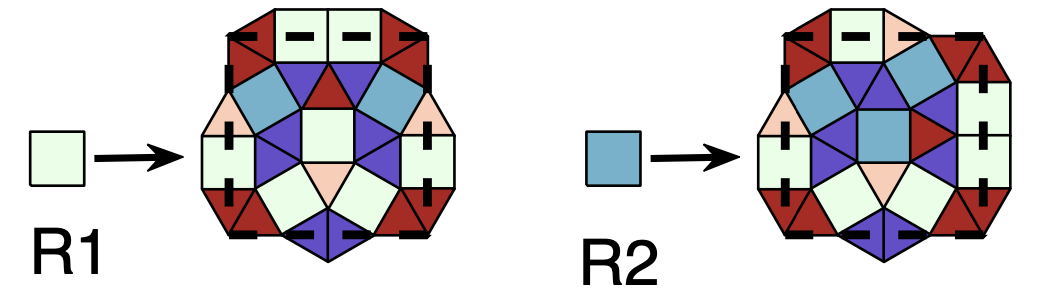}
\caption{Square Triangle tiling  substitution\label{SquareTriangleSubst}}
\end{figure}

\begin{figure}
\centering
\includegraphics[width=4 cm]{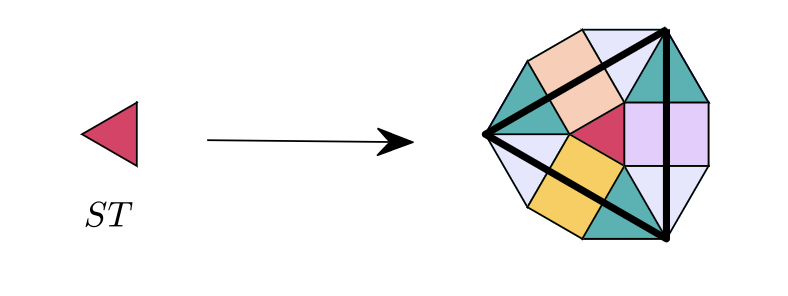}~\includegraphics[width= 4 cm]{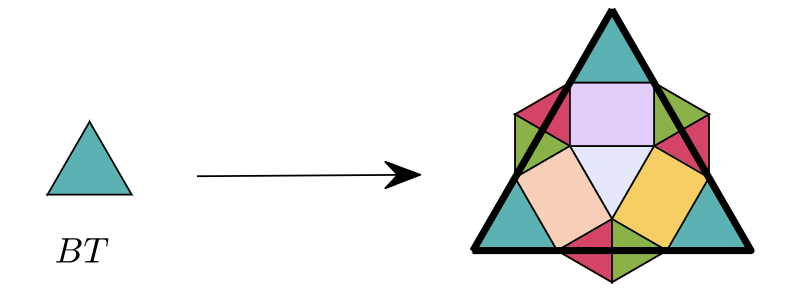}\quad\includegraphics[width= 4 cm]{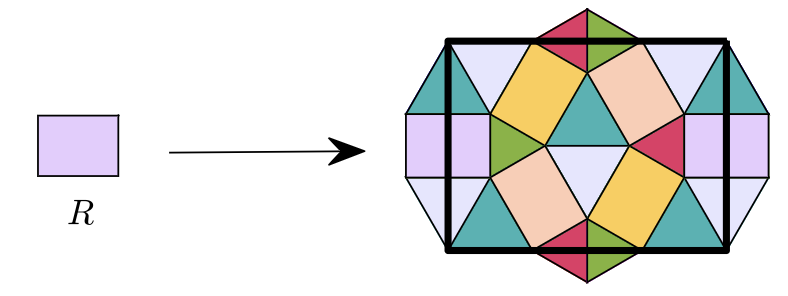}
\caption{Bronze-mean tiling substitution\label{BronzeSubst}}
\end{figure}

In this paper, we study the general theory of substitutions
called \emph{overlapping substitutions} that are not always stone inflations but may have ``protruding''.
{\color{blue}During the study, we realized that even the meaning of the substitution matrix is not yet clear. For example, the motivating example presented in Figure \ref{BronzeSubst}, introduced by Dotera-Bekku-Ziher \cite{DBZ}, which was later generalized and applied, e.g., in \cite{NZD,MKD,MKTMD}, seems unreasonable by the usual theory.}
% For example, see Figures \ref{SquareTriangleSubst} and \ref{BronzeSubst}.
% Figure \ref{SquareTriangleSubst} is found by Schlottman, whose descriptions are studied in \cite{Baake-Grimm:13} and \cite{Frettloeh:11}.
%It is not known to be mutually locally derivable from a tiling satisfying (\ref{SSE}). 
%Therefore the situation around this tiling is not crystalline clear yet.  
%For example,  the substitution in Figure \ref{BronzeSubst} was introduced by Dotera-Bekku-Ziher \cite{DBZ}. 
The expanded tiles are depicted by bold lines in the figure.
% We can iterate this substitution seemingly without problem (see Figure \ref{BronezeIte}). 
Identifying rotated tiles, the substitution shown in the figure suggests its associated substitution matrix being
\begin{equation}
\label{DoteraMatrix}
\begin{pmatrix}
1& 3 & 4 \\
3& 4 & 8 \\
3/2 & 3 & 5
\end{pmatrix}
\end{equation}
which contains a non-integer entry $3/2$. Here, the 1st column means that the inflation of ST tile by the bronze mean substitution contains 1 ST tile, 3 BT tiles, and 3/2 R tiles, as we count the tiles within the bold line triangle in Figure \ref{BronzeSubst}.
{\color{blue}
Moreover, generalizations of the Bronze mean substitution,
which we obtained in this study and are depicted in
Figure \ref{DZ},
\begin{figure}[h]
    \centering
\includegraphics[width=12 cm]{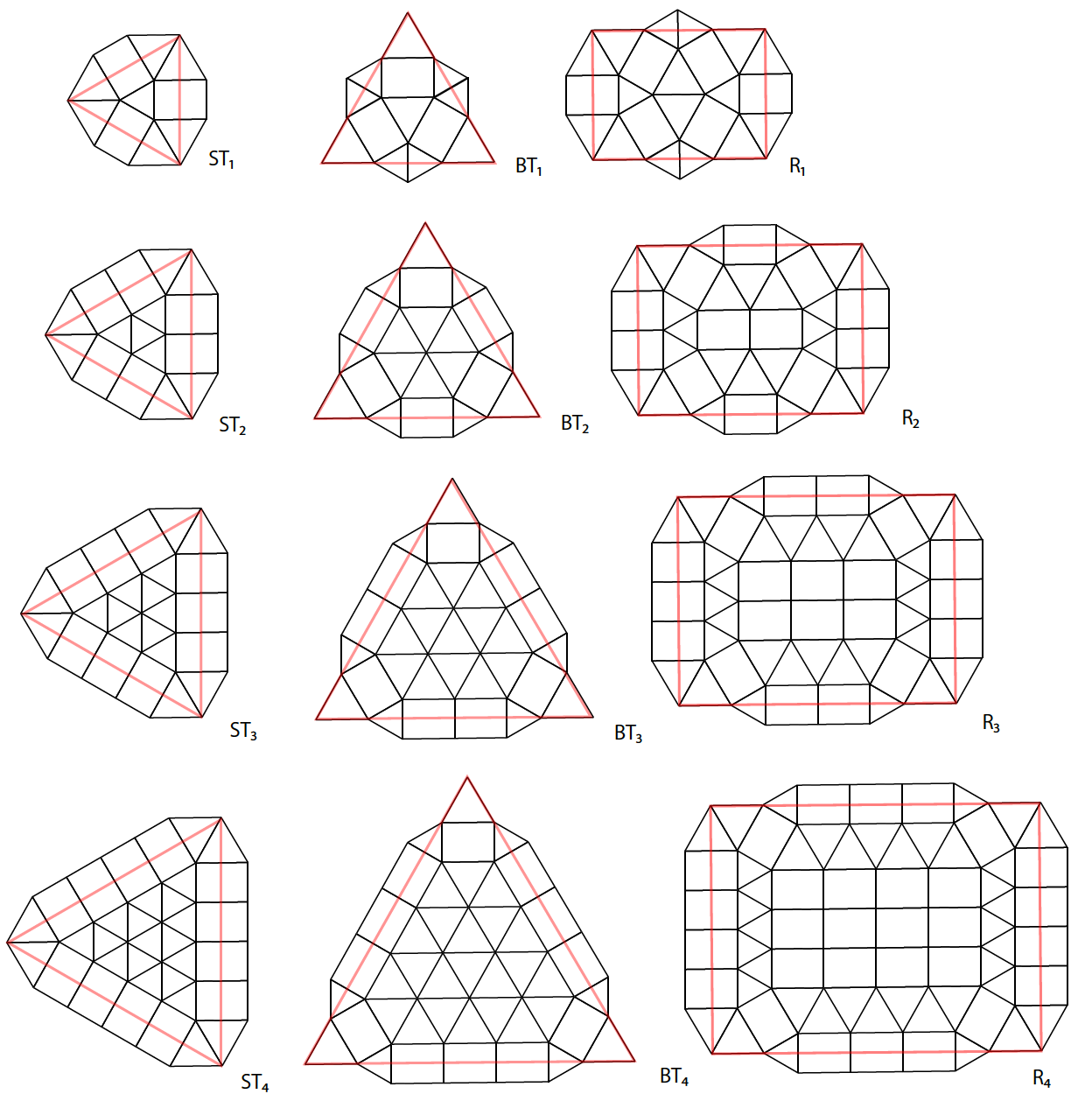}
    \caption{DZ substitutions of level $n=1,2,3,4$\label{DZ}}
\end{figure}
are a family of substitutions whose substitution matrices are
$$
\begin{pmatrix}
n^2 & 3 & 4n\\
3 & (1+n)^2 &4(1+n)\\
\frac{3n}2&\frac{3(1+n)}2& 3+n+n^2\\
\end{pmatrix}.
$$
These matrices are also non-integer valued and have surprising commutativity.
In Example \ref{DZsubst}, we discuss this example in detail.}
{\color{blue} We observe that the existing theoretical framework does not account for this situation: it always assumes that the substitution matrix has only integer entries.}
%We realize that the existing theory has not anticipated such a situation: the substitution matrices are
%assumed to have only integer entries. 
What information do such matrices with non-integer entries contain? This is the first question we address in this paper.

\textcolor{blue}{In particular, we will prove that the right Perron--Frobenius eigenvector encodes the tile frequencies (Theorem \ref{thm_uniform_tile_freq}). This may be assumed to be true in the literature, but not obvious because the overlaps of supertiles becomes larger and larger when we consider higher-level supertiles. We have to count tiles so that there is neither excess nor shortage. Theorem \ref{thm_uniform_tile_freq} seems less trivial when we observe one or two-parameter families of overlapping substitutions, for which the matrices depend on the parameters (see Section \ref{Sec5}).  The right eigenvectors for matrices are independent of the parameters, for the examples in the section. Theorem \ref{thm_uniform_tile_freq} implies this is always the case.
Moreover, we will prove that the patch frequencies also converge uniformly and are computed by the right Perron--Frobenius eigenvectors (Theorem \ref{thm_uniform_patch_freq}). This implies the associated dynamical systems are uniquely ergodic.}

\textcolor{blue}{
The second question is more fundamental, although it came later than the above discussion, because there are very few known results. It is
about what we call consistency.} 
{\color{blue}A substitution is said to be consistent if there are no illegal (partial) overlaps in all of its iterations.
For example, Figure \ref{Non-consis}
shows an example of illegal overlaps that emerge at the third step of iteration, but not at the first and second steps. Another example is the substitution in Figure \ref{BronzeSubst}, which can seemingly be iterated without any problems, as shown in Figure \ref{BronezeIte}. However, we must be careful before concluding that there are no illegal (partial) overlaps of tiles. We will prove, in the end, that this is the case, but we observe that the boundaries of the iterated patterns evolve in a complicated way, so that two distant parts might intersect illegally after sufficiently many iterations. In fact, in the iteration of another substitution (Figure \ref{fig:inflation_complicated-boundary}), two ``arms'' in the boundaries evolve to intersect each other.
We have to check whether such an intersection occurs, and if it does, we have to verify that it does not create illegal overlaps. In Section \ref{Sec4}, we discuss how to prove consistency for specific examples, including the Bronze-mean tilings (Figure \ref{BronzeSubst}) and Square triangle tilings (Figure \ref{SquareTriangleSubst}), using the method of fractal geometry.}
%{\color{red}In the course of this consistency analysis, in Example \ref{DZsubst}, we generalize the Bronze mean tiling substitution and 
%obtain a family of consistent overlapping substitutions with a surprising commutativity.}

%The second question is about what we call consistency: for example, seemingly, the substitution in Figure \ref{BronzeSubst} can be iterated without any problems, as in Figure \ref{BronezeIte}, but we have to be careful when we conclude that there are no illegal (partial) overlaps of tiles. 
%The substitution in Figure \ref{BronzeSubst} will probably yield no illegal overlaps in its iteration, but we observe that the boundaries of iterated patterns evolve in a complicated way so that two distant parts may intersect illegally after sufficiently many iterations.
%In fact, in the iteration of another substitution, two ``arms'' in the boundaries evolve to intersect each other (Figure \ref{fig:inflation_complicated-boundary}). 
%We have to check whether such an intersection happens or not, and if it happens, then we have to check that it does not create illegal overlaps, i.e., the substitution is consistent. 
%In this paper, we also discuss how to prove the consistency for examples. 
% We notice that the same question in Figure \ref{SquareTriangleSubst} exists as well. 
\begin{figure}
    \centering
    \includegraphics[width=0.7\linewidth]{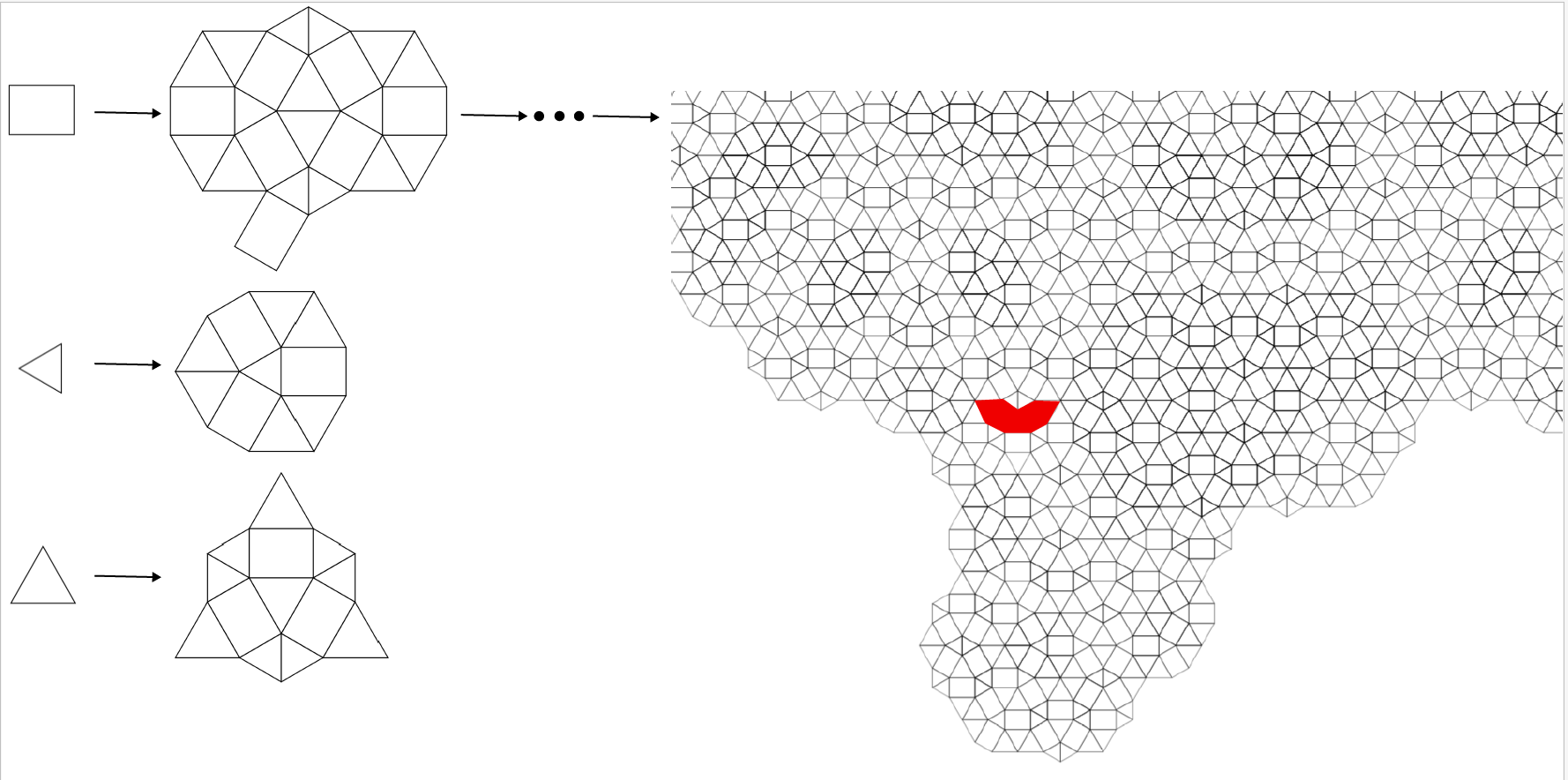}
    \caption{An example of overlapping substitution with complicated boundary growth. The right-hand side is the lower part of an inflation of the rectangle. We observe a ``hole'', which is indicated by the red color.}
    \label{fig:inflation_complicated-boundary}
\end{figure}
\begin{figure}[h]
\includegraphics[width=2.5cm]{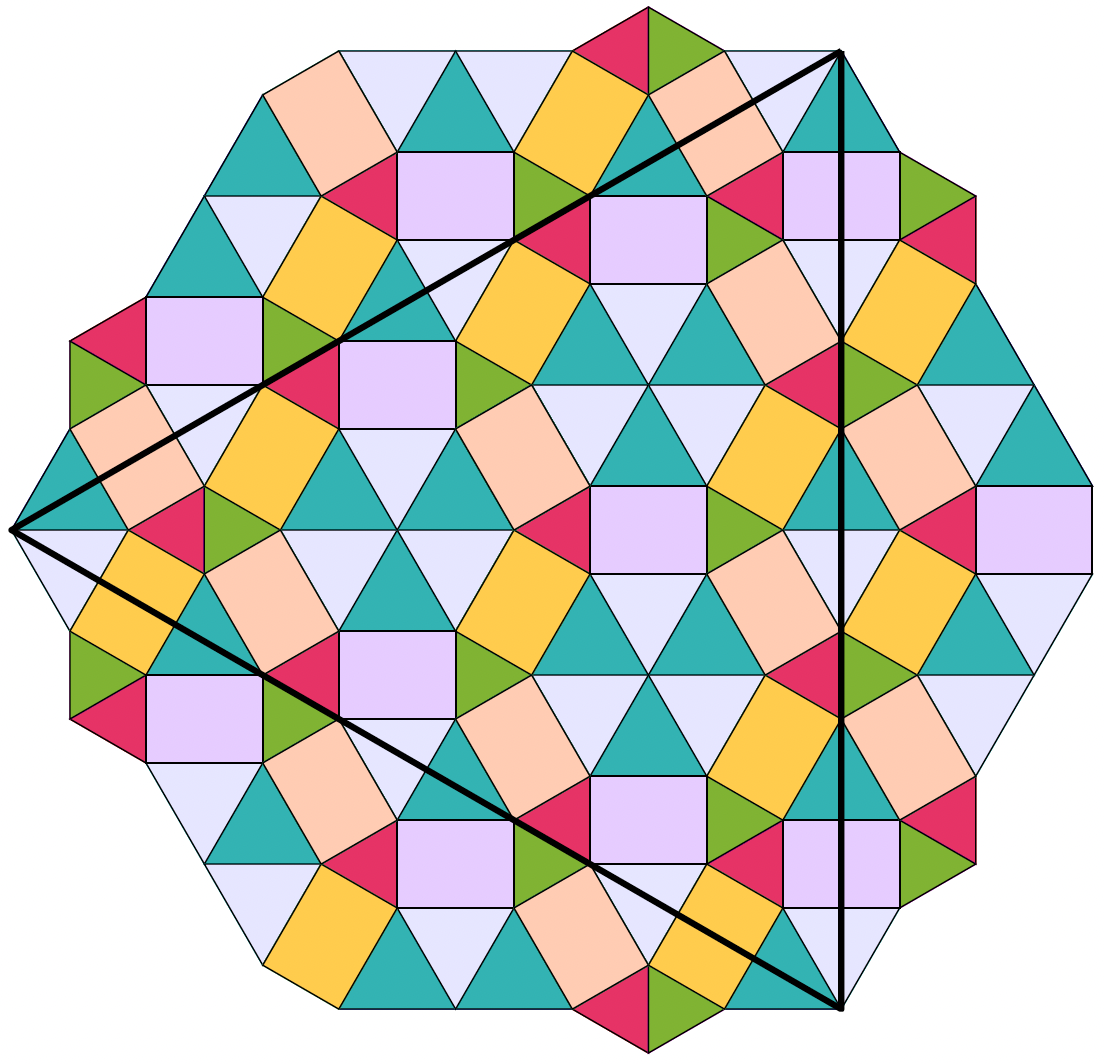}\quad \quad \quad \quad\quad \quad  \includegraphics[width= 2.8 cm]{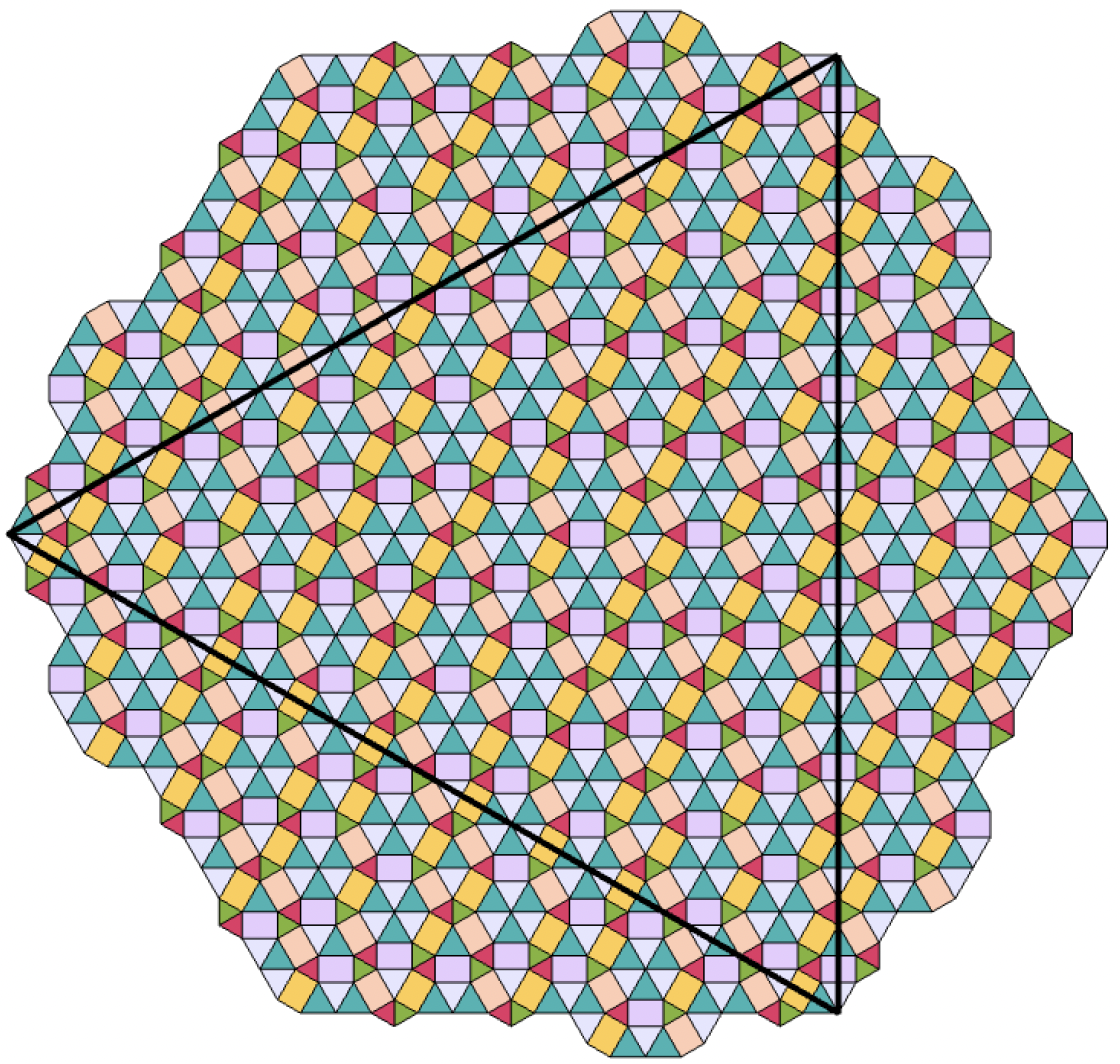}\\
\includegraphics[width=3 cm]{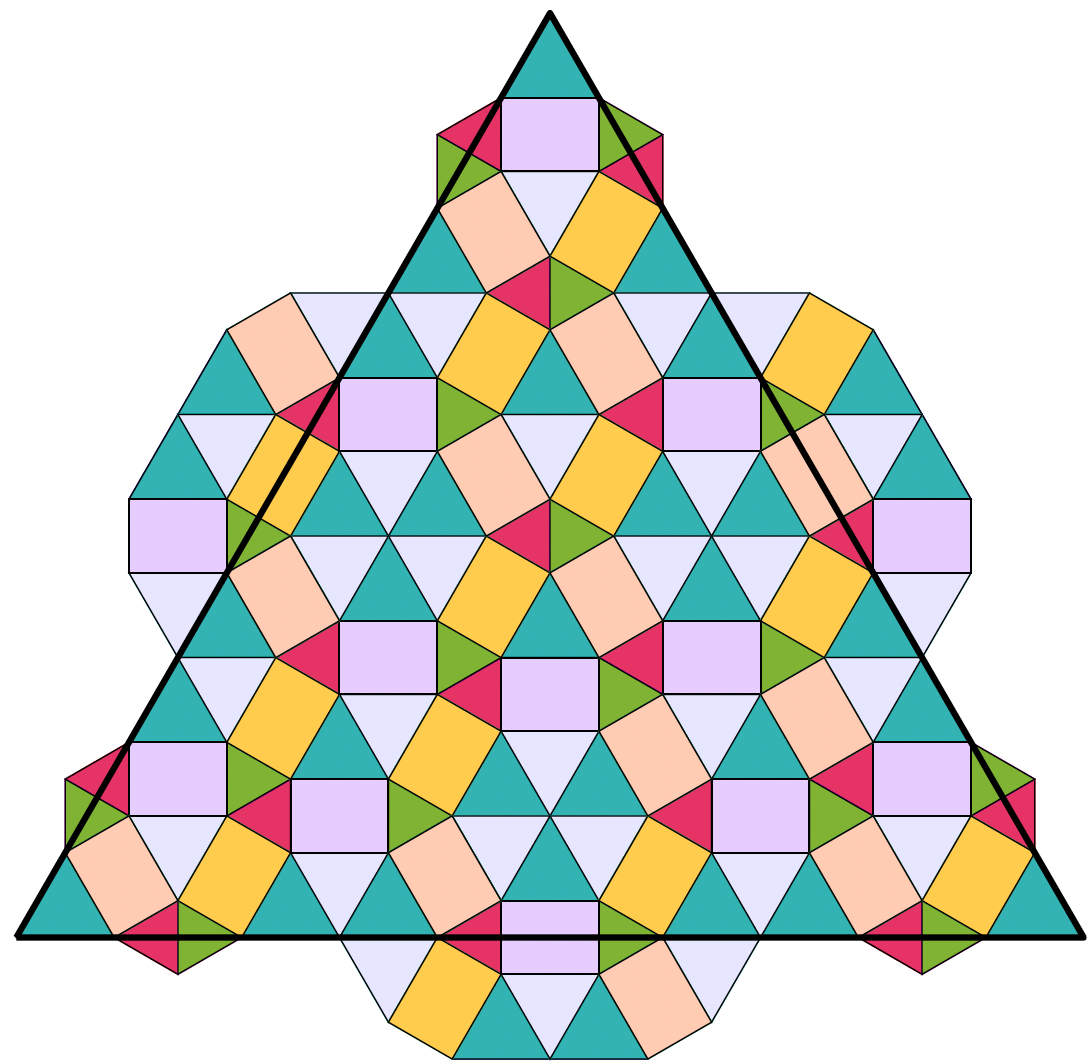}\quad \quad \quad \quad \quad \quad \includegraphics[width=3 cm]{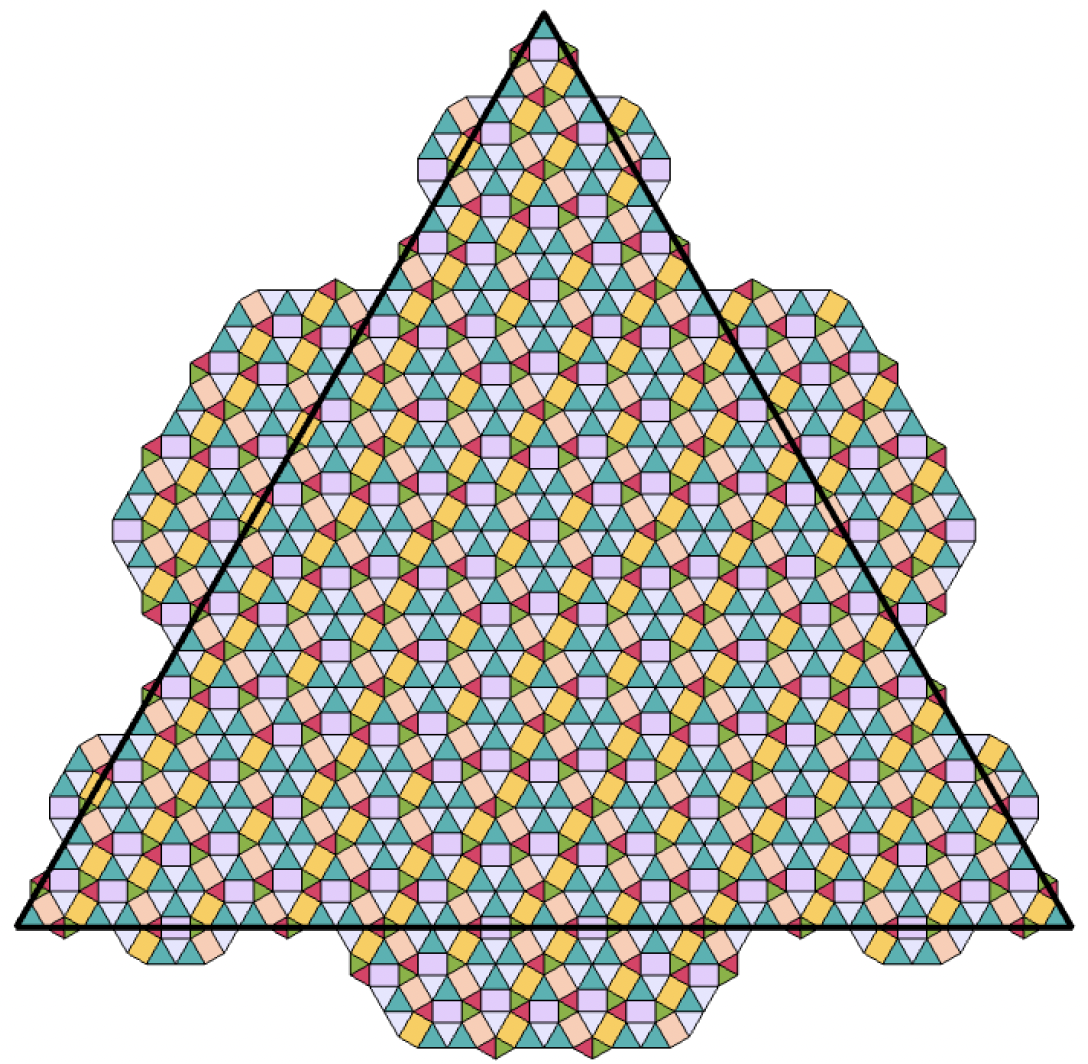}\\
\includegraphics[width=3.2 cm]{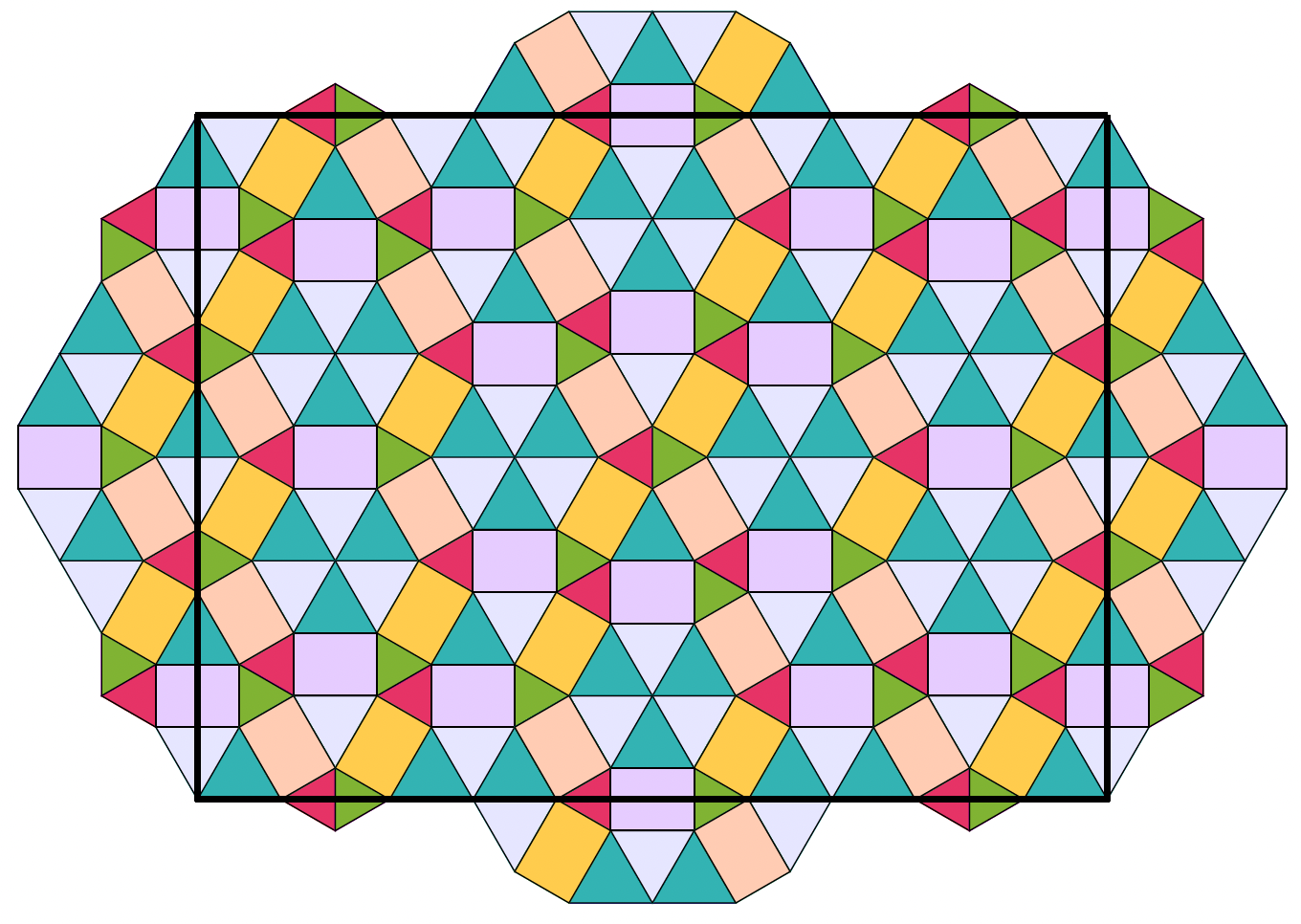}\quad \quad\quad \quad  \quad \quad \includegraphics[width=3.4 cm]{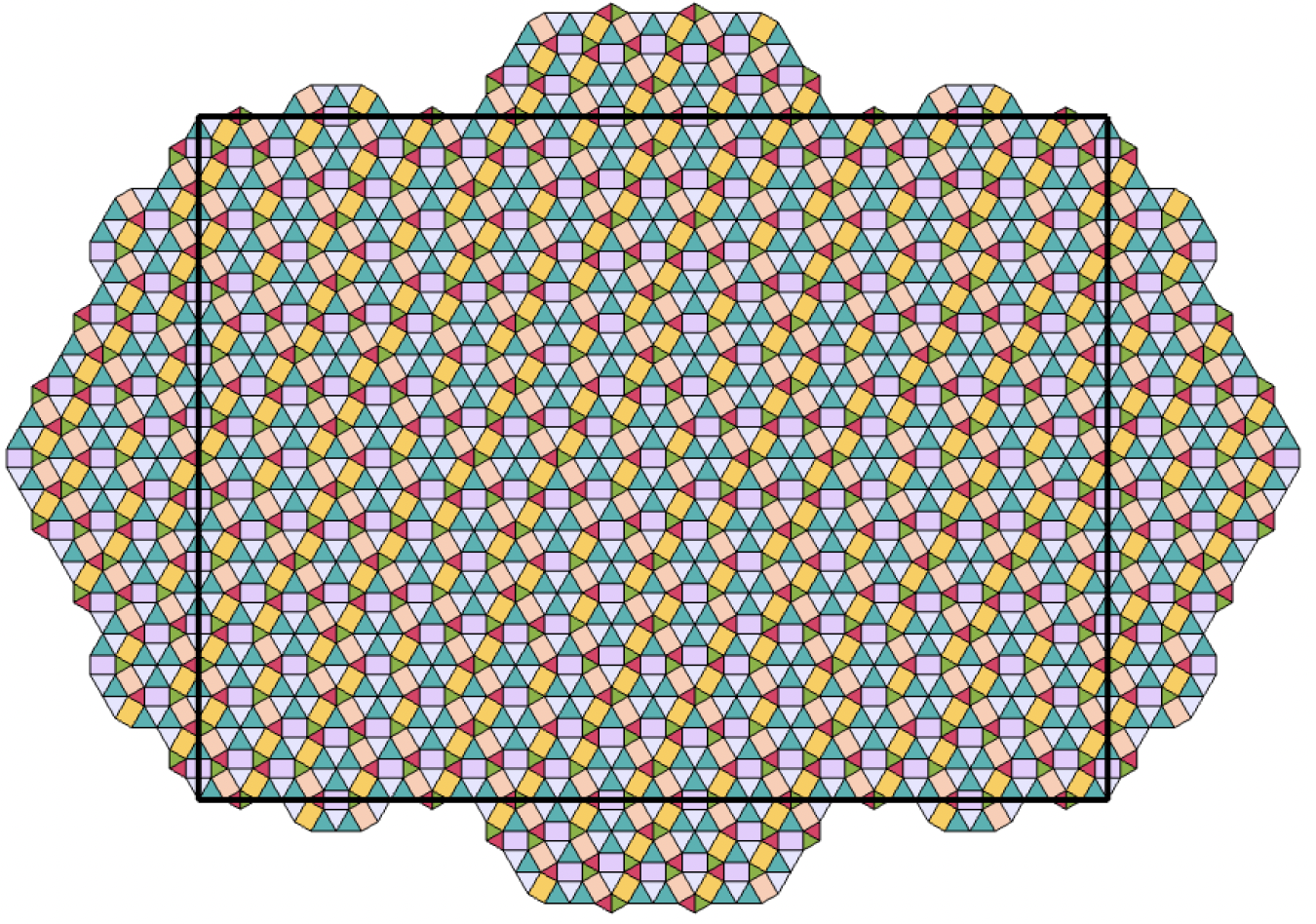}
\caption{The second and third iteration of Bronze-means `tiles' starting with `ST', `BT' and `R'.}\label{BronezeIte}
\end{figure}

% Once we have shown that such an overlapping substitution rule is globally consistent and produces an actual tiling, it is mutually locally derivable from a tiling satisfying (\ref{SSE})  by the result of Solomyak \cite{Solomyak_pseudo-self-similar}. 
% However, the associated set equation (\ref{SSE}) may become involved. The number of tiles generally increases, the tiles often have fractal boundaries, and the tiles can be disconnected. It is 
% meaningful to give a statistical meaning of the substitution matrix having non-integer entries, which
% allows us a direct and simple understanding of the tiling.
{\color{blue}To gain a better understanding of the overlapping substitution, we construct several interesting examples in Section \ref{Sec5} as well. Here, we present a one-dimensional example.}
%We also discuss constructions
%of interesting overlapping substitutions in Section \ref{Sec5}. 
%We start with
%one-dimensional examples,
%such as the following:

\begin{ex}[One dimensional symbolic weighted substitution]\label{ex1_one-dim-substi}
Let $\sigma$ be a substitution over $\{a,b,c\}$ as follows.
$$\sigma: \left\{\begin{aligned}
a&\longrightarrow \frac{a}{2}ba\\
b&\longrightarrow c\frac{a}{2}\\
c&\longrightarrow b\frac{a}{2}
\end{aligned}
\right.
$$ 
\end{ex}

\begin{figure}[h]
\centering
\includegraphics[width=7 cm]{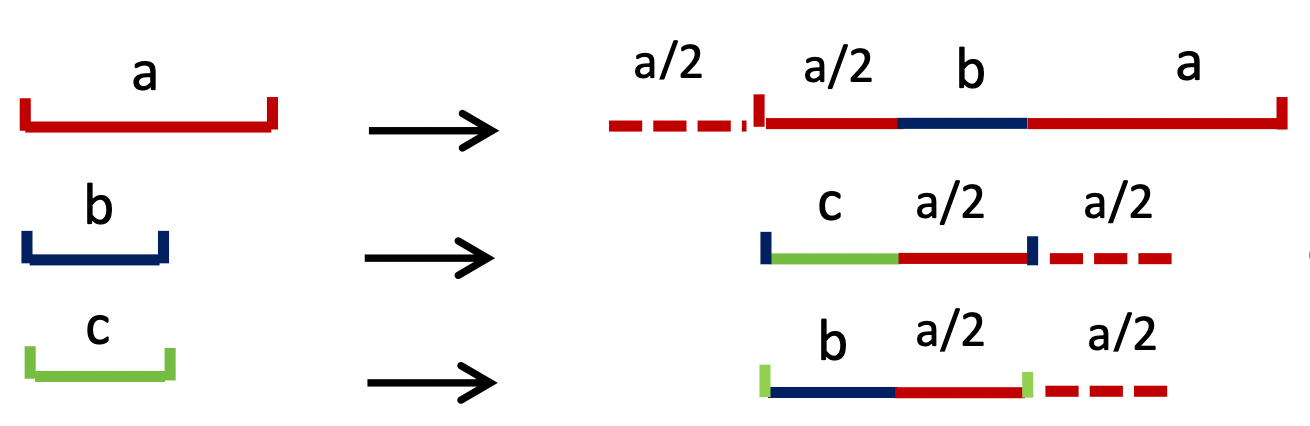}
\caption{Substitution of Example 1.1}\label{Ex1Subst}
\end{figure}

As we will see in Section \ref{Sec5},
we can realize this symbolic substitution as a geometric overlapping substitution, depicted in Figure \ref{Ex1Subst}.
The notation $a/2$ is understood similarly as in Figure \ref{BronzeSubst}, which means that half of the tile $a$ is outside the enlarged original tile (see the dotted line in Figure \ref{Ex1Subst}). The substitution matrix of $\sigma$
is
\begin{align}
    \begin{pmatrix}
3/2& 1/2 & 1/2 \\
1& 0 & 1 \\
0 & 1 & 0
\end{pmatrix}\label{matrix_ex1_symbolic_weighted_substi}
\end{align},
which has non-integer entries again. 
We can iterate this substitution naturally because the overlapping symbols match without causing problems {\color{blue} if we consider concatenating two $\frac{a}{2}$ by $a$}. For example, if we start from a patch $ba$

\begin{align*}
    &\sigma(ba)=c\frac{a}{2}\frac{a}{2}ba=caba\\
    &\sigma^2(ba)=b\frac{a}{2}\frac{a}{2}bac\frac{a}{2}\frac{a}{2}ba=babacaba\\
    &\sigma^3(ba)=c\frac{a}{2}\frac{a}{2}bac\frac{a}{2}\frac{a}{2}bab\frac{a}{2}\frac{a}{2}bac\frac{a}{2}\frac{a}{2}ba=cabacabababacaba
\end{align*}
%Here we consider concatenating two $\frac{a}{2}$ by $a$.
See the beginning of Section \ref{Sec5} for further details.

In Section \ref{subsec_construction_general_dim}, we will also construct overlapping substitutions in general dimensions. Another advantage of allowing overlaps is that
it is easier to construct examples of substitution.

It is noteworthy that in one dimension, we have yet another fascinating overlapping example that is not a tile but a measure.
Bernoulli convolution is a self-similar measure defined by overlapping interval maps. The support of the measure is the interval itself, but the resulting measure may not be absolutely continuous to the Lebesgue measure, i.e., the overlaps may localize at some small subset, see \cite{PSS}.
We expect that the systematic study of overlapping substitution results in a better understanding of Bernoulli convolutions.

% We think it is a good time to initiate a systematic
% study of such ``overlapping" substitutions with mathematical rigor.
% There are two target questions in this paper.  
% The first is how we can show the consistency of the substitution, whether we can iterate the substitution indefinitely without producing a contradictory overlap. 
% The second is to give the meaning of the substitution matrix and its eigenvectors
% for the overlapping substitution. This gives the basis to further statistical and ergodic properties of tilings as well as their diffraction spectra.
% This article is 
% the first attempt to give possible definitions of overlapping substitution and 
% explore their basic properties. We meet interesting new examples on the way. 
% Our target in this article is restricted to the case when the overlaps are small and they do not self-replicate, i.e., 
% the overlaps do not overlap again by inflation-subdivision process. 
% Of course, the overlaps may not be small in general, e.g., in the case of Bernoulli convolution, and they are also very interesting 
% but we leave them to future work, see \S \ref{OpenProblem}.

The rest of the paper is organized as follows. In Section \ref{Sec2}, we define overlapping substitutions and substitution matrix, and apply the Perron--Frobenius theory to obtain
information from the substitution matrices.
The righ Perron--Frobenius eigenvector gives the frequency of prototiles, as expected.
As an auxiliary step, we introduce the theory of weighted substitutions. \textcolor{blue}{Moreover, we will prove that the patch frequencies converge uniformly and are obtained by right Perron--Frobenius eigenvectors for appropriate modifications of the original substitution. This implies the unique ergodicity for the associated dynamical systems.}
In Section \ref{Sec3}, we prove another expected result: under mild conditions, the expansion constant of an overlapping substitution must be an algebraic integer.
In Section \ref{Sec4}, we discuss the consistency problem of two-dimensional overlapping substitutions. We show that the known overlapping substitutions are consistent using the fractal geometry method.
In Section \ref{Sec5}, we first show the procedure of constructing an overlapping substitution in one dimension. We present several parametric families of overlapping substitutions whose tile sizes vary according to the parameters. 
Then we construct overlapping substitutions from Delone sets with inflation symmetry, in an arbitrary dimension.
{\color{blue}Section \ref{OpenProblem} presents related problems that may be studied in future work. As a complement to Section \ref{Sec2}, we provide the proof of Proposition \ref{prop_frequency_weighted-substi} in Section \ref{appendix}.}

\section{The theory of weighted substitutions and overlapping substitutions}\label{Sec2}

\subsection{Definitions of tiles, patches and tilings}
Let $L$ be a finite set. An $L$-labeled tile is a pair $T=(S,l)$ where $l\in L$, $S\subset\mathbb{R}^d$ is compact non-empty and
$S=\overline{S^{\circ}}$, that is, the closure of the interior of $S$ coincides with $S$.
If we do not specify the set $L$ of labels, the $L$-labeled tiles are just called labeled tiles.
There are also unlabeled tiles, that are by definition a compact and non-empty $S\subset\mathbb{R}^d$ such that
$S=\overline{S^{\circ}}$.
Both labeled tiles and unlabeled tiles are called tiles.
For a labeled tile $T=(S,l)$, we use a notation $\supp T=S$ and $\Int T=S^{\circ}$.
We define the translation of $T$ by $x\in\mathbb{R}^d$ via $T+x=(S+x, l)$.
For an unlabeled tile $S$, we use a notation $\supp S=S$ and $\Int S=S^{\circ}$.
We give a presentation that deals with both labeled and unlabeled tiles by these notations.

A set consisting of tiles is called a pattern.
A pattern $\mathcal{P}$  such that
$T_1,T_2\in\mathcal{P}$ and $\Int T_1\cap\Int T_2\neq\emptyset$ imply $T_1=T_2$ is called a patch.
We say two tiles $T_1,T_2$ in a pattern are illegal overlap if $T_1\neq T_2$ and
$\Int T_1\cap\Int T_2\neq\emptyset$.
A patch is a pattern without any illegal overlaps.
For a patch $\mathcal{T}$, we define its support $\supp\mathcal{T}$ via
\begin{align*}
     \supp\mathcal{T}=\bigcup_{T\in\mathcal{T}}\supp T.
\end{align*}
A patch $\mathcal{T}$ such that $\supp\mathcal{T}=\mathbb{R}^d$ is called a tiling.
A translation of a pattern $\mathcal{P}$ by a vector $x\in\mathbb{R}^d$ is defined via
\begin{align*}
    \mathcal{P}+x=\{T+x\mid T\in\mathcal{P}\}.
\end{align*}

\subsection{The theory of weighted substitutions}

{\color{blue}Our goal here is to obtain information on substitution matrices that may have
fractional, or even irrational entries. To do so, we can get
a clearer picture by taking a detour to
the theory of weighted substitution.
In fact, the substitution matrix for an overlapping substitution will be defined as one for the associated weighted substitution.
For the latter, it is easy to prove that the right Perron--Frobenius eigenvector gives the frequencies.
In this section, we study the general
theory of weighted substitutions.}

Let $\mathcal{B}$ be a finite set of tiles. We use the notation $\mathcal{B}+\mathbb{R}^d$ for the set of all
the translations of the tiles in $\mathcal{B}$, that is, the set of all tiles $T+x$ where
$T\in\mathcal{B}$ and $x\in\mathbb{R}^d$.
In this section, we consider
the weights of the tiles in patches. For example,
we start with two patches $\mathcal{P}_1=\{[0,1],[1,2]\}$
and $\mathcal{P}_2=\{[1,2],[2,3]\}$,  whose elements are intervals, and
we consider weights on each tile, for example
\begin{enumerate}
    \item the weight of $[0,1]$ in 
    $\mathcal{P}_1$ is $2/3$ and one of $[1,2]$
    in $\mathcal{P}_1$ is $3/4$, and
    \item the weight of $[1,2]$ in 
    $\mathcal{P}_2$ is $1/2$ and one of $[2,3]$
    in $\mathcal{P}_2$ is $1$.
\end{enumerate}
(The weights can be arbitrary real numbers greater than or equal to 0.)
Without weights, the union
$\mathcal{P}_1\cup\mathcal{P}_2$ is
$\{[0,1],[1,2],[2,3]\}$, but with weights,
if we take the union, the weights of the
overlapping tiles add up, and so the result is
a patch $\{[0,1],[1,2],[2,3]\}$ with a weight
$2/3$ on $[0,1]$, $5/4$ on $[1,2]$ and $1$ 
on $[2,3]$. The readers may notice that
this is exactly what the functions on
patches are. 
However, we do not deal with functions
$v_1\colon\mathcal{P}_1\rightarrow\mathbb{R}$
and
$v_2\colon\mathcal{P}_2\rightarrow\mathbb{R}$,
because these cannot be added since the
domains are different. 
Instead, we should consider maps
$v_1$ and $v_2$ into which we can plug
arbitrary tiles in $\mathbb{R}$.
The values for $v_1$ should be
\begin{align}
    v_1(T)=
    \begin{cases}
    2/3 &\text{if $T=[0,1]$}\\
    3/4 &\text{if $T=[1,2]$}\\
    0&\text{otherwise.}
    \end{cases}\label{def_example_for_wp_v1}
\end{align}
We endow values for $v_2$ in an appropriate way.
In what follows, we study such functions $v\colon\mathcal{B}+\mathbb{R}^d\rightarrow\mathbb{R}$, where $\mathcal{B}$ is a given finite
set of tiles.

Let us enumerate the tile in $\mathcal{B}$ so that
we have $\mathcal{B}=\{T_1,T_2,\ldots T_{\kappa}\}$.
% be a finite set of tiles in $\mathbb{R}^d$ and set
% $\mathcal{B}+\mathbb{R}^d=\{T+x\mid T\in\mathcal{B},x\in\mathbb{R}^d\}$.
A map $v\colon\mathcal{B}+\mathbb{R}^d\rightarrow\mathbb{R}$ is called a weighted pattern
with alphabet $\mathcal{B}$. 
The set of all weighted patterns with alphabet $\mathcal{B}$ is denoted by $\WP_{\mathcal{B}}$.
This is a vector space over $\mathbb{R}$.

{\color{blue}Weighted substitutions
are maps that act on weighted patterns.
This is a substitution with weights.
To capture the notion, consider the
usual Fibonacci substitution:
\begin{align*}
    \rho\colon &[0,\tau]\mapsto \{[0,\tau],[\tau,\tau+1]\},\\
    &[0,1]\mapsto\{[0,\tau]\},
\end{align*}
where $\tau$ is the golden ratio.

We can add arbitrary weights to the images:
\begin{align*}
    \xi\colon &([0,\tau],1)\mapsto \left\{\left([0,\tau],1\right),\left([\tau,\tau+1],\frac{1}{2}\right)\right\},\\
    &([0,1],1)\mapsto\left\{\left([0,\tau],\frac{1}{3}\right)\right\},
\end{align*}
where numbers such as $1$ and $\frac{1}{2}$
are weights. If the tiles that we plug into
$\xi$ have non-one weights, the weight is
multiplied  to the image:
\begin{align}
    \xi\colon ([0,\tau],r)\mapsto \left\{\left([0,\tau],r\right),\left([\tau,\tau+1],\frac{r}{2}\right)\right\},
    ([0,1],r)\mapsto\left\{\left([0,\tau],\frac{r}{3}\right)\right\}.
    \label{eq_weighted_Fibonacci}
\end{align}
We realize this idea of weighted substitution as follows, via using the above
framework for weighted patterns. 
% \textcolor{red}{
% Weighted substitutions are maps that act on weighted patches. For example, Thue--Morse substitution is a map $\rho$ on $\mathcal{B}+\mathbb{R}$, where}
% \begin{align*}
%     &\mathcal{B}=\{T_0,T_1\},\\
%     &T_i=([0,1],i), i=0,1.
% \end{align*}
% We would like to define $\rho$ as
% \begin{align*}
%     \rho(T_0,r)=\{(T_0,r),(T_1+1,r)\}\\
%     \rho(T_1,r)=\{(T_1,r),(T_0+1,r)\},
% \end{align*}
% where $r$ is the weight for the tiles. We will describe this idea in the framework of weighted patches, in such a way that the theory includes substitutions with overlaps.

As in the usual theory of tilings,
we should use a topology and convergences
of weighted tilings and weighted patches.}
On $\WP_{\mathcal{B}}$, we define the weak topology. Consider all the maps
\begin{align*}
    \tau_{T}\colon \WP_{\mathcal{B}}\ni v\mapsto v(T)\in\mathbb{R},
\end{align*}
where $T$'s are arbitrary elements of $\mathcal{B}+\mathbb{R}^d$. The weak topology
is the weakest topology that makes all these maps continuous. {\color{blue}Thus,  a net of
weighted patterns $(v_i)$ converges to a
weighted pattern $v$ if $v_i(T)\rightarrow v(T)$ for arbitrary $T\in\mathcal{B}+\mathbb{R^d}$. }

Next, we define the weighted substitutions. We use the following two notions.
For a weighted pattern $v$ with alphabet $\mathcal{B}$, its support
pattern is
\begin{align*}
 \SP(v)=\{T\in\mathcal{B}+\mathbb{R}^d\mid v(T)\neq 0\}.
\end{align*}
This is the set of tiles with non-zero
weights.
The region that such tiles cover is called the support region and is denoted by
\begin{align*}
\SR(v)=\overline{\bigcup_{v(T)\neq 0}\supp T}. 
\end{align*}
{\color{blue}For example, for the $v_1$
in \eqref{def_example_for_wp_v1}, we have
\begin{align*}
    &\SP(v_1)=\{[0,1],[1,2]\}\\
    &\SR(v_1)=[0,2].
\end{align*}

Note that
the map $\xi$ in \eqref{eq_weighted_Fibonacci} should be defined as a map that sends the tile $T=[0,\tau]$ to a weighted pattern  $\xi(T)$ 
which has the values
\begin{align*}
    &\xi(T)([0,\tau])=1\\
    &\xi(T)([\tau,\tau+1])=\frac{1}{2},
\end{align*}
and $\xi(T)(S)=0$ for any other $S\in\mathcal{B}+\mathbb{R}^d$. $\xi([0,1])$ is defined in a similar way.
The support patches are 
$\SP(\xi([0,\tau]))=\{[0,\tau],[\tau,\tau+1]\}$ and $\SP(\xi([0,1]))=\{[0,\tau]\}$.}

\begin{defin}
Consider a linear and expanding map
$\phi\colon\mathbb{R}^d\rightarrow\mathbb{R}^d$.
Consider also a map $\xi\colon\mathcal{B}\rightarrow\WP_{\mathcal{B}}$ such that, for each $T\in\mathcal{B}$,
\begin{enumerate}
 \item $\SP(\xi(T))$ is a finite set, and
\item $\xi(T)$ is a non-negative map, that is, for any $S\in\mathcal{B}+\mathbb{R}^d$, we have
$\xi(T)(S)\geqq 0$, which means that all the weights are non-negative.
\end{enumerate}
We call the triple
$(\mathcal{B},\phi,\xi)$ a pre-weighted substitution,
or we just simply say $\xi$ a pre-weighted substitution with alphabet $\mathcal{B}$. 
\end{defin}

At this stage we do not assume any relations between $\xi$ and $\phi$.
In what follows, we will define iterations of $\xi$, and we call $\xi$ a weighted substitution if, after applying arbitrary iterations of $\xi$ to any tile $T\in\mathcal{B}$, the support pattern is a patch, that is, there are
no illegal overlaps (Definition \ref{consistency_weighted}).

We extend the domain of pre-weighted substitution as for the usual substitution rules.
Recall that for a usual substitution $\rho$,
given a definition of the image $\rho(P)$ 
of a proto-tile $P$, we extend the domain
of $\rho$ via $\rho(P+x)=\rho(P)+\phi(x)$,
where $\phi$ is the expansion map.
In our context, for a weighted substitution
$\xi$, $T\in\mathcal{B}$ and $x\in\mathbb{R}^d$, set
\begin{align*}
 \xi(T+x)(S)=\xi(T)(S-\phi(x)),
\end{align*}
where $S$ is an arbitrary element of $\mathcal{B}+\mathbb{R}^d$. In this way, we obtain a map
\begin{align*}
\xi\colon\mathcal{B}+\mathbb{R}^d\rightarrow\WP_{\mathcal{B}}.
\end{align*}
 It is easy to prove that
\begin{align*}
 \SP(\xi(T+x))=\SP(\xi(T))+\phi(x).
\end{align*}

Given a weighted substitution $\xi$, we define a substitution matrix for it.
For a $v\in\WP_{\mathcal{B}}$ with finite $\SP(v)$ and $T\in\mathcal{B}$, we set
\begin{align*}
 \tau_T(v)=\sum_{x\in\mathbb{R}^d}v(T+x).
\end{align*}
This is the sum of all the weights of the tiles that are translates of $T$ in the support of $v$.
Note that the sum is finite, i.e., there are only finitely many positive terms. 
The substitution matrix for $\xi$ is the $\kappa\times \kappa$ matrix whose $(i,j)$ element is
$\tau_{T_i}(\xi(T_j))$.
A weighted substitution $\xi$ is said to be primitive if its substitution matrix is
primitive.

Next, we will further ``extend'' the domain for $\xi$ to $\WP_{\mathcal{B}}$.
\begin{lem}\label{lem1_frequency}
 For any $T\in\mathcal{B}+\mathbb{R}^d$, there are only finitely many $S\in\mathcal{B}+\mathbb{R}^d$
such that $\xi(S)(T)\neq 0$.
\end{lem}
\begin{proof}
 If $S\in\mathcal{B}+\mathbb{R}^d$ and $\xi(S)(T)\neq 0$, there are $P\in\mathcal{B}$ and
$x\in\mathbb{R}^d$ such that $S=P+x$, and
$\xi(P)(T-\phi(x))\neq 0$. We have an injection
\begin{align*}
   \{S\mid \xi(S)(T)\neq 0\}\ni S
    \mapsto (P,T-\phi(x))\in\{(P,U)\mid P\in\mathcal{B},U\in\SP(\xi(P))\}.
\end{align*}
\end{proof}

\begin{lem}
 For any $v\in\WP_{\mathcal{B}}$, the sum
\begin{align*}
   \sum_{T\in\SP(v)}v(T)\xi(T)
\end{align*}
is convergent with respect to the weak topology, that is, for any $S\in\mathcal{B}+\mathbb{R}^d$,
 there is a real number $\alpha$ such
that, for any  $\varepsilon>0$, 
there is a finite set $F\subset\SP(v)$ with
\begin{align*}
 \left|\sum_{T\in F'}v(T)\xi(T)(S)-\alpha\right|<\varepsilon
\end{align*}
for any finite $F'$ with $F\subset F'\subset\SP(v)$.
\end{lem}

\begin{proof}
  For each $S\in\mathcal{B}+\mathbb{R}^d$, by Lemma \ref{lem1_frequency}, the set
\begin{align*}
 \{T\in\SP(v)\mid \xi(T)(S)\neq 0\}
\end{align*}
is finite, and so
\begin{align*}
 \sum_{T\in\SP(v)}v(T)\xi(T)(S)
\end{align*}
is a finite sum.
\end{proof}

Given a weighted pattern $v$, we define its image by $\xi$, which is the result of applying $\xi$ to each weighted tile in $v$.
We set
\begin{align*}
 \xi(v)=\sum_{T\in\SP(v)}v(T)\xi(T)=\sum_{T\in\mathcal{B}+\mathbb{R}^d}v(T)\xi(T),
\end{align*}
and so we have a map
\begin{align*}
 \xi\colon\WP_{\mathcal{B}}\rightarrow\WP_{\mathcal{B}}.
\end{align*}
\begin{defin}\label{consistency_weighted}
    The pre-weighted substitution is said to be consistent if the support pattern for $\xi^n(\delta_T)$
    is a patch (that is, there are no illegal overlaps) for any $n>0$ and
    $T\in\mathcal{B}$,
    where $\delta_T$ is
    defined via
    \begin{align*}
        \delta_T(S)
        =\begin{cases}
            1 &\text{ if $S=T$}\\
            0 &\text{ otherwise.}
        \end{cases}
    \end{align*}
\end{defin}
\begin{lem}
 The map
\begin{align*}
 \xi\colon \WP_{\mathcal{B}}\ni v\mapsto \xi(v)\in\WP_{\mathcal{B}}
\end{align*}
is linear and continuous with respect to the weak topology.
\end{lem}
\begin{proof}
 That the map is linear is easy to be proved. To prove the continuity, 
take an arbitrary $S\in\mathcal{B}+\mathbb{R}^d$. By Lemma \ref{lem1_frequency}, the set
\begin{align*}
     \mathcal{F}=\{T\in\mathcal{B}+\mathbb{R}^d\mid \xi(T)(S)\neq 0\}
\end{align*}
is a finite set and then the sum $M_S=\sum_{T\in \mathcal{B}+\mathbb{R}^d} \xi(T)(S)=\sum_{T\in \mathcal{F}}\xi(T)(S)$ is a real positive number. For arbitrary $\epsilon>0$, set $\delta=\epsilon/M_S$ such that whenever $|v(T)-v'(T)|<\delta$ for any $T\in \mathcal{F}$, we have
\begin{align*}
    \left|\xi(v)(S)-\xi(v')(S)\right|&=\left|\sum_{T\in\mathcal{B}+\mathbb{R}^d}v(T)\xi(T)(S)-\sum_{T\in\mathcal{B}+\mathbb{R}^d}v'(T)\xi(T)(S)\right|\\
    &=\left|\sum_{T\in\mathcal{F}}(v(T)\xi(T)(S)-v'(T)\xi(T)(S)\right|\\
    &\leqq\sum_{T\in\mathcal{F}}|v(T)-v'(T)|\xi(T)(S)<\epsilon
\end{align*}
%If $v,v'\in\WP_{\mathcal{B}}$ are close enough in the sense that
%\begin{align*}
 %v(T)\approx v'(T)
%\end{align*}
%for any $T\in\mathcal{F}$, then
%\begin{align*}
 %\sum_{T\in\mathcal{B}+\mathbb{R}^d}v(T)\xi(T)(S)\approx %\sum_{T\in\mathcal{B}+\mathbb{R}^d}v'(T)\xi(T)(S)
%\end{align*}
%and so $\xi(v)(S)\approx \xi(v')(S)$.
\end{proof}

\begin{cor}
 \begin{align*}
  \xi^n(v)=\sum_{T\in\mathcal{B}+\mathbb{R}^d}v(T)\xi^n(T).
 \end{align*}
\end{cor}

For the usual theory for tilings, the cutting-off operation is important.
We define the cutting-off operation for weighted patches.
For a $v\in\WP_{\mathcal{B}}$ and a $K\subset\mathbb{R}^d$ we define the cut-off $v\sci K$ of
$v$ by $K$ via
\begin{align*}
 v\sci K(T)=
\begin{cases}
 v(T) &\text{ if $\supp T\subset K$}\\
0 &\text{ otherwise.}
\end{cases}
\end{align*}

\begin{lem}
 For any $K\subset\mathbb{R}^d$, the map
\begin{align*}
 \WP_{\mathcal{B}}\ni v\mapsto v\sci K\in\WP_{\mathcal{B}}
\end{align*}
is linear and continuous with respect to the weak topology.
\end{lem}

\begin{proof}
 Take a $T\in\mathcal{B}+\mathbb{R}^d$ and fix it.
If $\supp T\not\subset K$ then $v\sci K(T)=0$ for any $v$. Otherwise, if $v(T)\approx v'(T)$,
then $v\sci K(T)=v(T)\approx v'(T)=v'\sci K(T)$.
\end{proof}

Finally,  we state the main result in this section, which claims that the right Perron--Frobenius eigenvector gives
the frequency for the fixed point for $\xi$.

\begin{prop}\label{prop_frequency_weighted-substi}
 Let $v$ be such that
\begin{enumerate}
\item $\SP(v)$ is a tiling in $\mathbb{R}^d$,
 \item $v(T)=1$ for any $T\in\SP(v)$,
\item $\xi^k(v)=v$ for some $k>0$, and
\item the substitution matrix for $v$ is primitive.
\end{enumerate}
Take a Perron--Frobenius left eigenvector $l=(l_1,l_2,\ldots ,l_{\kappa})$ and
right eigenvector
\begin{align*}
 r=
\begin{pmatrix}
 r_1\\r_2\\\vdots\\r_{\kappa}
\end{pmatrix}
\end{align*}
for the substitution matrix for $\xi$, 
and assume $l_i$ coincides with the Lebesgue measure $\vol(\supp T_i)$
for $\supp T_i$ for each $i$, and
$l$ and $r$ are normalized in the sense that $\sum_{i=1}^{\kappa}l_ir_i=\sum_{i=1}^{\kappa}r_i=1$.

Then $r$ describes the frequency for $v$, which is regarded as a tiling: for any van Hove sequence 
$(A_n)_{n=1,2,\ldots}$ and $x_1,x_2,\ldots\in\mathbb{R}^d$, we have a convergence
\begin{align*}
 \lim_{n\rightarrow \infty}\frac{1}{\vol(A_n)}\tau_{T_i}(v\sci (A_n+x_n))=r_i,
\end{align*}
which is uniform for $(x_n)$.
\end{prop}
The proof will be given in the appendix.

% A map $v\colon\mathcal{A}+\mathbb{R}^d\rightarrow\mathbb{C}$ is called a weighted pattern.

\subsection{The theory of overlapping substitutions}
A pre-overlapping substitution is a triple $(\mathcal{B},\phi,\rho)$ where
\begin{itemize}
\item $\mathcal{B}=\{T_1,T_2,\ldots ,T_{\kappa}\}$ is a finite set of tiles in $\mathbb{R}^d$,
\item $\phi\colon\mathbb{R}^d\rightarrow\mathbb{R}^d$ is a linear map of which eigenvalues are all greater than 1 in modulus, and
\item $\rho$ is a map defined on $\mathcal{B}$ such that $\rho(T_i)$ is a finite patch consisting of
translates of elements of $\mathcal{B}$.
\end{itemize}

The set $\mathcal{B}$ is called the alphabet for the pre-overlapping substitution and
each $T_i$ in $\mathcal{B}$ is called a proto-tile. For a proto-tile $T_i$ and $x\in\mathbb{R}^d$, we set
\begin{align*}
     \rho(T_i+x)=\rho(T_i)+\phi(x).
\end{align*}
For a pattern $\mathcal{P}$ consisting of translates of proto-tiles, we set
\begin{align*}
    \rho(\mathcal{P})=\bigcup_{T\in\mathcal{P}}\rho(T).
\end{align*}
Note that each $\rho(T)$ is already defined and $\rho(\mathcal{P})$ is a new pattern consisting of
translates of proto-tiles.
We use the same symbol $\rho$ for an pre-overlapping substitution and a map that sends a pattern to
another pattern. For a pattern $\mathcal{P}$ consisting of translates of proto-tiles, we can apply $\rho$ multiple times and
obtain $\rho^n(\mathcal{P})$ for $n=1,2,\ldots$.
If, for any $n>0$ and
proto-tile $T$, the
pattern $\rho^n(T)$ is a
patch (there are no illegal overlaps), then we say $\rho$ is consistent and call $\rho$ an overlapping substitution.

To an overlapping substitution $\rho$, we can often take a patch $\mathcal{P}$ consisting of translates of proto-tiles such that for each $n=1,2,\ldots$, the set $\rho^n(\mathcal{P})$ is a patch, that is, there are no
illegal overlaps.
Moreover, we often have that, for some $k>0$,
\begin{align*}
       \lim_{n\rightarrow\infty}\rho^{kn}(\mathcal{P})
\end{align*}
converges to a tiling
and so we have a fixed point for an overlapping substitution $\rho^k$.

For an expanding overlapping substitution,
the above process is
almost always possible,
given a primitivity condition is satisfied.
Here, we say that an overlapping substitution $(\mathcal{B},\phi,\rho)$ is expanding if
for each $P\in\mathcal{B}$, we have
\begin{align*}
\supp\rho(P)\supset\phi(\supp P).
\end{align*}
For an expanding overlapping substitution $\rho$, if we can find a $P\in\mathcal{B}$ 
and an $x\in\mathbb{R}^d$ such that
\begin{align*}
 P+x\in\rho^k(P+x),
\end{align*}
and $P+x$ is in the ``interior'' of the right-hand side,
then by an standard argument, we can construct a tiling
\begin{align*}
 \mathcal{T}=\lim_{n\rightarrow\infty}\rho^{kn}(P+x),
\end{align*}
which is a fixed point for $\rho$: we have $\rho^k(\mathcal{T})=\mathcal{T}$.

To an overlapping substitution $(\mathcal{B},\phi,\rho)$, we can associate a weighted substitution
$(\mathcal{B},\phi,\tilde{\rho})$, as follows.
The alphabet and the expansion map are the same. For $T_i\in\mathcal{A}$, the weighted patch $\tilde{\rho}(T_i)$
is a map whose support is $\rho(T_i)$ and the weights are defined via
\begin{align*}
       \tilde{\rho}(T_i)(S)=\frac{\vol(\phi(\supp T_i)\cap\supp S)}{\vol (\supp S)},
\end{align*}
where $\vol$ denotes the Lebesgue measure and $S\in\rho(T_i)$.
{\color{red}Note that this idea of substitution matrix coincides the one of (\ref{DoteraMatrix})}, although we only consider translations and not all isometries in $\mathbb{R}^d$. For an overlapping substitution $\rho$, its
substitution matrix is defined as the one for the associated substitution $\tilde{\rho}$.
We will prove that the tile frequencies for the tilings generated by an overlapping substitution are given by the
Perron--Frobenius eigenvector for the substitution matrix.

% \subsection{Overlapping substitutions and their associated weighted substitution}

We prove two results on the meaning
of left and right Perron--Frobenius
eigenvectors for the substitution matrix associated with an overlapping substitution, given that the matrix is primitive. The result for the left Perron--Frobenius eigenvector is easier.

\begin{prop}\label{prop_left_eigenvector}
    Let $\rho$ be an expanding overlapping substitution with a primitive substitution matrix. Then the left eigenvector for the Perron--Frobenius eigenvalue $\beta$ is a multiple of the vector consisting of Lebesgue measures of proto-tiles
    \begin{align*}
        (\vol(T_1),\vol(T_2),\ldots ,\vol(T_{\kappa})),
    \end{align*}
    and, moreover, we have
    \begin{align*}
        \beta=\det\phi.
    \end{align*}
\end{prop}
\begin{proof}
    The $(i,j)$-element of the substitution matrix is
    \begin{align*}
        \sum_{\substack{x\in\mathbb{R}^d,\\T_i+x\in\rho(T_j)}}\frac{\vol(\supp(T_i+x)\cap\phi(\supp T_j))}{\vol(\supp T_i)}.
    \end{align*}
    The claim follows from the Perron--Frobenius theory.
\end{proof}

Given an overlapping  substitution $\rho$ and its fixed point $\mathcal{T}$ (that is, for some
$k>0$, we have $\rho^k(\mathcal{T})=\mathcal{T}$), we define an weighted pattern $v$ via
\begin{align*}
 v(T)=
\begin{cases}
 1 &\text{ if $T\in\mathcal{T}$}\\
0 & \text{ otherwise}.
\end{cases}
\end{align*}

Then we can prove that $\tilde{\rho^k}(v)=v$. By Proposition \ref{prop_frequency_weighted-substi} and Proposition \ref{prop_left_eigenvector}, we have proved the following:

\begin{theorem}\label{thm_uniform_tile_freq}
 Let $\rho$ be an expanding overlapping subsitutiion and $r=(r_1,r_2,\ldots, r_{\kappa})$ be
a right Perron--Frobenius eigenvector which is normalized in the sense of 
Proposition \ref{prop_frequency_weighted-substi}. Then $r$ gives the frequency for any fixed points for $\rho$.
In particular, for any van Hove sequence $(A_n)$ and elements $x_n\in\mathbb{R}^d$, we have
a convergence
\begin{align*}
      \lim_{n\rightarrow\infty}\frac{\card\{x\in A_n+x_n\mid T_i+x\in\mathcal{T}\}}{\vol{A_n}}=r_i,
\end{align*}
which is uniform for $(x_n)$.
\end{theorem}

\textcolor{blue}{
We will use this uniform convergence later to prove the unique ergodicity for the corresponding dynamical system in the next theorem.
In particular, we see that, for repetitive fixed points for the substitution, patch frequencies converge uniformly:}
\begin{theorem}\label{thm_uniform_patch_freq}
    Assume a fixed point $\mathcal{{T}}$ for $\rho$ is repetitive. Then the patch frequencies converge uniformly. That is, 
    for any van Hove sequence $(A_n)$,
    elements $x_n\in\mathbb{R}^d$
    and a patch $\mathcal{{P}}$, the
    limit
    \begin{align*}
    \lim_{n\rightarrow\infty}\frac{\card\{x\in A_n+x_n\mid \mathcal{P}+x\subset\mathcal{T}\}}{\vol A_n}
    \end{align*}
    converges and is independent of $(A_n)$ and $x_n$.
    Moreover, if we fix $(A_n)$ and $\mathcal{P}$, the limit is uniform for $(x_n)$.
    In particular, the corresponding dynamical system for the repetitive fixed point is uniquely ergodic.
\end{theorem}

\begin{proof}
    For a $P\in\mathcal{B}$
    and an $x\in\mathbb{R}^d$, we define the representative point for $P+x$ via
    \begin{align*}
        c(P+x)=x.
    \end{align*}
    We may assume that $0\in\mathbb{R}^d$ is in the interior of each prototile $P\in\mathcal{{B}}$.

    For each $L>0$, we
    define an additional label for tiles in
    $\mathcal{T}$ by 
    looking at the patterns around the tiles inside the ball of raduis $L$.
    For  a $T\in\mathcal{T}$, 
    we will look at
\begin{align*}\mathcal{T}\sqcap B(c(T),L),
    \end{align*}
    but we want to identify two such patches that are translationally equivalent.
    The
    corresponding ``coloured'' tile $\tilde{T}$ is the tile $T$ with additional label or colour
    \begin{align*}
        (\mathcal{T}-c(T))\sqcap B_L.
    \end{align*}

    The new alphabet for the coloured tiling is
    \begin{align*}
        \tilde{\mathcal{B}}=\{(T-c(T),(\mathcal{T}-c(T))\sqcap B_L)\mid T\in\mathcal{T}\}
        =\{\tilde{T}-c(T)\mid T\in\mathcal{T}\}.
    \end{align*} 
    We define a substitution $\tilde{\rho}$ via
    \begin{align*}
        \tilde{\rho}(\tilde{T}-c(T))=
        \{\tilde{S}-\phi(c(T))\mid S\in\rho(T)\}.
    \end{align*}
    This is well-defined.
    This means that the image for $\tilde{T}$ is
    \begin{align*}
        \tilde{\rho}(\tilde{T})=\{\tilde{S},\mid S\in\rho(T)\}=\{(S,(\mathcal{T}-c(S))\sqcap B_L)\mid S\in\rho(T)\}.
    \end{align*}
    We set the colourd tiling via
    \begin{align*}
        \tilde{\mathcal{T}}=\{\tilde{T}\mid T\in\mathcal{T}\}.
    \end{align*}
    It is easy to show that $\tilde{\rho}(\tilde{\mathcal{T}})=\tilde{\mathcal{T}}$.

    For a patch $\mathcal{P}$, we take an arbitrary $S\in\mathcal{P}$ and fix it. We set
    \begin{align*}
        \widetilde{\mathcal{B}}_{\mathcal{P}}
        =\{(P,\mathcal{S})\in\tilde{\mathcal{B}}\mid \mathcal{P}-c(S)\subset \mathcal{S}\}.
    \end{align*}
        We have a bijection
        \begin{align*}
            &x\in\{x\in\mathbb{R}^d\mid \mathcal{P}+x\subset\mathcal{T}\}\\
            &\mapsto ((P,(\mathcal{T}-x-c(S))\sqcap B_L), x+c(S))\in\{((P,\mathcal{S}),y)\in\widetilde{\mathcal{B}}_{\mathcal{P}}\times \mathbb{R}^d\mid (P,\mathcal{S})+y\in\tilde{\mathcal{T}}\}.
        \end{align*}
We conclude
\begin{align*}
    \card\{x\in A_n\mid \mathcal{P}+x\subset\mathcal{T}\}=\sum_{\tilde{P}\in\widetilde{\mathcal{B}}_{\mathcal{P}}}\card\{x\in A_n\mid \tilde{P}+x\in\tilde{\mathcal{T}}\}+o(\vol (A_n)).
\end{align*}
The repetitivity of $\mathcal{T}$ implies the primitivity of $\tilde{\rho}$, which implies the uniform convergence for the tile frequencies for $\tilde{\mathcal{T}}$ by
Theorem \ref{thm_uniform_tile_freq}.
The theorem follows.
\end{proof}

\textcolor{blue}{
It should be noted here that, as in the cases of usual substitution, when we try to construct geometric substitution from symbolic one, we have to find an appropriate expansion map $\phi$. 
For one-dimensional cases, as we see in Section \ref{Sec5}, from the information on how tiles protrude in the image of substitution and Lagarias-Wang restriction on the dominant eigenvalue of the substitution matrix, we indirectly 
obtain the expansion factor (the map $\phi$ for this case) in such examples,
see \cite{Lagarias-Wang:03}.}

\section{The expanding constant is an algebraic integer}\label{Sec3}

Let $\T$ be a tiling in $\R^d$ which is a fixed point for an overlapping substitution $\rho$ with an expansion constant $\beta$
, that is,
$\rho(\mathcal{T})=\mathcal{T}$.

Since we obtain a weak Delone set from overlapping substitution, 
by Lagarias-Wang criterion (see \cite{Lagarias-Wang:03}), 
we see that $\beta^d$ is equal to the Perron--Frobenius root of the substitution matrix of $\T$. We see this fact in Proposition \ref{prop_left_eigenvector} as well. 

As we described in the introduction, the entries of the substitution matrix are not necessarily integers. Therefore the
Perron--Frobenius eigenvalue may not be an algebraic integer.
In this section, we show that under a mild assumption of FLC and repetitivity, $\beta$ is an
algebraic integer.

In general, a tiling $\mathcal{S}$ has finite local complexity (FLC, in short), if for any given positive $R$, there are only finitely many patches $P$, up to translation, which appear in a ball $B(x,R)$ of radius $R$.
A tiling $\mathcal{S}$ is repetitive if for any patch $P$ of $\mathcal{S}$, a positive $R$ exists so that any ball $B(x,R)$ contains a translate of $P$.

\begin{theorem}
If $\T$ satisfies FLC and is repetitive, the expansion constant $\beta$ must be an algebraic integer.
\end{theorem}

\begin{proof}
Let $T_i\ (i=1,\dots,\kappa)$ be the proto-tiles of $\T$. We define $\Lambda_i=\{ x-y\in \R^d\ |\ x+T_i \in \T \text{ and } y+T_i \in \T \}$
and $\langle\Lambda_i\rangle$ is the $\Z$-module generated by $\Lambda_i$. 
By FLC and repetitivity, we claim that 
$\langle\Lambda_i\rangle$ is finitely generated by the 
stepping-stone discussion in \cite{Lagarias:99}. 
Take $P=T_i$ as a patch. By FLC, 
$$
W=\{ x-y\in \R^d \ |\ x+T_i \in \T, y+T_i\in \T, \vert x-y \vert \le 3R \}
$$
is a finite set. Take
an element  $x-y\in \langle \Lambda_i \rangle$ such that $x+T_i \in \T, y+T_i \in \T$.
Choose an integer $K>0$ so that
$$
x_j=(1-j/K)x + j K y
$$
for $j=0,1,\dots,K$ and $\vert x_{j+1}-x_j \vert \le R$. Then every ball
$B(x_j,R)$ contains some $u_j+T_i\in \T$. Since $u_{j+1}-u_j\in W$, we see that
$\langle \Lambda_i \rangle$ is generated by $W$. The claim is proved. 

Therefore we have
$$
\left\langle\bigcup_{i=1}^{\kappa}\Lambda_i\right\rangle = v_1\Z +\dots + v_k \Z
$$
where $k\ge d$. Applying the overlap substitution, we have
$$
\beta v_j = \sum_{i=1}^k a_{ij} v_i
$$
with $a_{ij}\in \Z$. Consider a matrix $V=[v_1,v_2,\ldots, v_k]$ and
$M=(a_{ij})$.
Then $\beta V=VM$ and so
$\beta V^{T}x=M^{T}V^{T}x$, where
${}^{T}$ denotes the transpose
and $x\in\mathbb{R}^d$.
Thus $\beta$ is an eigenvalue of the $k\times k$ matrix $M$ with integer entries.
\end{proof}

\begin{figure}[h]
    \centering
\includegraphics[width=0.6\linewidth]{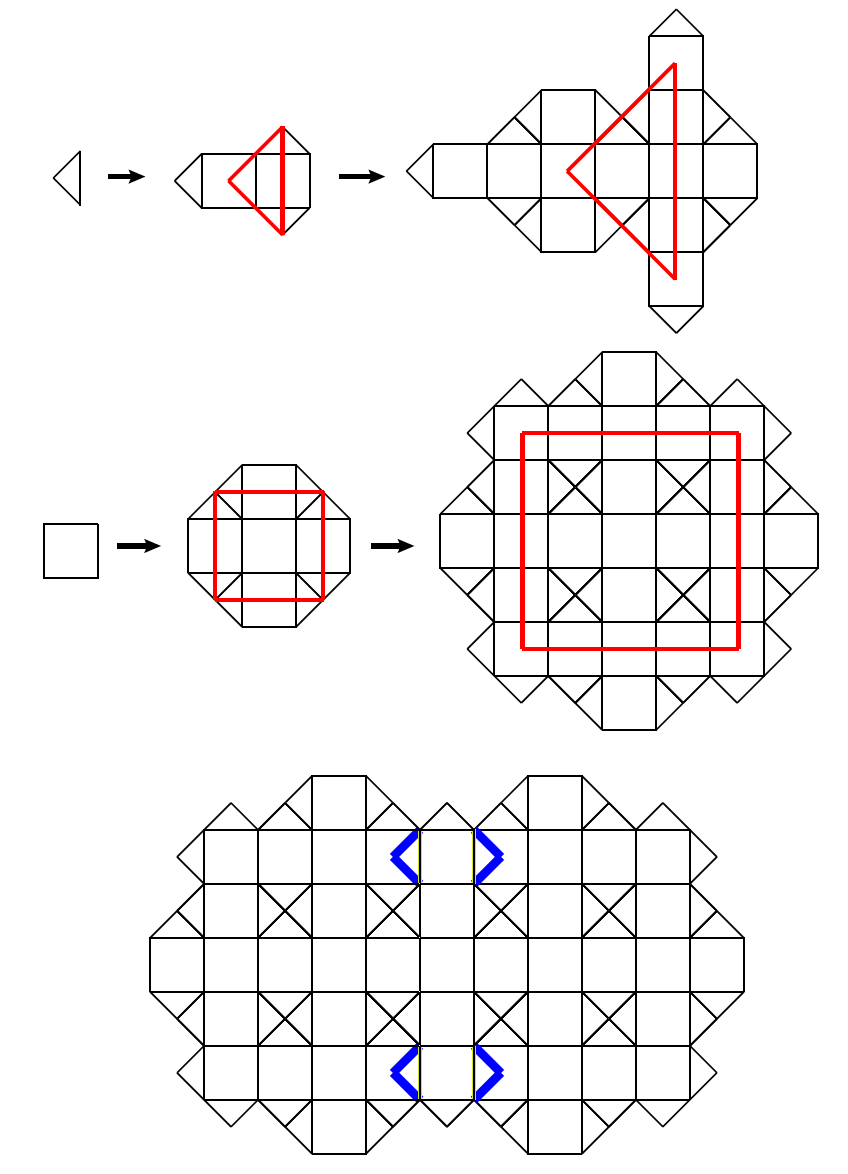}
    \caption{This is a substitution produced by triangles and squares. We see that the first and second iterations are well-defined but for the third time, illegal overlaps appear when we try to fit adjacent squares: blue segments in the last picture.}
    \label{Non-consis}
\end{figure}

%%%%%%%%%%%%%%%%%%%
\section{Consistency of pre-overlapping substitution: open set condition}\label{Sec4}

% We first claim that the description in this section is based on several $2$ dimensional overlap substitutions. We are anticipating a more general framework to settle this type of problem.

The example of Dotera-Bekku-Ziher \cite{DBZ} presents another problem. Can we iterate this substitution rule infinitely many times, without yielding illegal (partial) overlaps-two different tiles that overlap in the interiors?

\textcolor{blue}{
For example, see Figure \ref{Non-consis} for an example where the substitution produces illegal overlap only after several steps of iterations.
Even if every inflated edge is shared by several tiles exactly in the same way whenever we iterate,
as in Figure \ref{BronzeSubst},
and the illegal overlaps do not occur locally, there might be a global overlap: the boundary of $n$-th iterate might intersect with itself and may produce an illegal overlap. For example, in the iterations of substitution in Figure \ref{fig:inflation_complicated-boundary}, the boundary evolves in a complicated way.  How can one prove that the substitution is consistent? This is the problem we address in this section.}

If the boundary never meets itself, then the above inconsistency does not happen for good and we gain patches whose supports contain balls of arbitrary radius. Consequently, there exists a tiling by a usual argument
(c.f. \cite[Theorem 3.8.1]{Gruenbaum-Shephard:87}). 
We show that this safe 
situation is achieved by the technique in fractal geometry. 

A map $f:\R^2\to \R^2$ is {\it contractive} 
if there exists $0<r<1$ that
$\Vert f(x)-f(y) \Vert \le r \Vert x-y\Vert$. 
Given a strongly connected graph $(V,E)$ with vertex set $V=\{1,\dots,\kappa\}$ and
edge set $E\subset V\times V$. 
For each edge $e\in E$, 
we associate a contractive map $f_e$ whose contraction ratio may depend 
on $e$. 
Let us denote by $E_{i,j}$ the set of edges from vertex $i$ to vertex $j$ in $(V, E)$. 
Then there exist unique non-empty compact sets $K_i$'s which satisfy
\begin{equation}
\label{GIFS}
\bigcup_{j=1}^{\kappa} \bigcup_{e\in E_{i,j}} f_e(K_j)=K_i
\end{equation}
which are called the attractors of the graph-directed
iterated function system $\{f_e\ |\ e\in E\}$ (GIFS for short).  
We write $i(e)$ (resp. $t(e)$) the initial vertex (resp. terminal vertex).
The GIFS is satisfied the {\it open set condition} (OSC for short) if there exist bounded non-empty open sets $U_j\ (j\in V)$ such that
$$
\bigcup_{j=1}^{\kappa} \bigcup_{e\in E_{i,j}} f_e(U_j)\subset U_i
$$
and the left union is disjoint.
See \cite{Falconer:97}.
For a walk $\i=(e_1,\dots, e_n)$ in $(V,E)$, 
we write 
$K_{\i}=f_{e_1}f_{e_2}\dots f_{e_n}(K_{t(e_n)})$. The open set condition
guarantees that $K_{\i}$ and $K_{\j}$ 
do not share an inner point if and only if $\i$ and 
$\j$ are not comparable, i.e., $\i$ is not a prefix of $\j$, and $\j$ is
not a prefix of $\i$.
Given a GIFS, the open sets $U_j$ for OSC are called {\it feasible open sets}. Note that feasible open sets are 
not unique for a given GIFS. 

%{\color{red}We equip a GIFS with an order structure, which induces a lexicographical order of the associated symbolic space. Denote $\Gamma_i$ by the set of edges starting from vertex $i$.  We say $(V,E,\prec)$ is an \emph{ordered GIFS} if $\prec$ is a partial order on $E$ such that $(1)$ $\prec$ is a linear order when restricted on $\Gamma_i$ for every $i\in V$; (2) elements in $\Gamma_i$ and $\Gamma_j$ are not comparable if $i\neq j$.
%Denote $\Gamma_i^k$ by the set of paths starting from vertex $i$ with length $k$. Then it is easy to see that $(\Gamma_i^k, \prec)$ is a linear order; two paths $\i,\j$ are said to be adjacent if there is no path between them with respect to the order $\prec$.
%An ordered GIFS with invariant sets $\{K_i\}_{i=1}^m$ is called a \emph{linear GIFS} if 
%for all $i\in V$ and $k\geq 1$, $K_{\i}\cap K_{\j}\neq \emptyset$ provided $\i,\j$ are adjacent paths in $\Gamma_i^k$.}

For a GIFS, for each $i\in V$, 
we shall define a graph with the vertex set 
$G_i=\bigcup_{j=1}^\kappa E_{i,j}$.
%$$\{e|\ e \text{ is an edge starting from } i\}$$ 
%and 
There exists an edge 
from vertex $e\in G_{i}$ to $e'\in G_{i}$ if and only if
$f_e(K_{t(e)})\cap f_{e'}(K_{t(e')})\neq \emptyset$ in (\ref{GIFS}).
All $K_{i}$ are connected if $G_i$ are connected graphs (c.f. 
\cite{Hata:85,Luo-Akiyama-Thuswaldner:04}).
If every graph $G_i$ is a line graph and if
for each $(e,e')\in G_i$
$\overline{f_e(U_{t(e)})}\cap \overline{f_{e'}(U_{t(e')})}=\{a_{e,e'}\}$ 
with a single point $a_{e,e'}\in \R^2$, 
we say the feasible open sets satisfy {\it linear GIFS condition}.
Linear GIFS condition implies that $K_i$ is homeomorphic to a proper segment, and the homeomorphism between $[0,1]$ and $K_i$ is naturally constructed from this graph.
Similar construction appeared in many articles, e.g. 
\cite{Bandt-Mekhontsev:18,
Akiyama-Loridant:11}.

To apply this, we consider the limit tile
$$
\mathcal{T}_i=\lim_{n\to \infty} \beta^{-n} \rho^n(T_i)
$$
here the limitation is given by the Hausdorff metric, where $\rho$ is an expanding overlapping substitution with a primitive substitution matrix,$\beta$ is the left eigenvector for the Perron--Frobenius eigenvalue and $\{T_1,T_2\dots,T_{\kappa}\}$ are proto-tiles. 
The non-overlapping situation is guaranteed when 
the boundary $\partial(\mathcal{T}_i)$ of $\mathcal{T}_i$
satisfies a feasible open set with linear GIFS condition. 
We shall find a GIFS with attractors $J_j\ (j=1,\dots,\ell)$ 
for the boundaries of the limit tiles $\mathcal{T}_i$ such that $\partial(\mathcal{T}_i)$ is the union of some $J_i$; moreover, there exist feasible open sets having a linear GIFS condition. 

Let $\sigma:\{1,\dots,\ell\}\to\{1,\dots,\ell\}$ be the substitution rule given by the GIFS of the boundary then we have $J_1,\dots,J_{\ell}$ satisfying
\begin{align}\label{eq:boundary}
J_i=\bigcup_{e\in \sigma(i)} g_e(J_e).
\end{align}
For each $J_i$, we start with a broken line connecting the successive junction points $\{a_{e,e'}\}$. By iteration of the broken lines with the boundary substitution, it converges to
these pieces $J_1,\dots J_{\ell}$ by the Hausdorff metric. Since $\partial(\mathcal{T}_i)$ is a union of $J_i$, we could regard that all $\partial(\mathcal{T}_i)$ are the limit of sequences of broken lines. Clearly, by this construction, the self-intersection of the boundary pieces can not occur at any level for any pieces. We can also check from the feasible open sets that appeared in the boundary that the boundary
pieces $J_i\ (i=1,\dots,\ell)$ 
form closed curves by showing
\begin{itemize}
    \item two adjacent pieces of right side of \eqref{eq:boundary} is a singleton;
    \item $J_{\omega_1}\cap J_{\omega_2}=\emptyset$ if $\omega_1,\omega_2\in \{1,2,\dots,\ell\}^n$ are not adjacent. 
Here the adjacent means order given by lexicographical order.
\end{itemize}
After this confirmation, we see that no self-intersection occurs even if we iterate infinitely many times the substitution. 
Consequently $\partial(\mathcal{T}_i)$ is a Jordan closed curve and we see
that $\mathcal{T}_i$ is a topological closed disk in view of Jordan's curve theorem.

Once we finish this task, the consistency will hold for all levels of
the patch and thus the tiling itself. 

Linear GIFS condition guarantees that boundary does not cause self-intersection and the consistency can be checked locally. This restriction is of course not necessary in general. There may be an overlapping substitution that the boundary
intersects but does not cause contradiction. (See Figure \ref{BronzeSupple})

\begin{ex}[Feasible open sets for boundaries Bronze-mean 'tiles']
\label{Dotera}
\begin{figure}[h]
    \centering
\includegraphics[width=8.5 cm]{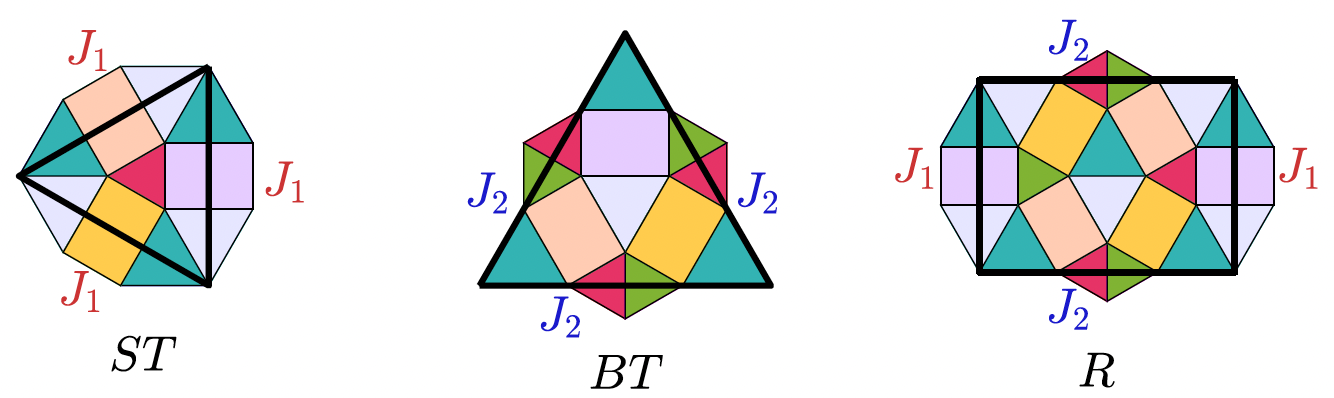}
    \caption{Two different types $J_1, J_2$ of the boundaries of Bronze-means tiles}
    \label{BronzeJ1J2}
\end{figure}

\begin{figure}[h]
\centering
\includegraphics[width=3.5 cm]{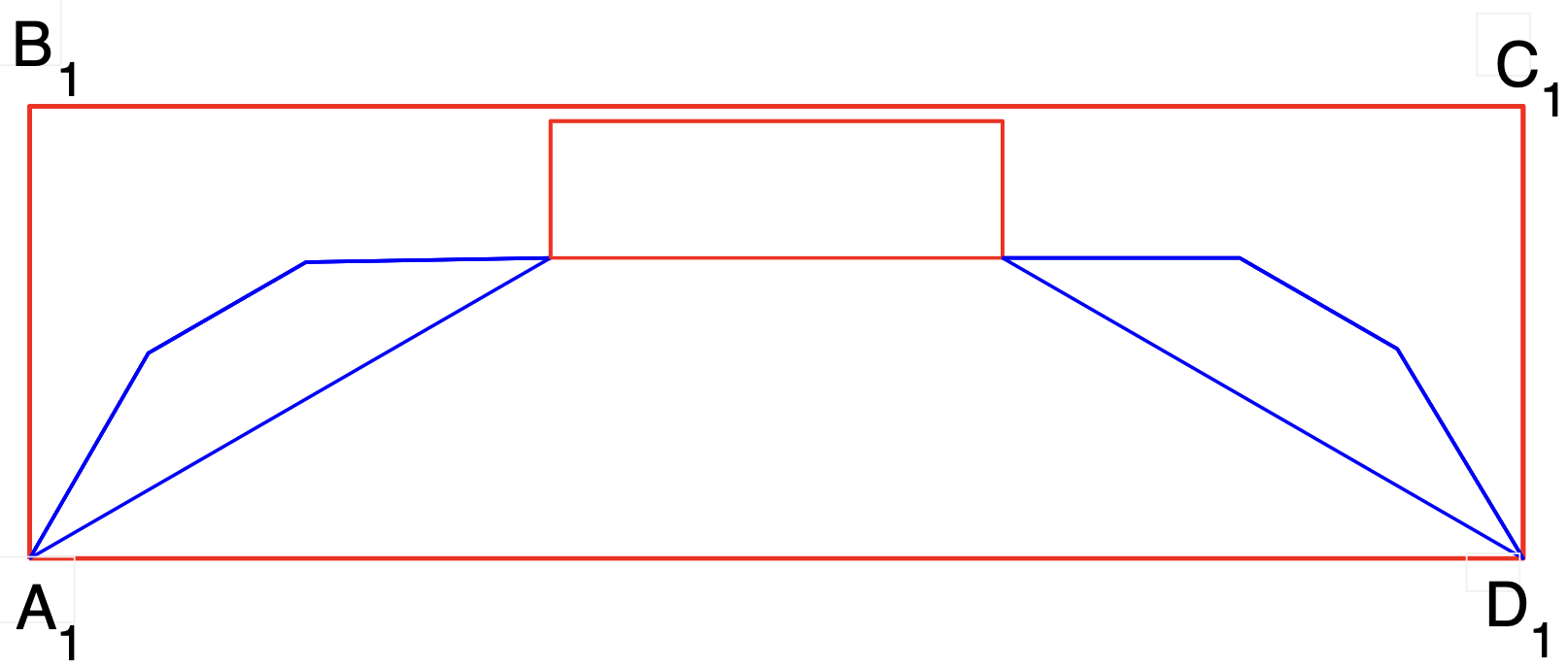}~\includegraphics[width= 4 cm]{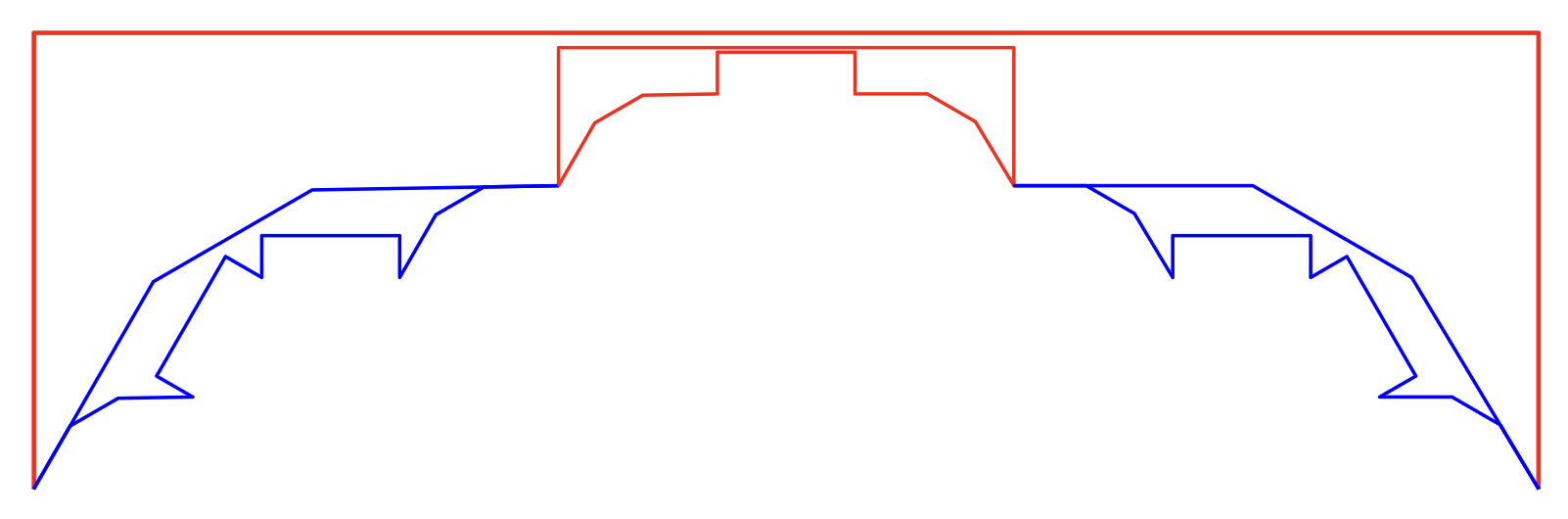}\quad\includegraphics[width= 4 cm]{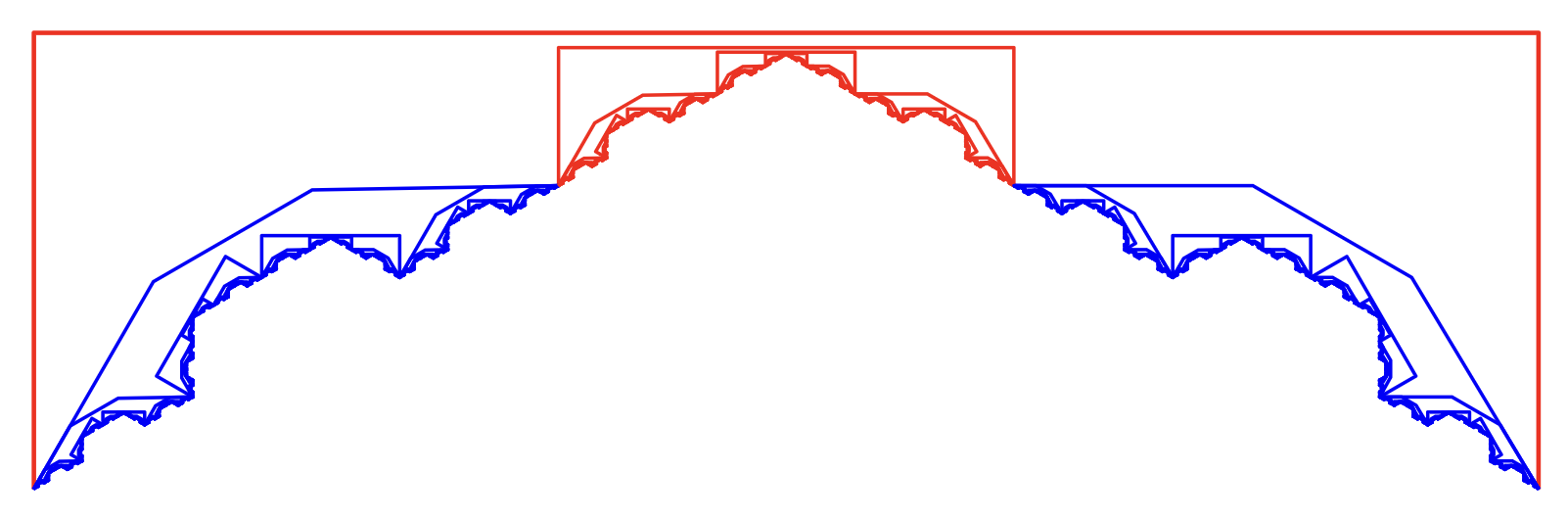}
\caption{Feasible open sets for $J_1$: the interior of rectangle $A_1B_1C_1D_1$. We show the coordinates which are $A_1=(0,0),B_1=(0,r),C_1=(1,r),D_1=(1,0)$, where $r=\frac{\sqrt{13}-3}{2}.$}\label{Bronze_tpyeA}
\end{figure}
\begin{figure}[h]
\includegraphics[width=3.5 cm]{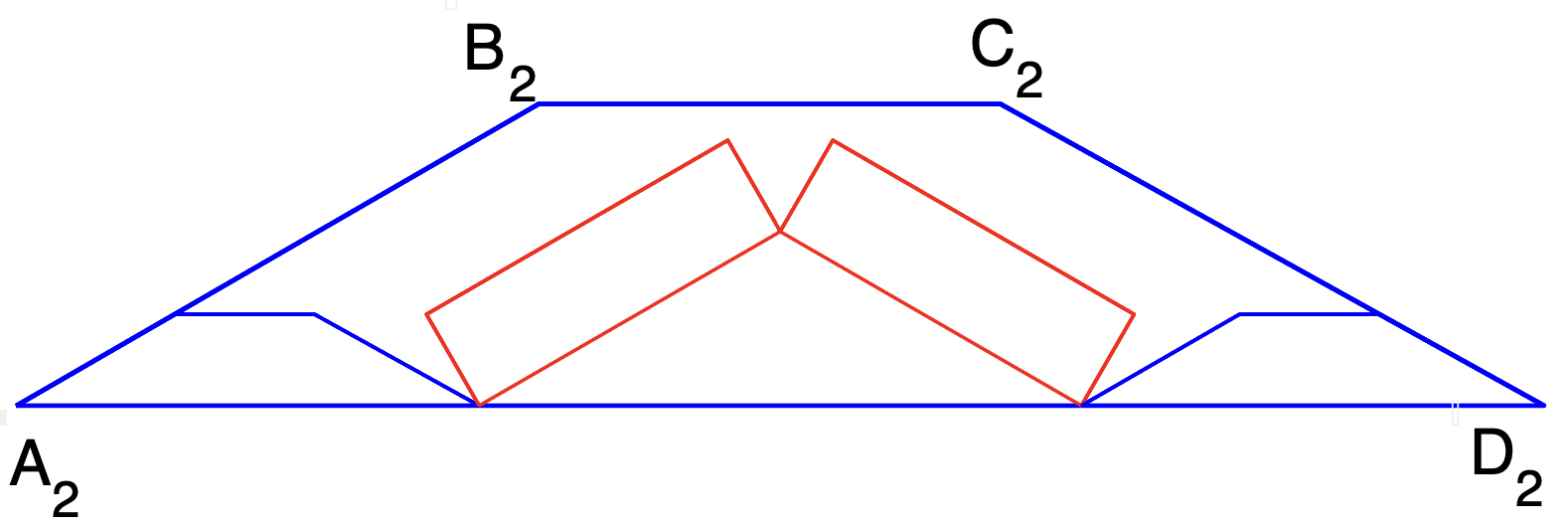}~\includegraphics[width= 4 cm]{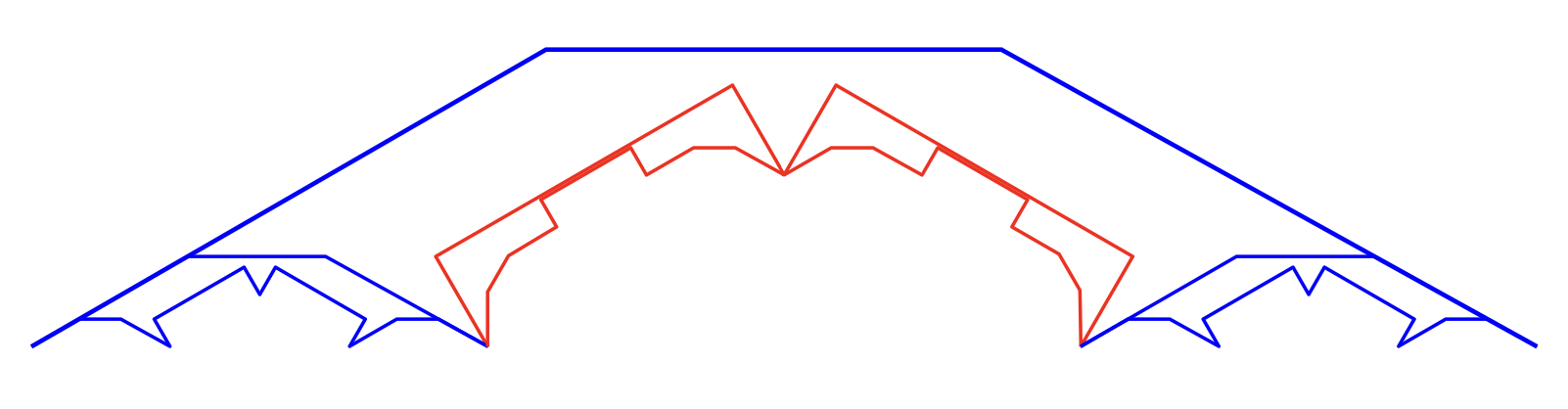}\quad\includegraphics[width= 4 cm]{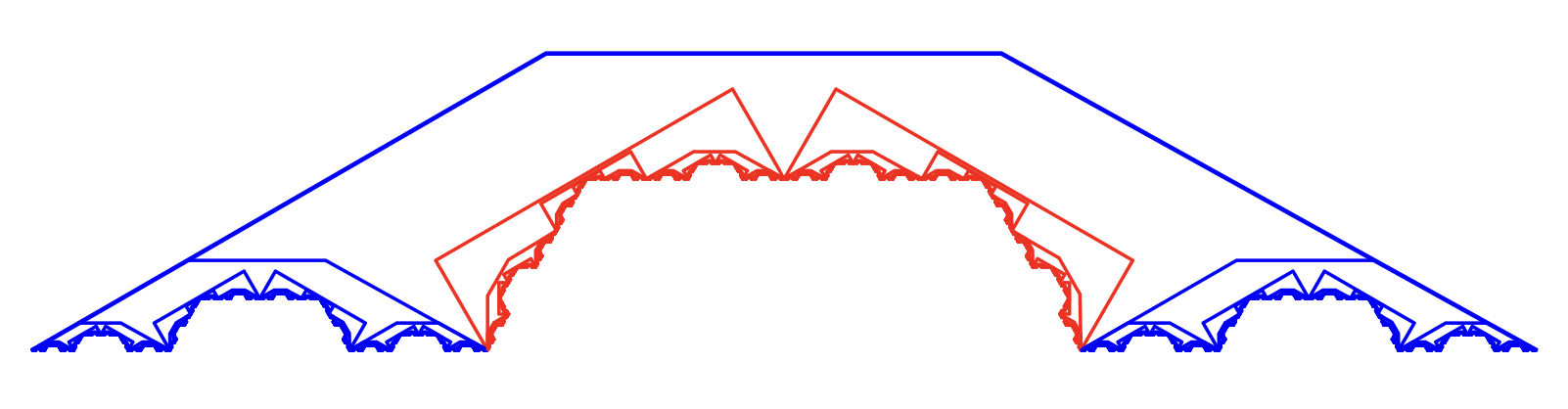}
\caption{Feasible open sets for $J_2$: the interior of the isosceles 
 trapezoid $A_2B_2C_2D_2$ with the coordinates $A_2=(0,0),B_2=(\frac{3}{2},\frac{\sqrt{3}}{2})r,C_2=(\frac{3}{2}+s,\frac{\sqrt{3}}{2})r,D_2=(s,0)$, here $s=\frac{r+2}{\sqrt{3}}$.}\label{Bronze_typeB}
\end{figure}

For the Bronze-mean tiles we introduced at the beginning (See Figure \ref{BronzeSubst}), we will show here that the boundaries of tiles satisfy the open set condition and the feasible sets satisfy linear IFS condition.

Denote the Bronze-means tiles by $T_1$ (produced by $ST$), $T_2$ (produced by $BT$) and $T_3$ (produced by $R$).
By the construction of $T_1,T_2$ and $T_3$ , the boundaries of them are distinguished by two different types, say $J_1$ and $J_2$ (see Figure \ref{BronzeJ1J2}). And $J_1$, $J_2$ satisfy the following set equations
\begin{align*}
J_1&=f_{12}(J_2)\cup f_{11}(J_1)\cup f_{12}'(J_2)\\
J_2&=f_{22}(J_2)\cup f_{21}(J_1)\cup f_{21}'(J_1)\cup f_{22}'(J_2)
\end{align*}
(Here $f_{ij}$ and $f_{ij}'$ could be written through Figure \ref{BronzeJ1J2}.) Let $U_1$ be the rectangle $A_1B_1C_1D_1$ without the boundary in Figure \ref{Bronze_tpyeA} and $U_2$ be the trapezoid $A_2B_2C_2D_2$  without the boundary in Figure \ref{Bronze_typeB}. Then we know $U_1$ and $U_2$ satisfy
\begin{align*}
f_{12}(U_2)\cup f_{11}(U_1)\cup f_{12}'(U_2)&\subset U_1\\
f_{22}(U_2)\cup f_{21}(U_1)\cup f_{21}'(U_1)\cup f_{22}'(U_2)&\subset U_2
\end{align*}
To check the feasible open sets satisfying the linear GIFS condition, we take $U_1$ as an example, it is easy to see from Figure \ref{Bronze_tpyeA} and \ref{Bronze_typeB} that 
$\overline{f_{12}(U_2)}\cap \overline{f_{11}(U_1)}$ and $\overline{f_{11}(U_1)}\cap \overline{f_{12}'(U_2)}$ are both singletons.

\end{ex}

\begin{ex}[Feasible open sets for boundaries square triangle 'tiles']
 The square triangle  is showing in Figure \ref{SquareTriangleSubst}  and  it is easy to know that they have fractal boundaries. We will present that the boundaries satisfy the open set condition and the feasible sets satisfy linear IFS condition. 
 
\begin{figure}[h]
    \centering
\includegraphics[width=9.5 cm]{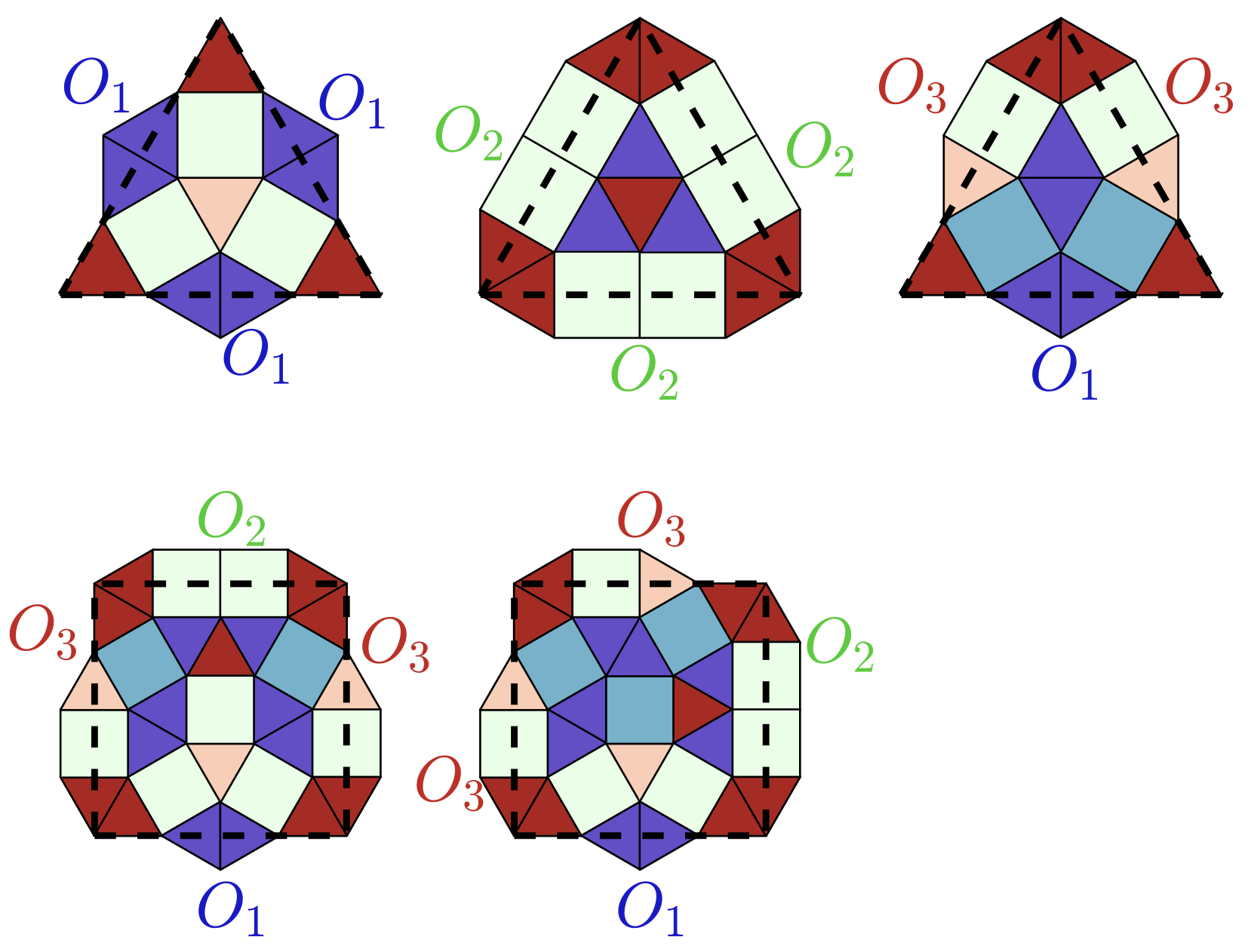}
    \caption{Three different types $O_1, O_2$ and $O_3$ of the boundaries of square triangle tiles}
\end{figure}

\begin{figure}[h]
\centering
\includegraphics[width=4 cm]{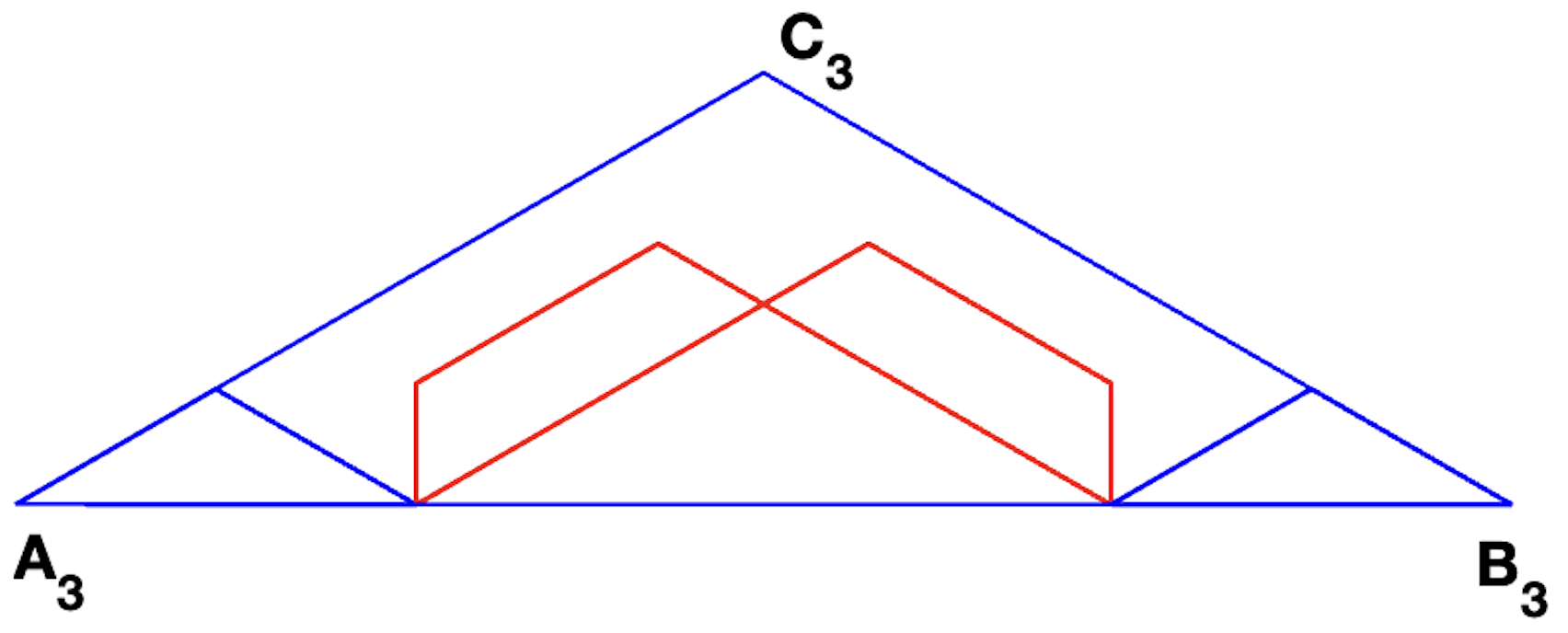}~\includegraphics[width= 4 cm]{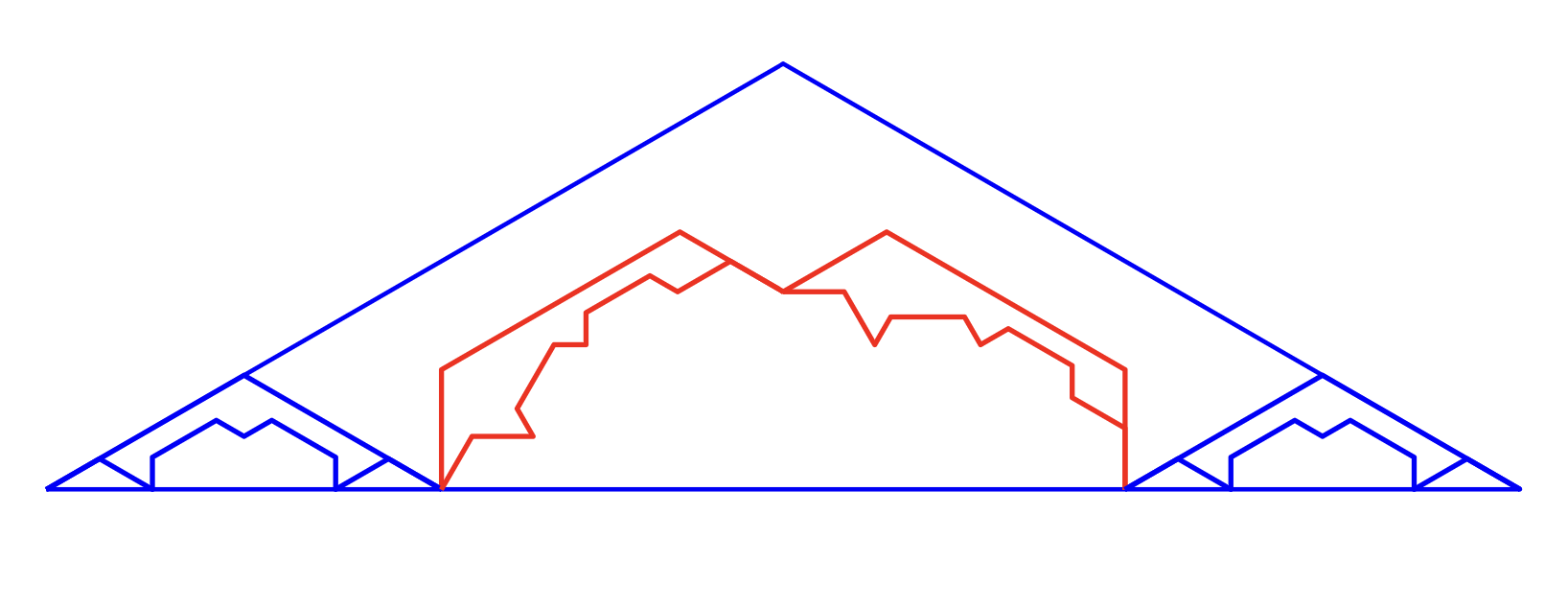}\quad\includegraphics[width= 4 cm]{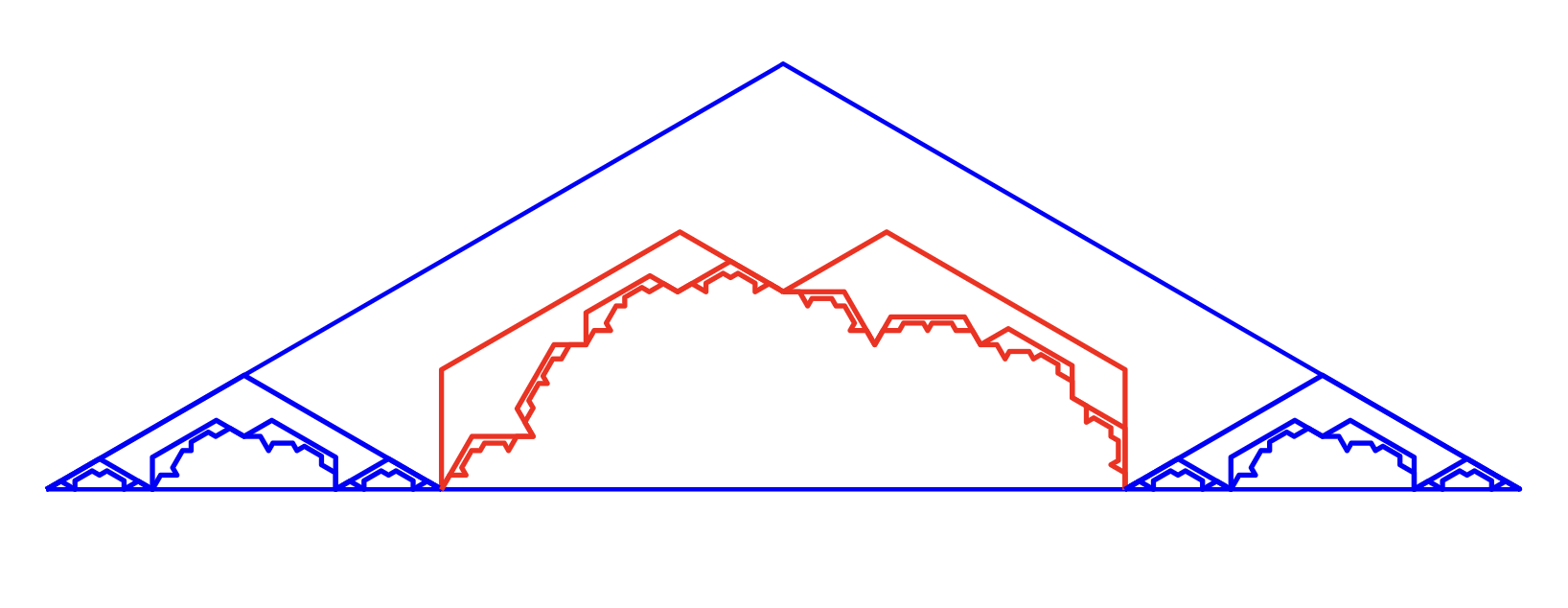}
\caption{Feasible open set for $K_1$: the interior of the triangle which is an isosceles triangle with base angle is $\frac{\pi}{6}$ and length of base edge $1$.}\label{SquareTriangle_tpyeA}
\end{figure}
\begin{figure}[h]
\includegraphics[width=4 cm]{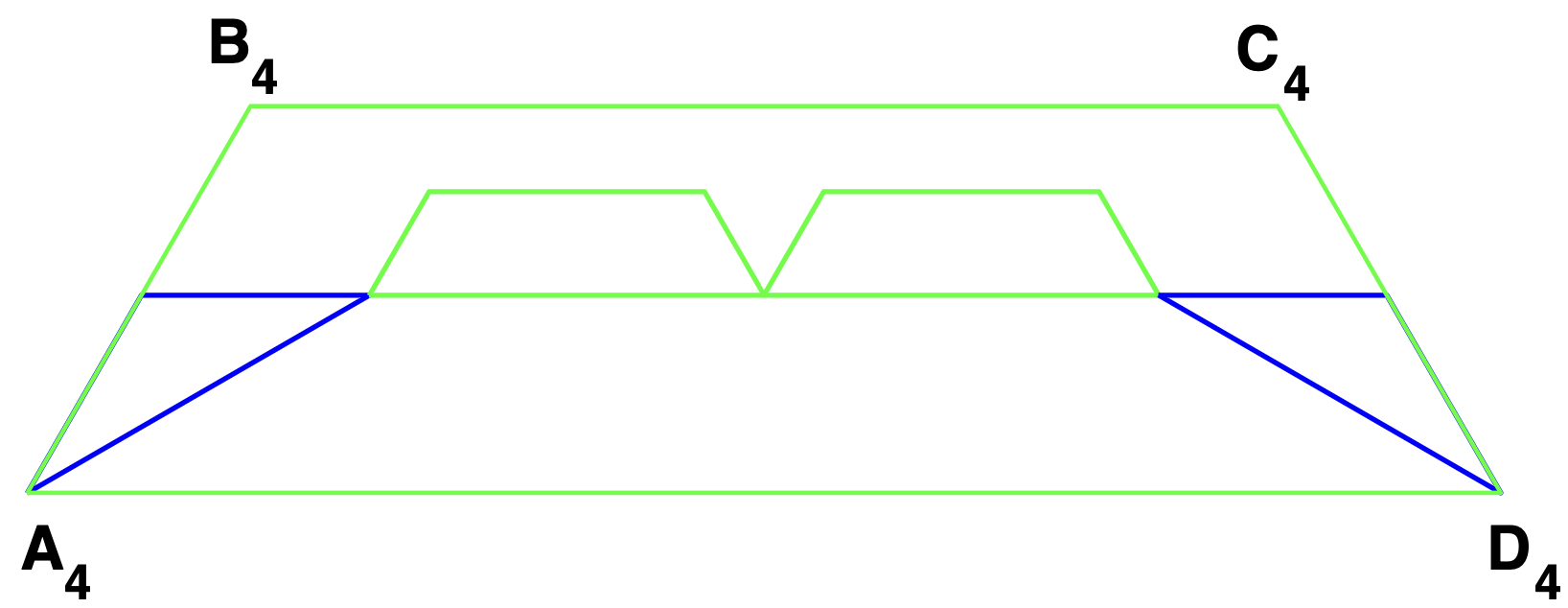}~\includegraphics[width= 4 cm]{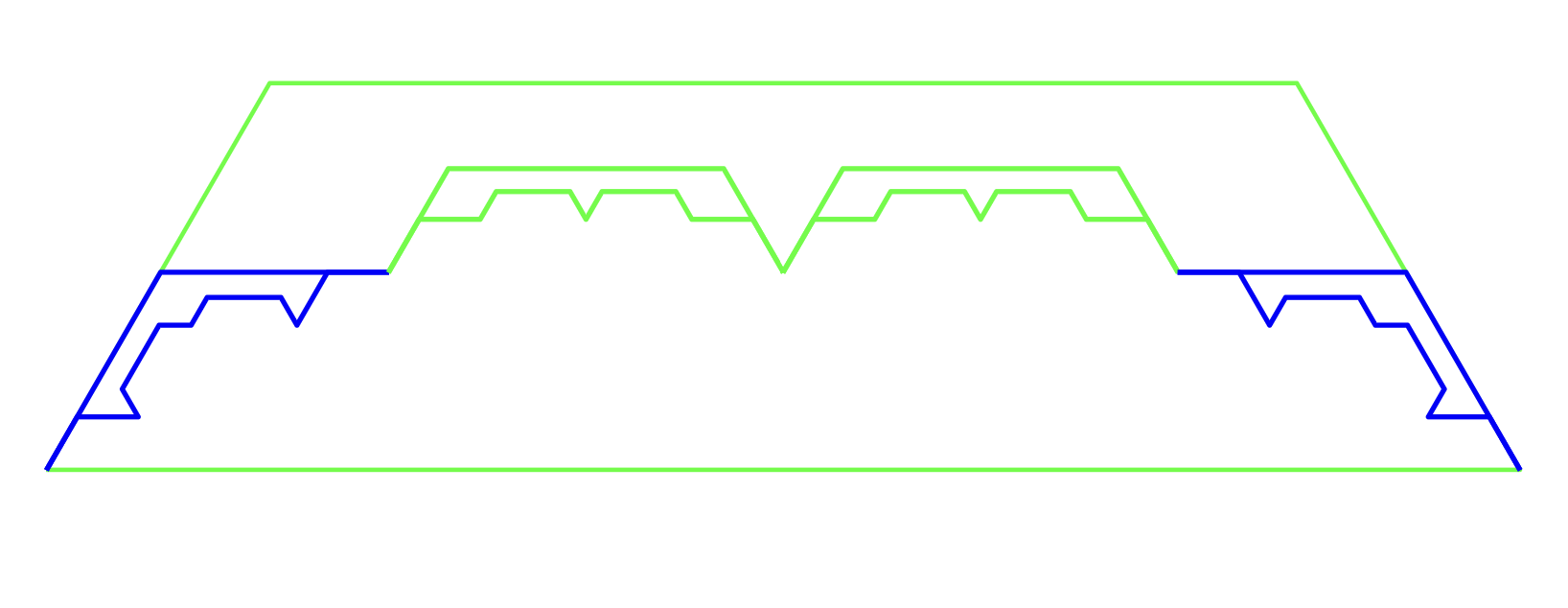}\quad\includegraphics[width= 4 cm]{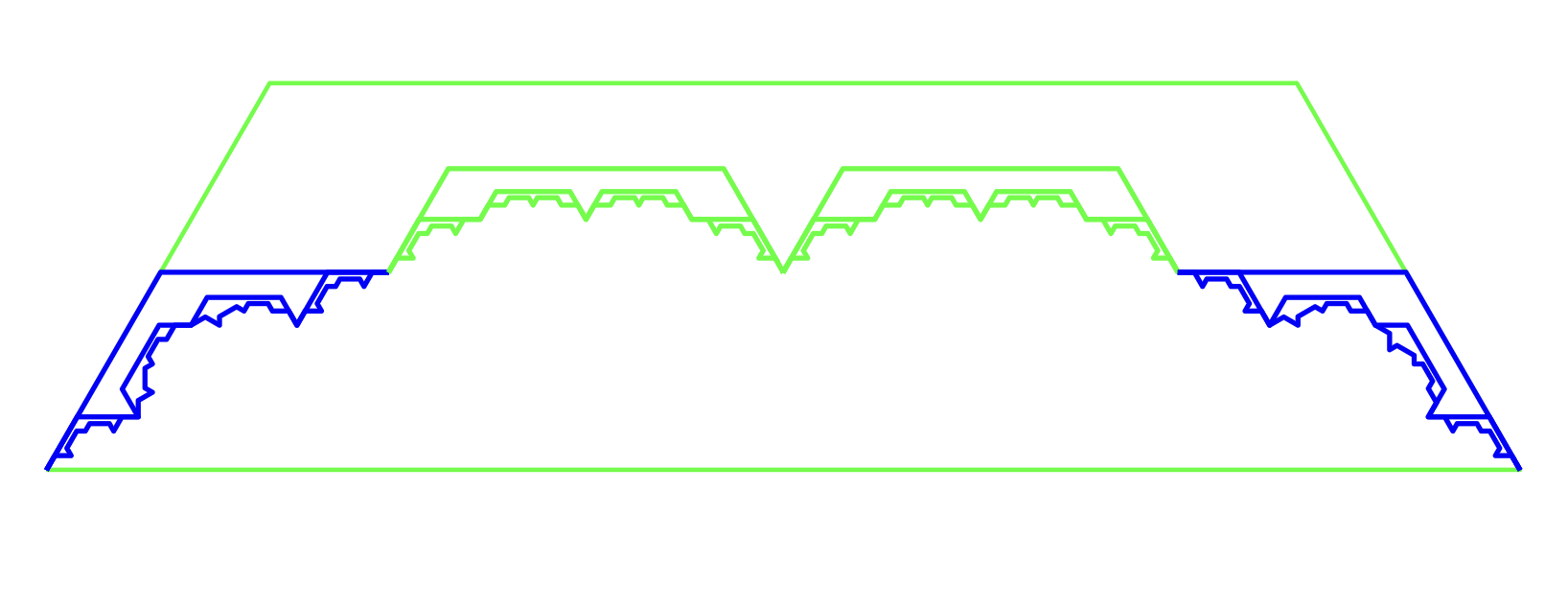}
\caption{Feasible open set for $K_2$: the interior of an isosceles trapezoid with base angle $\frac{\pi}{3}$, base edge with length $1$ and equal edges with length $r=\frac{\sqrt{13}-3}{2}$.}\label{SquareTriangle_tpyeB}
\end{figure}
\begin{figure}[h]
\includegraphics[width=4 cm]{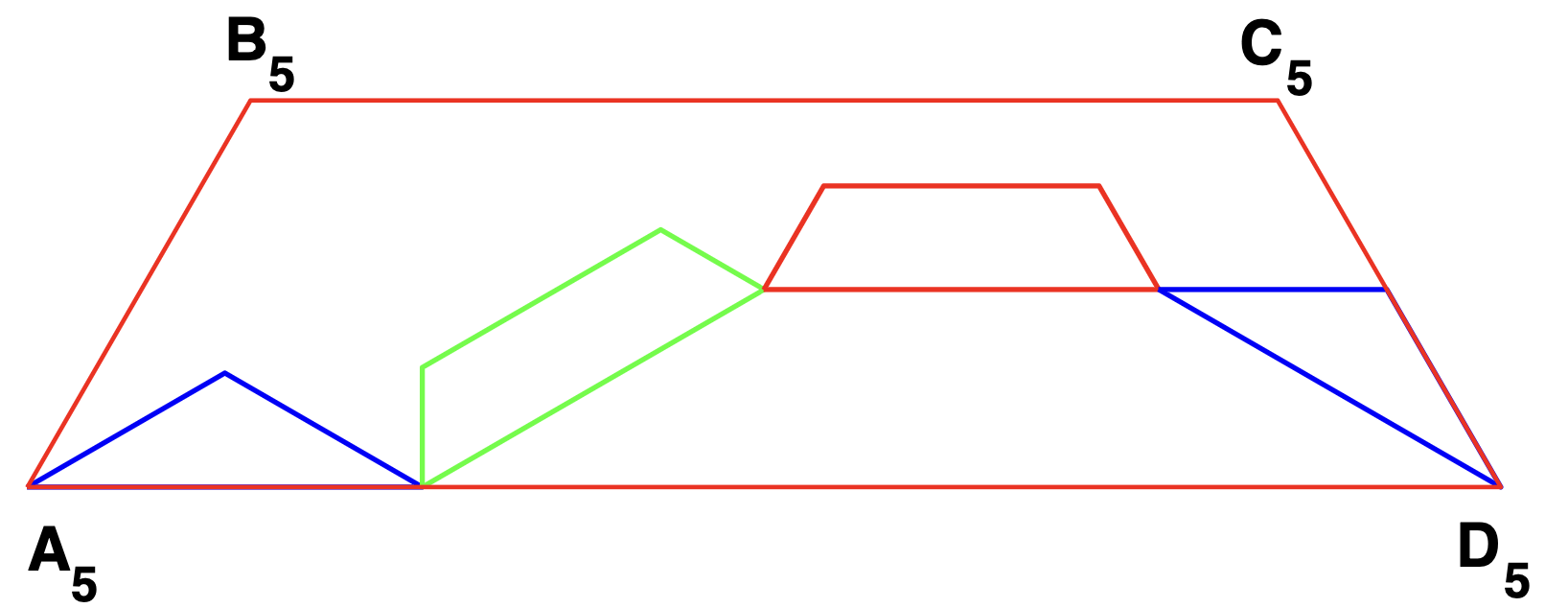}~\includegraphics[width= 4 cm]{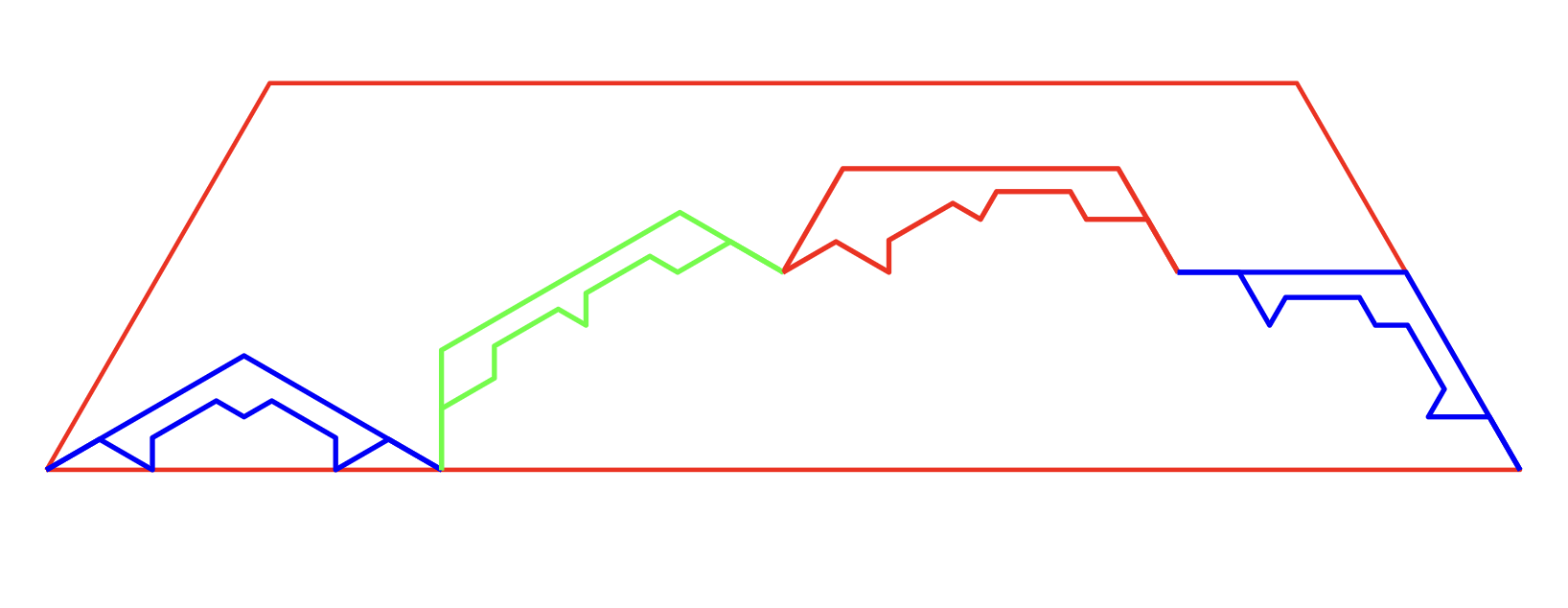}\quad\includegraphics[width= 4 cm]{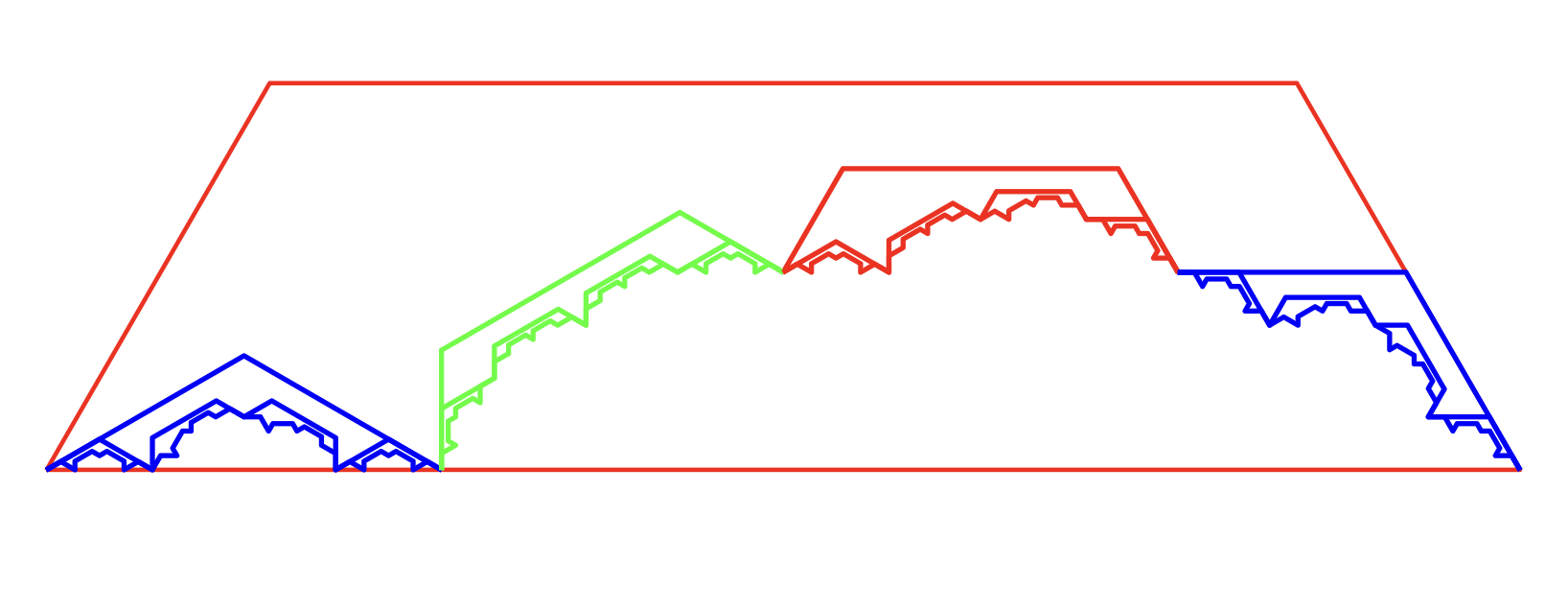}
\caption{Feasible open set for $K_3$: the interior of the isosceles trapezoid same as above Figure \ref{SquareTriangle_tpyeB}.}\label{SquareTriangle_tpyeC}
\end{figure}

Set the tiles related to square triangle tiling substitution by $T_1, T_2, T_3$ 
(produced by $Tr1,Tr2,Tr3$, respectively) and $T_4, T_5$ (produced by $R1, R2$ respectively.) By the same idea with Bronze-mean tiles we know that the boundaries of $T_i$ could be classified by three different types $K_1,K_2$ and $K_3$. There exist set equations on $K_i$ for $i=1,2,3$ as follows.
\begin{align*}
K_1&=g_{11}(K_1)\cup g_{13}(K_3)\cup g_{13}'(K_2)\cup g_{11}'(K_1),\\
K_2&=g_{21}(K_1)\cup g_{22}(K_2)\cup g_{22}'(K_2)\cup g_{21}'(K_1),\\
K_3&=g_{31}(K_1)\cup g_{32}(K_2)\cup g_{33}(K_3)\cup g_{31}'(K_1).
\end{align*}
Let $O_1$ be the interior of triangle $A_3B_3C_3$ in Figure \ref{SquareTriangle_tpyeA} and $O_2$ (and $O_3$) be the interior of the trapezoid $A_4B_4C_4D_4$ (and $A_5B_5C_5D_5$) in Figure \ref{SquareTriangle_tpyeB} and \ref{SquareTriangle_tpyeC}.  Thus $O_1,O_2$ and $O_3$ have the following inclusion relations which show the open set condition and satisfy the linear IFS condition.
\begin{align*}
 g_{11}(O_1)\cup g_{13}(O_3)\cup g_{13}'(O_2)\cup g_{11}'(O_1)&\subset O_1,\\
g_{21}(O_1)\cup g_{22}(O_2)\cup g_{22}'(O_2)\cup g_{21}'(O_1)&\subset O_2,\\
g_{31}(O_1)\cup g_{32}(O_2)\cup g_{33}(O_3)\cup g_{31}'(O_1)&\subset O_3.
\end{align*}

\end{ex}

\begin{ex}[{\color{blue}DZ substitution for $n=1,2,\dots$}]
\label{DZsubst}
Bronze mean substitution is generalized to a one-parameter family of overlapping substitutions as in Figure \ref{DZ}
where we use the right triangle $ST$ of length $1$, and the right triangle $BT$
of length $a=(1+\sqrt{13})/(2\sqrt{3})\approx 1.3295$,
and the rectangle $R$ of size $1\times a$.

The idea is to insert more rectangles into the overlapping boundary to form a larger substitution. One can check that 
the interior
is automatically decided from the shape of the expected boundaries. The substitution matrix of DZ substitution of level
$n$ is
$$
A_n=
\begin{pmatrix}
n^2 & 3 & 4n\\
3 & (1+n)^2 &4(1+n)\\
\frac{3n}2&\frac{3(1+n)}2& 3+n+n^2\\
\end{pmatrix}.
$$
These matrices are simultaneously diagonalized. Setting
$$
P=\begin{pmatrix}
4& \frac{-1 - \sqrt{13}}3 & \frac{-1 + \sqrt{13}}3\\
-4& \frac{1 - \sqrt{13}}3 & \frac{1 + \sqrt{13}}3\\ 
1& 1 & 1\\
\end{pmatrix},
$$
we have
$$
P^{-1}A_nP=
\begin{pmatrix}
-3 + n + n^2& 0& 0\\ 
0 & \left(\frac{2n+1- \sqrt{13}}2\right)^2& 0\\ 
0 & 0 & \left(\frac{2n+1+ \sqrt{13}}2\right)^2 \\
\end{pmatrix}
$$
and therefore $A_n$ an $A_m$ for $n\neq m$ are commutative and share the same eigenvectors.
The expansion factor of DZ substitution of level $n$ is $(2n+1+\sqrt{13})/2$, and its square
coincides with the Perron-Frobenius root of $A_n$. This is consistent with the criterion of
Lagarias-Wang \cite{Lagarias-Wang:03}. The common Perron-Frobenius right eigenvector 
$((-1+\sqrt{13})/3, (1+\sqrt{13})/3, 1)^T$ gives the frequency of tiles $ST, BT, R$ independent of $n$.
We can also check that the area of tiles $(\sqrt{3}/4,\sqrt{3}a^2/4,a)$ gives 
the common Perron-Frobenius left eigenvector. 
Similarly to Example \ref{Dotera}, we can check the consistency of these substitutions as in Figure \ref{DZ_OSC}
where we put the right triangle of size $a$ on the segment of length $a$ and the unit square on the unit segment.
One can easily verify that these feasible open sets satisfy the linear GIFS condition.

\begin{figure}[h]
    \centering
\includegraphics[width=12 cm]{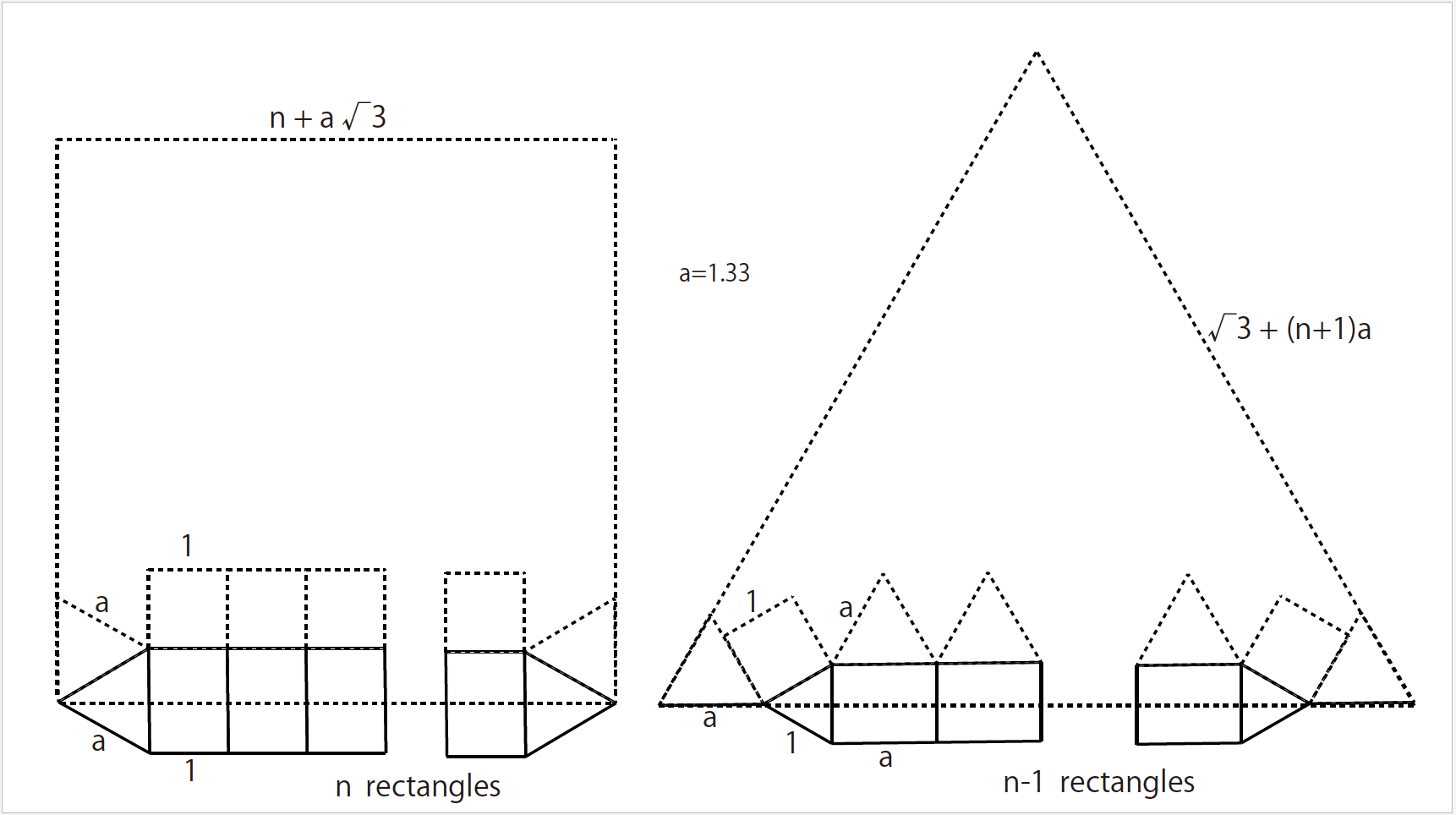}
    \caption{Feasible open sets for DZ substitutions\label{DZ_OSC}}
\end{figure}

This example is amazing in itself. We can use the same $ST, BT, R$ to define infinitely many
substitutions having different expansion constants. One can apply different substitutions at each step 
to obtain $S$-adic tilings which share the same frequency of tiles!
\end{ex}

\begin{rem}
Inspired by Example \ref{DZsubst}, we may create
a family of 1-dimensional substitutions
having two parameters $a,b$
with similar properties.
Put $m,c\in \Z$ with $m>c^2$. The matrices
$$
A_{a,b}:=\begin{pmatrix}
a-bc & b\\
b(m-c^2)& a+bc
\end{pmatrix}
$$
with $a,b\in \N$ and $a\ge |bc|$
are simultaneously diagonalized. Putting
$$
Q=
\begin{pmatrix}
\sqrt{m}+c& -1\\
\sqrt{m}-c& 1
\end{pmatrix},
$$
we have
$$
Q A_{a,b} Q^{-1}
=
\begin{pmatrix}
a-b\sqrt{m}&0\\
0& a+b\sqrt{m}
\end{pmatrix}.
$$
The substitution
$$
1\to \overbrace{1\dots1}^{a-bc}\overbrace{2\dots 2}^{b(m^2-c)},\quad
2\to \overbrace{1\dots1}^{b}\overbrace{2\dots 2}^{a+bc}.
$$
can be realized as tiling substitution using 
intervals of length $\sqrt{m}-c$ and $1$. Of course, the order of $1$ and $2$ can be arbitrary changed. The frequency of tiles is  $1:\sqrt{m}+c$
independent of $a$
and $b$.
\end{rem}

{\color{blue}We note that this OSC method described in this section is not satisfactory.
Figure \ref{BronzeSupple} gives a similar but different substitution inspired by Bronze-mean substitution, whose expansion constant is $\sqrt{2}+\sqrt{3}$. 
Edge lengths are $1$ and $\sqrt{2}$, and its substitution matrix is 
irreducible with period two.  
We can show that this substitution is consistent 
by classifying all local configurations
of adjacent edges under iteration. However, the transition graph associated with these pairs
of edges shows that the IFS of the boundary does not satisfy OSC. }

\begin{figure}[h]
    \centering
\includegraphics[width=3.5 cm]{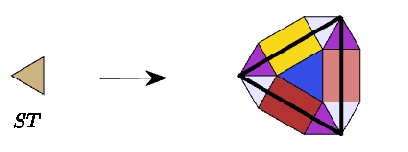}\includegraphics[width=3.5 cm]{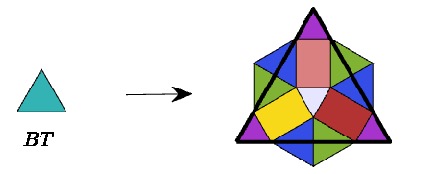}\includegraphics[width=3.5 cm]{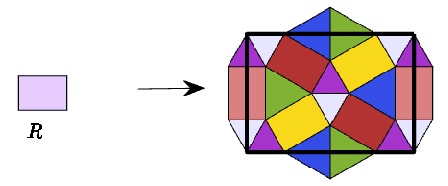}
    \caption{Consistent overlapping substitution without open set condition}
    \label{BronzeSupple}
\end{figure}

\section{Construction of overlapping substitutions}
\label{Sec5}

% \textcolor{green}{
% In this section, we construct several one-dimensional overlapping substitutions which allow deformations by one and two parameters.
% }

{\color{blue} In this section, we first give a general idea for constructing one-dimensional overlapping substitutions, which may depend on one or two parameters.  We also construct overlapping substitutions in higher dimensions from Delone sets with inflation symmetry.}

\subsection{A construction of one-dimensional
overlapping substitutions}

\label{section_construction_one-dim}
In this section, we show a way to construct
interesting one-dimensional overlapping
substitutions.

\textcolor{blue}{Let us briefly explain the main idea. First, we consider a symbolic weighted substitution (the definition will be given Step 1 below), such as in Example \ref{ex1_one-dim-substi}.
% \begin{ex}\label{ex1_symb_over_substi}
%      Let the alphabet $\mathcal{A}$ be $\{a,b,c\}$. Consider a rule $\sigma$
%     defined as
%     \begin{align*}
%         a\rightarrow [a]_{1/2}ba\\
%         b\rightarrow c[a]_{1/2}\\
%         c\rightarrow b[a]_{1/2}
%     \end{align*}
% \end{ex}
Then, we consider a geometric alphabet consisting of three tiles
$T_a, T_b, T_c$, 
which correspond to the letters $a,b,c$,
and, as in Figure \ref{Ex1Subst}, an inflation rule such that}
\begin{enumerate}
    \item the inflation of $T_a$ consists of $T_a, T_b$ and another $T_a$, from left to right, where the left half of the left-most $T_a$ protrudes from the expansion of the original $T_a$,
    \item the inflation of $T_b$ consists of $T_c$ and $T_a$, from left to right, where the right  half of $T_a$ protrudes from the expansion of the original $T_b$, and
    \item the inflation of $T_c$ consists of $T_b$ and $T_a$, from left to right, where the right half of $T_a$ protrudes from the expansion of $T_c$.
\end{enumerate}
We have to choose appropriate tile lengths and an expansion factor, but this is possible by choosing the Perron--Frobenius eigenvalues and eigenvectors for the substitution matrix \eqref{matrix_ex1_symbolic_weighted_substi}.
% \begin{align}
%     \begin{pmatrix}
%         3/2&1/2&1/2\\
%         1&0&1\\
%         0& 1&0
%     \end{pmatrix}.\label{matrix_Def5.1}
% \end{align}

However, we have to be careful because if a word $aa$ appears, then there will be an illegal overlap in the inflation. For this reason, we have to check all the possible adjacency rules, which we can do when we draw the adjacency graph in Step 2 below.

\textcolor{blue}{We now discuss the general construction. We divide the construction into four steps,
as follows, where we also give an example for an understanding.}
%of the method consists of
% \begin{enumerate}
%     \item find a ``symbolic overlapping
%     substitution'' $\sigma$ with an alphabet
%     $\mathcal{A}=\{1,2,\ldots ,\kappa\}$, 
%     for which we draw
%     graphs called adjacency graphs.
%     \item define the substitution matrix
%     for the symbolic overlapping substitution $\sigma$.
%     \item assuming the substitution 
%     matrix is primitive, we compute the
%     Perron--Frobenius eigenvalue $\beta$
%     and a Perron--Frobenius left
%     eigenvector $(l_1,l_2, \ldots ,l_{\kappa})$.
%     Then consider a tile $T_i=([0,l_i],i)$.
%     The set $\{T_1,T_2,\ldots ,T_{\kappa}\}$ of these
%     tiles becomes the alphabet for
%     (geometric) overlapping substitution,
%     which is obtained by juxtaposing
%     these $T_i$'s in the same manner as
%     the original symbolic overlapping 
%     substitution.
% \end{enumerate}

\emph{\underline{Step 1.}: construct a symbolic overlapping substitution.}
{\color{blue}A symbolic overlapping substitution $\sigma$ is a (symbolic) weighted substitution
 with an additional condition, as we explain below} (c.f. \cite{Kamae:05}).
 A weighted substitution on an alphabet
 $\mathcal{A}$ is a rule which maps
 $a\in\mathcal{A}$ to a string
 \begin{align*}
     [a_1]_{r_1}[a_2]_{r_2}\cdots [a_{m}]_{r_{m}},
 \end{align*}
 where $m>0, a_i\in\mathcal{A}$ and
 $0<r_i\leqq 1$.
We call such a map $\sigma$ a
symbolic overlapping substitution
if it satisfies the conditions.
\begin{enumerate}
    \item if $m=1$, then $r_1=1$.
    \item if $m>1$, then $r_i=1$ for any $1<i<m$.
\end{enumerate}
We omit the corresponding bracket of $[a_i]_{r_i}$ and simply denote by $a_i$  if $r_i$ is 1.

\emph{\underline{Step 2:} draw the adjacency graph for a symbolic weighted substitution.}
We will draw adjacency graphs $G_k$, $k=1,2,\ldots$, as follows.
They have the common vertex set $\mathcal{A}$.
The edges are defined in the following way.
First, we draw an edge $e\rightarrow f$
for $e,f\in\mathcal{A}$ if $[e]_r[f]_s$ is a subword
of the image of a letter by $\sigma$.
The graph
thus obtained is denoted by $G_1$.

Next, when we obtained $G_k$, we define $G_{k+1}$. 
In addition to edges in $G_k$, we draw 
an edge $e\rightarrow f$
if there is an edge $h\rightarrow g$ in $G_k$
such that 
    $\sigma(h)$ ends with $[e]_1$ and
    $\sigma(g)$ starts with $[f]_1$.

The growing graphs $G_1,G_2\ldots$ eventually stabilize
and the final graph is called the adjacency graph $G$ for
$\sigma$. 
This $G$ is said to be consistent if whenever
$e\rightarrow f$ is an edge in $G$, we have either
\begin{itemize}
    \item The last weight of $\sigma(e)$ and
    the first weight of $\sigma(f)$ are both 1, or
    \item The word $\sigma(e)$ ends with $[h]_r$ and
    $\sigma(f)$ starts with $[h]_{1-r}$, for some 
     $h\in\mathcal{A}$ and $r\in (0,1)$.
\end{itemize}
We only consider symbolic overlapping
substitutions with consistent adjacency graphs,
since this is necessary for the consistency for
the resulting geometric overlapping substitutions.

In practice, when we try to construct an 
example of geometric overlapping substitution by this
way, we usually first construct a graph $G$
and then find a symbolic overlapping substitution
with $G$ as a consistent adjacency graph.

%\begin{ex}

{\bf \textit{Adjacency graphs for Example \ref{ex1_one-dim-substi}.}}
By the definition of $G_k$, we have the $G_1$ and $G_2$ as in the following Figure \ref{Ex1-Gk}. Since $G_3=G_2$, we have $G_k=G_2$ for all $k\geq 3$.

\begin{figure}[h] % `d1[p1]`d2[p2] ... `dn[pn] [f]
\hskip 0.1cm \xymatrix{
G_1:& *++[o][F]{a}  
\ar@/^{1ex}/[r]& 
*++[o][F]{b}\ar@/^{1ex}/[l]& G_2=G_k: &*++[o][F]{a}  
\ar@/^{1ex}/[r]\ar@/^{1ex}/[d]& 
*++[o][F]{b}\ar@/^{1ex}/[l]\\
&*++[o][F]{c}\ar[u]&& &*++[o][F]{c}\ar@/^{1ex}/[u]& 
}
\caption{$G_1, G_2$ and $G_k$ for $k\geq 3$.}\label{Ex1-Gk}

\end{figure}
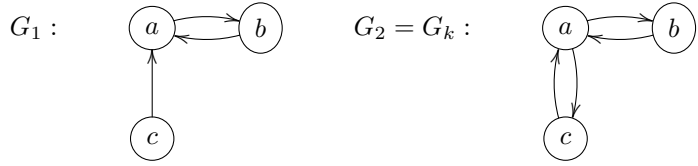

\emph{\underline{Step 3.}: compute the substitution matrix for the symbolic overlapping substitutions.}
Given a symbolic overlapping substitution $\sigma$ with alphabet $\mathcal{A}=\{1,2,\ldots ,\kappa\}$, we define its substitution matrix.
It is a $\kappa\times \kappa$ matrix $M$ such that
\begin{align*}
    M_{i,j}=
\sum_{k\colon a_k=i}r_k
\end{align*}
where the image of $j$
is written as
\begin{align*}
    \sigma(j)=[a_1]_{r_1}[a_2]_{r_2}\cdots [a_k]_{r_k}.
\end{align*}
In other words, we count the number of
appearances of $i$ in the image of $j$,
while regarding the appearances at the beginning
and the end are counted by their weights.
 
 The symbolic overlapping substitution is said to be
 primitive if its substitution matrix is
 primitive.
 
{\bf\textit{The substitution matrix for Example
    \ref{ex1_one-dim-substi}.}}  is given by \eqref{matrix_ex1_symbolic_weighted_substi}.
    % \begin{align*}
    %     M=
    %     \begin{pmatrix}
    %          3/2 &1/2&1/2\\
    %          1&0&1\\
    %          0&1&0
    %     \end{pmatrix}.
    % \end{align*}

 \emph{\underline{Step 4}: construct the corresponding overlapping (geometric) substitutions.}
 Assume the substitution matrix $M$ for a symbolic
 overlapping substitution $\sigma$ is primitive.
 Let $\beta$ be the Perron--Frobenius eigenvalue
 and $(l_1,l_2,\ldots ,l_{\kappa})$ be a Perron--Frobenius
 left eigenvector with positive entries.
 We construct a set of tiles
 $\{T_1,T_2,\ldots ,T_{\kappa}\}$ by
 $T_i=([0,l_i],i)$. 
 By definition, we have
 \begin{align*}
     \sum_{j=1}^{\kappa}\sum_{i\colon a_i=j}r_il_j=\beta l_k,
 \end{align*}
 if
 \begin{align*}
     \sigma(k)=[a_1]_{r_1}[a_2]_{r_2}\cdots [a_k]_{r_k}.
 \end{align*}
 We then define a (geometric)
 pre-overlapping substitution $\rho$ by
 \begin{align*}
     \rho(T_i)=\left\{T_{a_1}, T_{a_2}+l_{a_1},
     \ldots ,T_{a_k}+\sum_{j=1}^{k-1}l_{a_j}\right\}-(1-r)l_{a_1},
 \end{align*}
 where
 \begin{align*}
     \sigma(i)=[a_1]_ra_2a_3\cdots a_{k-1}[a_k]_s.
 \end{align*}
 In other words, we juxtapose the copies of alphabet
 by the order of $\sigma$ and translate every tile by
 $(1-r)l_{a_1}$, so that a part of the initial $T_{a_1}$
 ``sticks out'' from the enlarged original tile $[0,\beta l_i]$
 by the ratio $1-r$.
 We see a part of the final tile $T_{a_k}$ also
 ``sticks out'' 
 from the enlarged tile $[0,\beta l_i]$
 by the ratio $1-s$ (if $s<1$), by the
 computation
 \begin{align*}
     \beta l_i=rl_{a_1}+l_{a_2}+l_{a_3}+\cdots +l_{a_{k-1}}+sl_{a_k}.
 \end{align*}
{\bf \textit{The (geometric) overlapping substitution for Example \ref{ex1_one-dim-substi}.}}
It is easy to know that the Perron--Frobenius eigenvalue of $M$ in Example \ref{ex1_one-dim-substi} is $2$ and the corresponding left eigenvector is $(l_a,l_b,l_c)=(2,1,1)$. Then we have $$T_a=([0,l_a],1), T_b=([0,l_b],2) \text{ and } T_c=([0,l_c],3).$$

And the (geometric) overlapping substitution $\rho$ (see Figure \ref{Brokenline5.1}) by
$$\left\{\begin{aligned}
\rho(T_a)&= \{[0,2],[0,1]+l_{a},[0,2]+l_a+l_b\}-(1-1/2) l_a= \{[-1,1],[1,2],[2,4]\}\\
\rho(T_b)&=\{[0,1],[0,2]+l_c\}-(1-1)l_c=\{[0,1],[1,3]\}\\
\rho(T_c)&=\{[0,1],[0,2]+l_b\}-(1-1) l_b=\{[0,1],[1,3]\}
\end{aligned}
\right.
$$ 
\begin{figure}[h]
\centering
\includegraphics[width=10 cm]{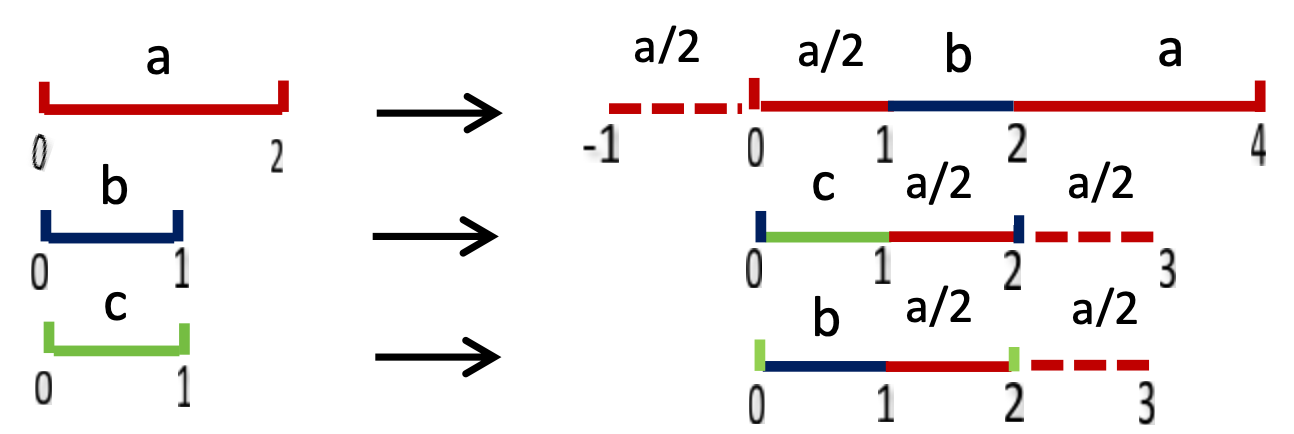}
\caption{Geometric overlapping substitution}\label{Brokenline5.1}
\end{figure}
by
$$\left\{\begin{aligned}
\sigma(a)&= [a]_{\frac{1}{2}}[b]_1[a]_1\\
\sigma(b)&=[c]_1[a]_{\frac{1}{2}}\\
\sigma(c)&=[b]_1[a]_{\frac{1}{2}}
\end{aligned}
\right.,
$$

 In the next proposition, we show that the (geometric)
 pre-overlapping substitution $\rho$ constructed
in this way is consistent. 
%First, we observe a lemma, for which the proof is easy.

% \begin{lem}
%     Let $t$ be an integer greater than 2 and $a_1,a_2,\ldots ,a_t\in\mathcal{A}$.
%     Assume $a_i\rightarrow a_{i+1}$ is in $G$ for each $i$. Then
%     $\Int\rho(T_{a_1})\cap \Int\rho(T_{a_t}+l_{a_1}+l_{a_2}+\cdots +l_{a_{t-1}})=\emptyset$.
%     Moreover, we have
%     $\supp\rho(T_{a_1})\cap \supp\rho(T_{a_t}+l_{a_1}+l_{a_2}+\cdots +l_{a_{t-1}})\neq\emptyset$
%     only when $t=3$ and
%     the patch $\rho(T_{a_2}+l_{a_1})$ consists of the last tile of $\rho(T_{a_1})$ and the first tile of $\rho(T_{a_3}+l_{a_1}+l_{a_2})$
% \end{lem}
 
 \begin{prop}
     The pre-overlapping substitution $\rho$ is consistent.
 \end{prop}
 \begin{proof}
     Define another graph
     $G'_k$, $k=1,2,\ldots,$ as follows.
     The set of vertices for $G_k'$ is $\mathcal{A}$. If
     for some $\nu=1,2,\ldots ,k$ and $a\in\mathcal{A}$, a translate of the patch $\{T_i,T_j+l_i\}$ appears in $\rho^{\nu}(T_a)$,
     we draw an edge $i\rightarrow j$ for
     $G_k'$. We see $G_1=G_1'$.
     Let $k$ be an integer greater than 0. If $G_k=G_k'$ and
     $\rho^k(T_a)$ is consistent (that is, there are no illegal overlaps) for each $a\in\mathcal{A}$, then we can show that
     $G_{k+1}=G'_{k+1}$ and
     $\rho^{k+1}(T_a)$ is
     consistent for each $a\in\mathcal{A}.$
 \end{proof}

 In the method above, it is not difficult to construct one-dimensional overlapping substitutions, which may depend on parameters. We finish this section by giving some examples.

 \begin{ex}\label{Example5.2}
Let the alphabet $\mathcal{A}$ be $\{a,b,c,d\}$. Consider a rule $\sigma$ defined as 
\quad\\
$\sigma: \left\{\begin{aligned}
a&\longrightarrow [a]_{r}bca\\
b&\longrightarrow bc\\
c&\longrightarrow d[a]_{1-r}\\
d&\longrightarrow [a]_rbc[a]_{1-r}\\
\end{aligned}
\right.
$ for any $0< r< 1$.
\\

Then the adjacent graphs determined by the above rule are drawn in Figure \ref{Ex2-Gk}.  By the definition, we know that $G_k=G_2$ for $k\geq 3$.

\begin{figure}[h] % `d1[p1]`d2[p2] ... `dn[pn] [f]
\hskip 0.1cm \xymatrix{
G_1:&*++[o][F]{d}  
\ar@/^{0ex}/[r]& *++[o][F]{a}  
\ar@/^{0ex}/[r]& 
*++[o][F]{b}\ar@/^{0ex}/[ld]& G_k=G_2:&*++[o][F]{d}  
\ar@/^{0ex}/[r]&*++[o][F]{a}  
\ar@/^{0ex}/[r]& 
*++[o][F]{b}\ar@/^{0ex}/[ld]& \\
&&*++[o][F]{c}\ar[u]&& &&*++[o][F]{c}\ar@/^{0ex}/[u]\ar@/^{0ex}/[ul]& &
}
\caption{$G_1, G_2$ and $G_k$  for $k\geq 3$.} \label{Ex2-Gk}

\end{figure}
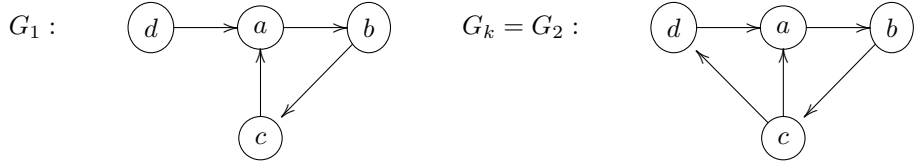
%G_k:&*++[o][F]{d}  
%\ar@/^{1ex}/[r]&*++[o][F]{a}  
%\ar@/^{1ex}/[r]\ar@/^{1ex}/[d]& 
%*++[o][F]{b}&
%&&&*++[o][F]{c}\ar@/^{1ex}/[u]&
The substitution matrix is 
$$M=\left(\begin{matrix}
1+r& 0 & 1-r & 1 \\
1&1& 0&1 \\
1& 1 & 0 & 1 \\
0& 0 &1 & 0 \\
\end{matrix}
\right).
$$
And the characteristic polynomial is $P_M(x)=x^4-(2+r)x^3+(2r-1)x^2+rx$.\\
The Perron--Frobenius eigenvalue is $\beta=1+\sqrt{2}$ and the right eigenvector for $\beta$ is $(1-\frac{\sqrt{2}}{2}, 1-\frac{\sqrt{2}}{2}, 1-\frac{\sqrt{2}}{2}, \frac{3\sqrt{2}}{2}-2)$ but the left eigenvector 
$$(l_a,l_b,l_c,l_d)=\left(\frac{\beta}{\beta-r},\frac{\beta-r-1}{\beta-r},\sqrt{2}\frac{\beta-r-1}{\beta- r},1\right)$$
varies by the parameter $r$. We take $r=\sqrt{2}/2$ as an example then the left eigenvector for $\beta$ is $(l_a,l_b,l_c,l_d)=(\sqrt{2},\sqrt{2}-1,2-\sqrt{2},1)$. Then we have 
$$T_a=([0,l_a],1), T_b=([0,l_b],2), T_c=([0,l_c],3),T_d=([0,l_d],4).$$
So the (geometric) overlapping substitution $\rho$ is
$$\left\{\begin{aligned}
\rho(T_a)&= \{[0,l_a],[0,l_b]+l_{a},[0,l_c]+l_a+l_b,[0,l_a]+l_a+l_b+l_c\}-(1-r) l_a\\
&= \{[1-\sqrt{2},1],[1,\sqrt{2}],[\sqrt{2},2],[2,2+\sqrt{2}]\}\\
\rho(T_b)&=\{[0,l_b],[0,l_c]+l_b\}-(1-1)l_b=\{[0,\sqrt{2}-1],[\sqrt{2}-1,1]\}\\
\rho(T_c)&=\{[0,l_d],[0,l_a]+l_d\}-(1-1) l_d=\{[0,1],[1,1+\sqrt{2}]\}\\
\rho(T_d)&=\{ [0,l_a],[0.l_b]+l_a,[0,l_c]+l_a+l_b,[0,l_a]+l_a+l_b+l_c\}-(1-r) l_a\\
&= \{[1-\sqrt{2},1],[1,\sqrt{2}],[\sqrt{2},2],[2,2+\sqrt{2}]\}
\end{aligned}
\right.
$$ 
\begin{figure}[h]
\centering
\includegraphics[width=10 cm]{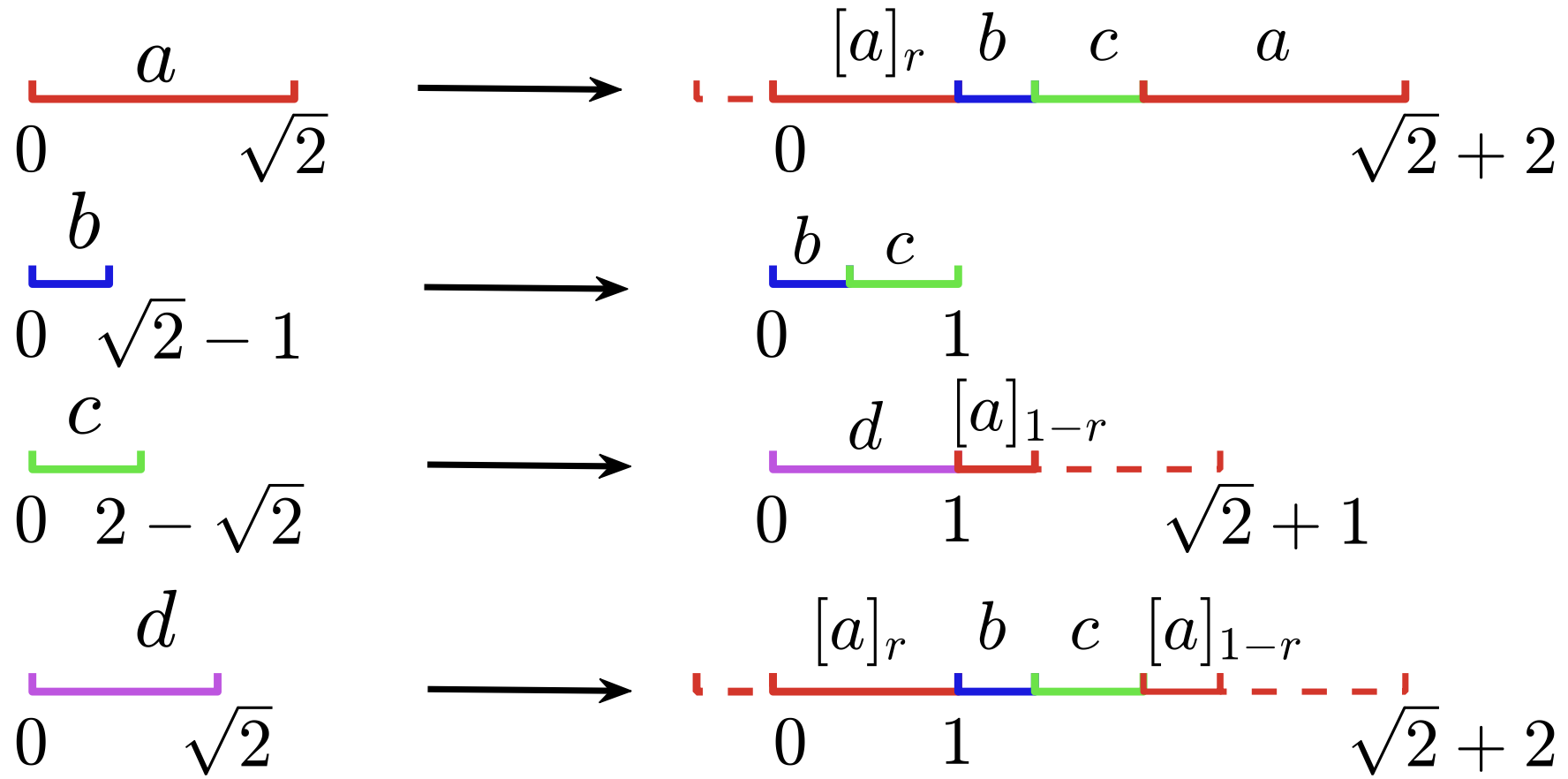}
\caption{Geometric overlapping substitution for Example \ref{Example5.2} when $r=\sqrt{2}/2$.}\label{Brokenline5.2}
\end{figure}

% by
% $$\left\{\begin{aligned}
% \sigma(a)&= [a]_{r}[b]_1[c]_1[a]_1\\
% \sigma(b)&=[b]_1[c]_1\\
% \sigma(c)&=[d]_1[a]_{1-r}\\
% \sigma(d)&=[a]_r[b]_1[c]_1[a]_{1-r}
% \end{aligned}
% \right.,
% $$
% here $r=\sqrt{2}/2$.

 \end{ex}

\begin{ex}[An example with two parameters]
Let the alphabet $\mathcal{A}$ be $\{a,b,c,d,e\}$. Consider a rule $\sigma$ defined as 

\quad
$\sigma: \left\{\begin{aligned}
a&\longrightarrow [c]_{1-s}abcade[a]_r\\
b&\longrightarrow [a]_{1-r}de[c]_s\\
c&\longrightarrow [c]_{1-s}ab[c]_{s}\\
d&\longrightarrow [a]_{1-r}d\\
e&\longrightarrow eab[c]_s\\
\end{aligned}
\right.
$ for any $0< r,s <1$.

By the above rule we have the adjacent graphs shown in Figure \ref{Ex3-Gk}.

\begin{figure}[h] % `d1[p1]`d2[p2] ... `dn[pn] [f]
\hskip 0.1cm \xymatrix{
G_1:&*++[o][F]{e}  
\ar@/^{0ex}/[rd]& *++[o][F]{c}\ar@/^{0ex}/[d]
&& G_k=G_1:&*++[o][F]{e}  \ar@/^{0ex}/[rd]&*++[o][F]{c}  
\ar@/^{0ex}/[d]& \\
&*++[o][F]{d}\ar[u] &*++[o][F]{a}\ar[l]\ar[r]&*++[o][F]{b}\ar@/^{0ex}/[lu] &&*++[o][F]{d}\ar@/^{0ex}/[u]& *++[o][F]{a}\ar[l]\ar[r] &*++[o][F]{b}\ar@/^{0ex}/[lu] 
}
\caption{$G_1$ and $G_k$  for $k\geq 2$.}\label{Ex3-Gk}

\end{figure}
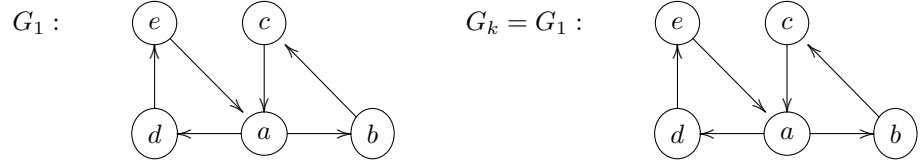

The substitution matrix is 

$$M=\left(\begin{matrix}
2+r& 1-r &1& 1-r & 1 \\
1&0&1& 0&1 \\
2-s&s& 1 & 0 & s \\
1& 1 &0 & 1&0 \\
1& 1 &0 & 0&1 \\
\end{matrix}
\right)
.$$\\
\medskip
So the characteristic polynomial with respect to $M$ is $$P_M(x)=(r - x) (x-1) (x^3- 4 x^2- x +2),$$
 and $\beta$ is the Perrod--Frobenius root $\approx 4.12$  of $x^3-4x^2-x+2=0$.  
A right eigenvector for $\beta$ is 
 $$(\beta,\beta-2,\beta^2-3 \beta-2,2,2).$$
A left Perron--Frobenius eigenvector 
\begin{align*}
&(\beta^2 (7-s)+\beta s-3,(\beta-1) \left(\beta^2
   (1-r)+\beta (r+s)-1\right),(\beta-1) \beta (\beta-r),\\
&(1-r)
   \left(2 \beta^2-\beta s+\beta-1\right),\beta (\beta-r) (\beta+s-1) )
\end{align*}
is a function on $r$ and $s$.
Taking $r=s=1/2$ as an example, we have 
a left eigenvector $$(l_a,l_b,l_c,l_d,l_e)=(2 \left(13 \beta^2+\beta-6\right),2 \beta (5 \beta-3),2 \beta(5 \beta+3)-8,4
   \beta^2+\beta-2,\beta (12 \beta+5)-8).$$
Then we have $T_a=([0,l_a],1), T_b=([0,l_b],2),T_c=([0,l_c],3),T_d=([0,l_d],4), T_e=([0,l_e],5).$
 So we have the geometric overlapping substitution $\rho$  by $\sigma$ in the following way
$$\left\{\begin{aligned}
\rho(T_a)&= \{[0,l_c],[0,l_a]+l_{c},[0,l_b]+l_a+l_c,[0,l_c]+l_a+l_b+l_c,[0,l_a]+l_c+l_a+l_b+l_c,\\
&[0,l_d]+2l_c+2l_a+l_b, [0,l_e]+2l_c+2l_a+l_b+l_d, [0,l_a]+2l_c+2l_a+l_b+l_d+l_e\}-sl_c\\
&= \{[-\frac{l_c}{2},\frac{l_c}{2}],[\frac{l_c}{2},l_a+\frac{l_c}{2}],[l_a+\frac{l_c}{2},l_a+l_b+\frac{l_c}{2}], [l_a+l_b+\frac{l_c}{2},l_a+l_b+\frac{3l_c}{2}],\\
&[l_a+l_b+\frac{3l_c}{2},2l_a+l_b+\frac{3l_c}{2}], [2l_a+l_b+\frac{3l_c}{2},2l_a+l_b+\frac{3l_c}{2}+l_d],\\
&[2l_a+l_b+\frac{3l_c}{2}+l_d,2l_a+l_b+\frac{3l_c}{2}+l_d+l_e], [2l_a+l_b+\frac{3l_c}{2}+l_d+l_e,3l_a+l_b+\frac{3l_c}{2}+l_d+l_e]\}\\
\rho(T_b)&=\{[0,l_a],[0,l_d]+l_a,[0,l_e]+l_a+l_d,[0,l_c]+l_a+l_d+l_e\}-(1-(1-r))l_a\\
&=\{[-\frac{l_a}{2},\frac{l_a}{2}],[\frac{l_a}{2},\frac{l_a}{2}+l_d],[\frac{l_a}{2}+l_d,\frac{l_a}{2}+l_d+l_e], [\frac{l_a}{2}+l_d+l_e,\frac{l_a}{2}+l_d+l_e+l_c]\}\\
\rho(T_c)&=\{[0,l_c],[0,l_a]+l_c,[0,l_b]+l_c+l_a,[0,l_c]+l_c+l_a+l_b\}-(1-(1-s)) l_c\\
&=\{[-\frac{l_c}{2},\frac{l_c}{2}],[\frac{l_c}{2},\frac{l_c}{2}+l_a],[\frac{l_c}{2}+l_a,\frac{l_c}{2}+l_a+l_b], [\frac{l_c}{2}+l_a+l_b,\frac{3l_c}{2}+l_a+l_b]\}\\
\rho(T_d)&=\{ [0,l_a],[0, l_d]+l_a,\}-(1-(1-r)) l_a= \{[-\frac{l_a}{2},\frac{l_a}{2}],[\frac{l_a}{2},\frac{l_a}{2}+l_d]\}\\
\rho(T_e)&=\{[0,l_e],[0,l_a]+l_e,[0,l_b]+l_e+l_a,[0,l_c]+l_e+l_a+l_b\}-(1-1)l_a\\
&=\{[0,l_e],[l_e,l_a+l_e],[l_a+l_e,l_a+l_b+l_e],[l_a+l_b+l_e,l_a+l_b+l_c+l_e]\}
\end{aligned}
\right.
$$ 
here
$\left\{\begin{aligned}
\sigma(a)&= [c]_{\frac{1}{2}}abcade[a]_{\frac{1}{2}}\\
\sigma(b)&= [a]_{\frac{1}{2}}de[c]_{\frac{1}{2}}\\
\sigma(c)&= [c]_{\frac{1}{2}}ab[c]_{\frac{1}{2}}\\
\sigma(d)&= [a]_{\frac{1}{2}}d\\
\sigma(e)&= eab[c]_{\frac{1}{2}}\\
\end{aligned}
\right.
$
\end{ex}

\subsection{Construction in high-dimensions from Delone sets with inflation symmetry}
\label{subsec_construction_general_dim}

In this subsection, we construct overlapping substitutions from Delone (multi)sets
with inflation symmetry. 
The idea is that, if a Delone (multi)set $X$ has an inflation symmetry,
then the corresponding Voronoi tiling must also have one and yields an overlapping substitution.
A similar idea to construct a self-similar tiling from a point set satisfying a
set equation is found in \cite{Feng-Wen:02} in 
one-dimensional setting. 
Note also that \cite{Solomyak_pseudo-self-similar} gave a construction of substitution from a tiling with inflation symmetry. Our construction is weaker in the sense that the substitutions constructed are
overlapping, but stronger in the sense that the tiles are polygonal.
(In \cite{Solomyak_pseudo-self-similar}, tiles are not necessarily
connected and can have fractal boundaries.)
We first construct Delone sets with
inflation symmetry, and then overlapping
substitutions and tilings.

Here is the setting for this subsection.
We fix a positive integer $d>0$ and a linear map
$A\colon\mathbb{R}^d\rightarrow\mathbb{R}^d$. 
We assume the space $\mathbb{R}^d$
admits a decomposition
$\mathbb{R}^d=E_u\oplus E_s$ into
linear subspaces $E_u,E_s$ such that on
$E_u$, $A$ is expanding, and
on $E_s$, $A$ contracting, and so
there exist $\lambda>1$ and $\lambda'\in (0,1)$ such that
\begin{align*}
    \|A(v)\|\geqq\lambda\|v\|
\end{align*}
for $v\in E_u$ and
\begin{align*}
    \|A(v)\|\leqq\lambda'\|v\|
\end{align*}
for $v\in E_s$. Moreover, we assume
that $A(E_u)\subset E_u$ and
$A(E_s)\subset E_s$.
Let $\pi_u$ be the projection 
of $\mathbb{R}^d$onto
$E_u$ with kernel $E_s$, and
$\pi_s$ be the projection onto $E_s$
with kernel $E_u$.

We take a lattice $L$ of $\mathbb{R}^d$
such that the restriction of $\pi_u$
to $L$ is injective and $AL\subset L$. We have a star-map
\begin{align*}
    \pi_u(L)\ni x=\pi_u(v)\mapsto x^*=\pi_s(v)\in E_s.
\end{align*}

Take a compact subset $K$ of
$E_u$ such that the
union $\Omega=\bigcup_{n=0}^{\infty}A^nK$ is convex and closed under addition. 
We will
construct a Delone subset of $\Omega$,
from which we construct an overlapping
substitution.

We first prove that
\begin{prop}\label{prop1_construction_high-dim}
    Let $D$ be a finite subset of $\pi_u(L)\cap K$ such that
    \begin{align*}
        \bigcup_{d\in D}(A^{-1}(K)+d)\supset K.
    \end{align*}
Then the set
\begin{align*}
    \Lambda=\left\{\sum_{n=0}^NA^n(d_n)\mid N>0, d_n\in D\right\}
\end{align*}
is a Delone subset of $\Omega$ of finite
 local complexity with an
inflation symmetry
\begin{align*}
    \Lambda=\bigcup_{d\in D}A\Lambda+d.
\end{align*}
\end{prop}

\begin{proof}
    For each element 
    \begin{align*}
       x=\sum_{n=0}^NA^nd_n 
    \end{align*}
    of $\Lambda$, the corresponding
    element in $E_s$
    \begin{align*}
        x^*=\sum_{n=0}^NA^nd_n^*
    \end{align*}
    admits a norm estimate
    \begin{align*}
        \|x^*\|\leqq\sum_{n=0}^N(\lambda')^n\|d_n^*\|\\
        \leqq\frac{1}{1-\lambda'}\max_{d\in D}\|d^*\|,
    \end{align*}
    and so for each compact $K\subset E_u$, the set
    \begin{align*}
        (\Lambda-\Lambda)\cap K
    \end{align*}
    is a finite set. This implies 
    $\Lambda$ is uniformly discrete
    and has finite local complexity.

    To prove that $\Lambda$ is relatively
    dense in $\Omega$, take a natural 
    number $n$ and an
    element $x_n\in K$. There exist
    $x_{n-1}\in K$ and $d_{n-1}\in D$
    such that
    \begin{align*}
        x_n=A^{-1}x_{n-1}+d_{n-1}.
    \end{align*}
    There exist $x_{n-2}\in K$ and
    $d_{n-2}\in D$ such that
    \begin{align*}
        x_{n-1}=A^{-1}x_{n-2}+d_{n-2}.
    \end{align*}
    By continuing this process, we obtain
    $x_{n-1},x_{n-2},\ldots$ and
    $d_{n-1},d_{n-2},\ldots$. We see
    \begin{align*}
        x_n=d_{n-1}+A^{-1}d_{n-2}+\cdots
        A^{-n}d_{-1}+A^{-n-1}(x_{-1}),
    \end{align*}
    and
    \begin{align*}
        A^nx_n=A^nd_{n-1}+A^{n-1}d_{n-2}+\cdots d_{-1}+A^{-1}x_{-1}\in\Lambda+A^{-1}K.
    \end{align*}
\end{proof}
\begin{rem}
    By the compactness argument, it is
    easy to prove the existence of such
    a $D$, and for an example of $K$,
    it is often easy to get an example
    of $D$ which satisfies the condition
    of the proposition.
\end{rem}
From $\Lambda$, we construct an overlapping substitution.
We first define the half-spaces: set
\begin{align*}
 H_{x,y}=\{z\in E_u\mid \|x-z\|\leqq \|y-z\|\}
\end{align*}
for each $x,y\in\mathbb{R}^d$.
We will define tiles as finite 
intersections of half-spaces.

% Let $X$ be a Delone set in $\Omega$.
For each $x\in \Lambda$, we set the Voronoi cell by
\begin{align*}
 V_x=\{y\in\Omega\mid \text{$\|x-y\|\leqq\|y-z\|$ for any $z\in \Lambda\setminus\{x\}$}\}.
\end{align*}
Take an $L>0$ large enough and  an $x\in X$, and let $T_{x,L}$ be a labeled tile with support $V_x$ and a label that encodes
the local structure of $\Lambda$ within $B(x,L)$.

We set Voronoi tiling $\mathcal{T}_{L}$ via
\begin{align*}
  \mathcal{T}_{L}=\{T_{x,L} \mid x\in \Lambda\}.
\end{align*}

We will find an overlapping substitution
via a cutting-off operation.
A cutting-off operation for patches is defined as follows:
\begin{align*}
 \mathcal{P}\sqcap K=\{T\in\mathcal{P}\mid \supp T\cap K\neq\emptyset\},
\end{align*}
where $\mathcal{P}$ is a patch in $E_u$ and $K\subset E_u$.

\begin{theorem}\label{thm_const_from_Delone}
If $L$ is large enough, 
% so that
% \begin{align*}
%  AB_{L-R_0-1}\supset B_{R_1+R_L+R_D},
% \end{align*}
let us define an alphabet $\mathcal{A}$ via
\begin{align*}
 \mathcal{A}=\{T_{x,L}-x\mid x\in \Lambda\}.
\end{align*}
Then $\mathcal{A}$ is a finite set, 
a map
\begin{align*}
 \rho(T_{x,L}-x)=\mathcal{T}_{L}\sqcap A(\supp T_{x,L})-A(x)
\end{align*}
is well-defined, and $(\mathcal{A}, A,\rho)$ is a consistent overlapping substitution of $E_u$.

Moreover
\begin{align*}
 \rho(\mathcal{T}_{L})=\mathcal{T}_{L}.
\end{align*}
\end{theorem}

We omit the technical but straightforward proof.

% \begin{proof}
%  If $T_{x,L}-x=T_{y,L}-y$, then we have
% \begin{align*}
%  (\mathcal{T}_L-x)\sci\{0\}= (\mathcal{T}_L-y)\sci\{0\},
% \end{align*}
% and by Lemma \ref{lem3_const_from_Delone},
% \begin{align*}
%  (M-x)\sci B_{L-R_0-1}= (M-y)\sci B_{L-R_0-1}.
% \end{align*}
% This means that
% \begin{align*}
%  (AM-Ax)\sci AB_{L-R_0-1}= (AM-Ay)\sci AB_{L-R_0-1},
% \end{align*}
% and so
% \begin{align*}
%  (M-Ax)\sci B_{R_1+R_L}= (M-Ay)\sci B_{R_1+R_L},
% \end{align*}
% by the inflation symmetry for $M$.
% Moreover, we have $(\Omega-Ax)\sci B_{R_1+R_3}=(\Omega-Ay)\sci B_{R_1+R_3}$.
% By Corollary \ref{cor1_const_from_Delone}, we have
% \begin{align*}
%  (\mathcal{T}_{L}-x)\sci B_{R_1}= (\mathcal{T}_{L}-y)\sci B_{R_1},
% \end{align*}
% from which the claim follows.
% \end{proof}

% \begin{rem}
%  The tiles have many labels, but the last one can be removed by considering only $x\in X$
% that is ``inside'' $\Omega$.
% After such a removal, if $X$ has finite local complexity, then $\mathcal{A}$ is finite.
% \end{rem}

\begin{ex}
    We take a matrix
    \begin{align*}
        A=\begin{pmatrix}
            0 &1 &0\\
            0 &0 &1 \\
            -3 &4 &1
        \end{pmatrix}
    \end{align*}
    and consider the linear map
    on $\mathbb{R}^3$ by the multiplication
    by $A$. There are three eigenvalues
    $\lambda_1\approx 2.20, \lambda_2\approx -1.91, \lambda_3\approx 0.714$. The unstable
    space $E_u$ is spanned by eigenvectors
    for $\lambda_1$ and $\lambda_2$, and
    the stable subspace $E_s$ is spanned
    by eigenvectors for $\lambda_3$.
    We choose the lattice $L=\mathbb{Z}^3$ .

   To draw a picture, we transform $E_u$ onto $\mathbb{R}^2$ by applying a map $\psi\colon\mathbb{R}^3\rightarrow\mathbb{R}^3$ defined by
    \begin{align*}
        \psi\colon\begin{pmatrix}
            1\\\lambda_1\\\lambda_1^2
        \end{pmatrix}
        \mapsto\begin{pmatrix}
            1\\0\\0
        \end{pmatrix},
        \begin{pmatrix}
            1\\\lambda_2\\\lambda_2^2
        \end{pmatrix}
        \mapsto\begin{pmatrix}
            0\\1\\0
        \end{pmatrix},
        \begin{pmatrix}
            1\\\lambda_3\\\lambda_3^2
        \end{pmatrix}
        \mapsto\begin{pmatrix}
            0\\0\\1
        \end{pmatrix}.
    \end{align*}

    Let $\phi$ be the linear map on $\mathbb{R}^3$ defined by $\phi(v)=\psi(A\psi^{-1}(v))$.

    Whenever we take a compact neighborhood $K$ for $0\in\mathbb{R}^2$, we can
    take a digit $D$ that satisfies the
    condition in Proposition \ref{prop1_construction_high-dim}.
    For example,
    we consider a ball
    \begin{align*}
        B=\left\{v\in\mathbb{R}^2\mid\|v\|\leqq \frac{1}{4}\right\},
    \end{align*}
    and set $K=\phi(B)$. We can consider the following digit:
    \begin{align*}
        D=\{\pi\psi(0,0,0), \pi\psi(0,1,0), \pi\psi(0,-1,0), \pi\psi(0,-1,2), \pi\psi(0,0,2), \\\pi\psi(0,1,2), 
\pi\psi(0,1,-2), \pi\psi(0,0,-2), \pi\psi(0,-1,-2)\},
    \end{align*}
    where $\pi\colon\mathbb{R}^3\rightarrow\mathbb{R}^2$
    is the projection onto the
    first two coordinates.

    The Delone set $\Lambda$ it generates
    is relatively dense with respect to
    a radius $1/4$ and unifmormly discrete
    by Proposition \ref{prop1_construction_high-dim}.
    The set $\Lambda$ and its Voronoi tiling is depicted in
    Figure \ref{fig:construction_high-dim}. In the figure, a tile, its image by $\phi$, and its image by the corresponding overlapping substitution are
    highlighted.

    \begin{figure}
        \centering
    \includegraphics[width=0.5\linewidth]{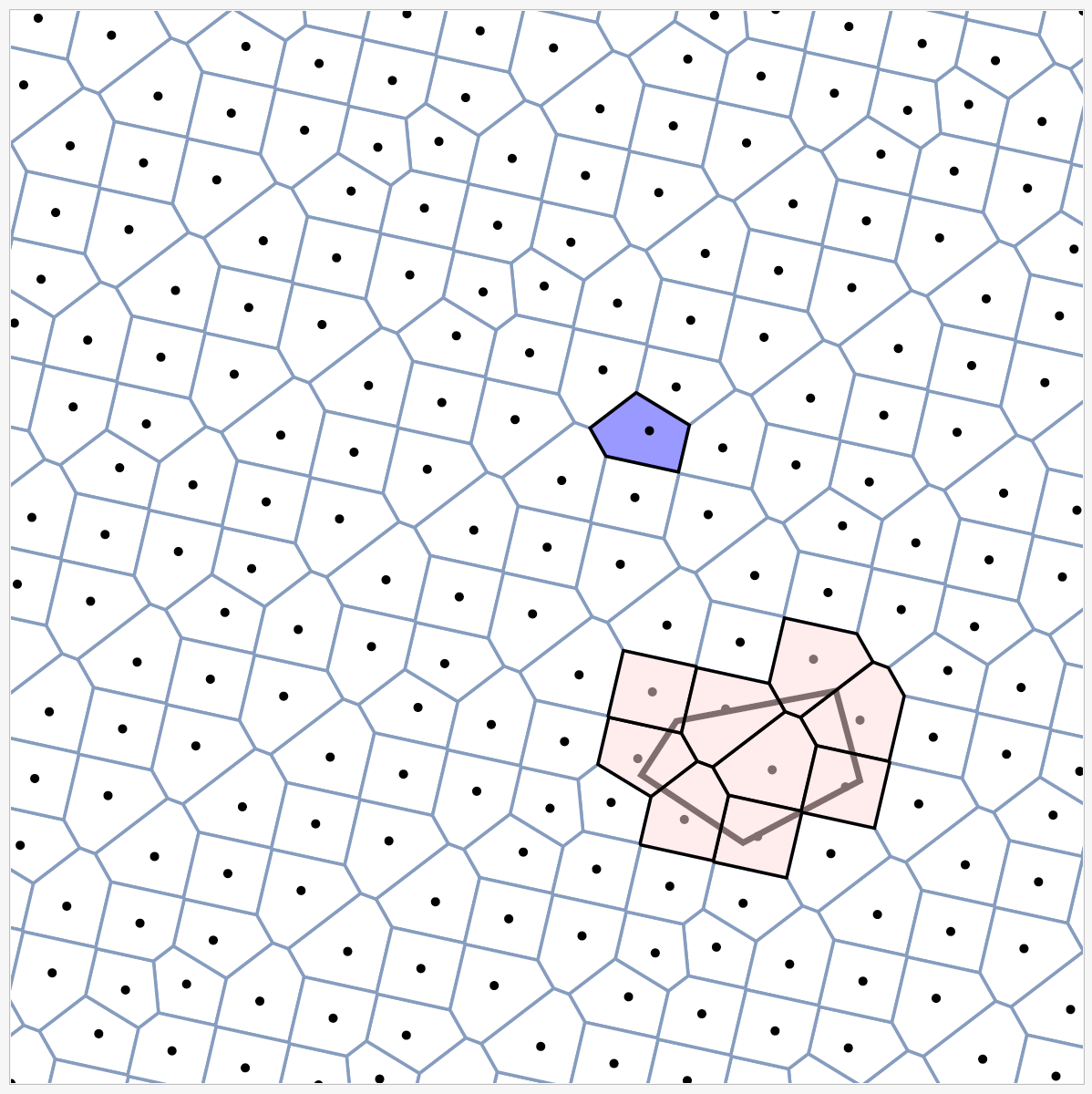}
        \caption{The dots are points in $\Lambda$ and each polygons are tiles in the tiling. For a blue tile, its expansion by the map $\phi$ is depicted in the thick edge and its image by the overlapping substitution is colored in red.}
        \label{fig:construction_high-dim}
    \end{figure}
\end{ex}

\section{Further problems}
\label{OpenProblem}

A self-similar tiling gives an associated IFS.
Our definition of overlapping substitution is 
 restricted to the cases where the overlap occurs only around the boundary of expanded pieces. %Considering the associated IFS problems with overlaps, we discussed only the lucky cases. 
 The problem of Bernoulli convolution is more difficult: it deals with heavy overlaps (c.f. \cite{PSS}), and it is challenging to extend our method.

P.~Gummelt \cite{Gummelt:96} gave a single decagon with special markings, which covers the plane but only in non-periodic ways. This 
covering encodes a version of Penrose tiling by Robinson's triangle.
This covering uses overlaps of pieces but it does not fall into our framework. It is interesting to have a possible generalization of overlap tiling, which includes such an example.

The construction in Theorem \ref{thm_const_from_Delone} is valid for the Delone (multi) sets with inflation symmetry for an expansive linear map such as $Ax=\lambda x$, where $\lambda$ is a Pisot number. However, a similar construction is possible without the Pisot-type assumption: for example, we can take the expansion by an arbitrary real number $\lambda>1$.
For the one-dimension case ($d=1$), we set $\Omega=[0,\infty)$, in order to avoid getting a 
non-discrete subset of $\mathbb{R}$.
 Fix an integer $m\geqq\lambda-1$.
We set
\begin{align*}
 X=\left.\left\{\sum_{j=0}^n\varepsilon_j\lambda^j\right| n\in\mathbb{Z}_{\geqq 0},\varepsilon_j\in\{0,1,\ldots ,m-1,m\} \right\}
\end{align*}
$X$ is a discrete subset of $\Omega$ that satisfies an equation
\begin{align*}
 X=\bigcup_{i=0}^m(\lambda X+i),
\end{align*}
see \cite{Feng-Wen:02}.
By the construction in Theorem \ref{thm_const_from_Delone}, we can get an overlapping substitution $\rho$.
In the case of non-Pisot $\lambda$, the set
 $X$ does not have finite local complexity, and the alphabet forms an infinite set.
However, $\rho$ has ``continuity'', in the sense that if $P,Q$ in the alphabet are ``close''
(the supports and labels are close with respect to Hausdorff metric), then
the resulting patches are ``close'' (each tile in one patch has a counterpart in the other, which is
``close''). Recently, substitutions with an infinite alphabet have been paid attention to.
For example, \cite{FGM} constructed substitutions with a compact infinite alphabet for an arbitrary
real number $\lambda>2$ as the inflation factor.
Our method is significantly different from \cite{FGM}, but also enables us to
construct substitutions with arbitrary $\lambda>1$ as the inflation factor.

% We conjecture that, for an appropriate choice of a finite subset $F$ of $\mathbb{R}^d$,
% an expansive linear map $A$ on $\mathbb{R}^d$, and $\Omega=\{\sum_{f\in F}t_ff\mid t_f>0\}$, 
% \begin{align*}
%  \left\{\sum_{n=1}^NA^na_n\mid N=1,2,\ldots, a_n\in F\right\}
% \end{align*}
% is a Delone set of $\Omega$ 
% with inflation symmetry and we will have a corresponding overlapping substitution.
% In high dimensions, few examples of substitutions are known, but this may give us many 
% overlapping substitutions in such dimensions.

The family of overlapping substitutions with one or two parameters in Section \ref{section_construction_one-dim}
gives examples of a deformation of substitutions.
The deformations of tilings are discussed in \cite{Clark-Sadun}
and some of the deformations of
a fixed point of a substitution
are fixed points of deformed substitutions. It is interesting
to investigate the nature of such special deformations.
%%%%%%%%%%%%%%%%%%%%%%%%%%%%%%%%%%%%%%%%%%%%%%%%%%%%%%%%%%%%%%
%%%%%%%%%%%%%%%%%%%%%%%%%%%%%%%%%%%%%%%%%%%%%%%%

\section{Appendix}\label{appendix}
Here, we prove Proposition \ref{prop_frequency_weighted-substi}.
\subsection{Notation}
First, we fix notation.  Let $v$ be a weighted pattern such that $\SP(v)$ is a tiling and
$v(T)=1$ for each $T\in\SP(v)$. We assume $v$ is a fixed point for an iteration of a weighted
substitution $\xi$, so that there is a $k>0$ such that
$\xi^k(v)=v$. Let $\mathcal{B}=\{T_1,T_2,\ldots ,T_{\kappa}\}$ be the set of tiles for $\xi$ and
$\phi$ be its expansive map. The substitution matrix for $\xi$ is denoted by $M$.
We take a left Perron--Frobenius eigenvector $\mathbf{l}=(l_1,l_2,\ldots ,l_{\kappa})$ and
right eigenvector $\mathbf{r}={}^t(r_1,r_2,\ldots ,r_{\kappa})$ which are normalized in the sense
that
\begin{align*}
 \sum_{i=1}^{\kappa}l_ir_i=\sum_{i=1}^{\kappa}r_i=1.
\end{align*}

\textcolor{blue}{By taking similitude of proto-tiles if necessary, }we may assume that
\begin{align*}
 l_i=\vol(T_i)
\end{align*}
holds for each $i$.

For a weighted pattern $u$ with finite support patch, we set
\begin{align*}
 \vol(u)=\sum_{T\in\SP(u)}u(T)\vol(T).
\end{align*}
In what follows, we consider averaging over van Hove sequences in $\mathbb{R}^d$, 
and it is important
to divide an element $A$ of a van Hove sequence into ``interior'' and ``boundary''.
For each $L>0$, we set
\begin{align*}
 A^{-L}&=\{t\in\mathbb{R}^d\mid B(t,L)\subset A\},\\
\partial^LA&=\{t\in\mathbb{R}^d\mid B(t,L)\cap A\neq\emptyset, B(t,L)\cap A^c\neq\emptyset\}.
\end{align*}

We also consider subsets of  a patch $\SP(v)$, as follows, depending on whether the inflation
of a tile in $\SP(v)$ is inside or not in side of $A$, and on the  ``boundary'' of $A$:
\begin{align*}
 &\mathcal{P}(m,\subset A)=\{T\in\SP(v)\mid \SR(\xi^{km}(T))\subset A\},\\
& \mathcal{P}(m,\not\subset A)=\{T\in\SP(v)\mid \SR(\xi^{km}(T))\not\subset A\},\\
& \mathcal{P}(m,\partial A)=\{T\in\SP(v)\mid \SR(\xi^{km}(T))\not\subset A, 
\SR(\xi^{km}(T))\cap A\neq \emptyset\}.
\end{align*}

Finally, for each natural number $m$, let $L_m$ be the maximal diameter for the
region $\SR(\xi^{km}(T_i))$, $i=1,2,\ldots ,\kappa$.

\subsection{Proof}
\begin{lem}
 We have
\begin{align*}
 \sum_{T\in\mathcal{P}(m,\subset A)}\tau_{T_i}(\xi^{km}(T))\leqq\tau_{T_i}(v\sci A).
\end{align*}
\end{lem}

\begin{proof}
 For $T\in\mathcal{P}(m,\subset A)$, we have $\xi^{km}(T)=\xi^{km}(T)\sci A$, and so
\begin{align*}
 \tau_{T_i}\left(\sum_{T\in\mathcal{P}(m,\subset A)}\xi^{km}(T)\right)
= \tau_{T_i}\left(\sum_{T\in\mathcal{P}(m,\subset A)}\xi^{km}(T)\sci A\right)\leqq\tau_{T_i}(v\sci A),
\end{align*}
where the last inequality is due to $v\geqq\sum_{T\in\mathcal{P}(m,\subset A)}\xi^{km}(T)$.
\end{proof}

\begin{lem}\label{lem1_appendix}
 For a subset $A$ of $\mathbb{R}^d$, a natural number $m$ and $i=1,2,\ldots ,\kappa$, we have
\begin{align*}
 \sum_{T\in\mathcal{P}(m,\not\subset A)}\tau_{T_i}(\xi^{km}(T)\sci A)\leqq\tau_{T_i}(v\sci\partial^{L_m}A).
\end{align*}
\end{lem}

\begin{proof}
 If $T\in\mathcal{P}(m,\not\subset A)$ and $\xi^{km}(T)\sci A\neq 0$, we have
$T\in\mathcal{P}(m,\partial A)$.  We have
\begin{align*}
 \sum_{T\in\mathcal{P}(m,\not\subset A)}\tau_{T_i}(\xi^{km}(T)\sci A)&
\leqq\sum_{T\in\mathcal{P}(m,\partial A)}\tau_{T_i}(\xi^{km}(T)\sci\partial^{L_m}A)\\
&=\tau_{T_i}\left(\sum\xi^{km}(T)\sci\partial^{L_m}A\right)\\
&\leqq\tau_{T_i}\left(\sum_{T\in\SP(v)}\xi^{km}(T)\sci\partial^{L_m}A\right)\\
&=\tau_{T_i}(\xi^{km}(v)\sci\partial^{L_m}A)\\
&=\tau_{T_i}(v\sci \partial^{L_m}A).
\end{align*}
\end{proof}

\begin{lem}\label{lem2_appendix}
 For a set $A\subset\mathbb{R}^d$ and a natural number $m$, we have
\begin{align*}
 0\leqq\vol(A)-\sum_{T\in\mathcal{P}(m,A)}\vol(\xi^{km}(T))\leqq\vol(\partial^{km}A),
\end{align*}
and
\begin{align*}
 \left|\frac{1}{\vol(A)}-\frac{1}{\sum_{T\in\mathcal{P}(m,A)}\vol(\xi^{km}(T))}\right|
\leqq \frac{\vol\partial^{L_m}A}{\vol(A)\vol(A^{-L_m})}.
\end{align*}
\end{lem}
\begin{proof}
 Since $\sum_{S\in\SP(v)}\xi^{km}(S)(T)=v(T)=1$ for each $T\in\SP(v)$, we have
\begin{align*}
 \vol(A)&=\sum_{T\in\SP(v)}\vol(T\cap A)\\
&=\sum_{S,T\in\SP(v)}\xi^{km}(S)(T)\vol(T\cap A)\\
&\geqq\sum_{S\in\mathcal{P}(m,A)}\sum_{T\in\SP(v)}\xi^{km}(S)(T)\vol(T\cap A)\\
&\geqq\sum_{S\in\mathcal{P}(m,A)}\sum_{T\in\SP(v)}\xi^{km}(S)(T)\vol(T)\\
&=\sum_{S\in\mathcal{P}(m,A)}\vol(\xi^{km}(S)),
\end{align*}
where the second inequality follows from the fact that if 
$S\in\mathcal{P}(m,A)$ and $\xi^{km}(S)(T)\neq 0$, then we have $T\subset A$.

We also have
\begin{align*}
 \sum_{S\in\mathcal{P}(m,A)}\sum_{T\in\SP(v)}\xi^{km}(S)(T)\vol(T)
&\geqq\sum_{\substack{T\in\SP(v),\\\supp T\cap A^{-L_m}\neq \emptyset}}
\sum_{S\in\mathcal{P}(m,A)}\xi^{km}(S)(T)\vol(T)\\
&\geqq\sum_{\substack{T\in\SP(v),\\\supp T\cap A^{-L_m}\neq \emptyset}}
\sum_{S\in\SP(v)}\xi^{km}(S)(T)\vol(T)\\
&\geqq\sum_{\substack{T\in\SP(v),\\\supp T\cap A^{-L_m}\neq \emptyset}}\vol(T)\\
&\geqq\vol(A^{-L_m}),
\end{align*}
by which the first two inequalities are proved.
The third inequality is an easy consequence of the first two and an inequality
\begin{align*}
 \sum_{T\in\mathcal{P}(m,A)}\vol(\xi^{km}(T))\geqq\vol(A^{-L_m}),
\end{align*}
which was proved above.
\end{proof}

\begin{prop}
 Let $(A_n)$ be a van Hove sequence in $\mathbb{R}^d$ and $i=1,2,\ldots ,\kappa$.
Then we have a convergence
\begin{align*}
 \lim_{n\rightarrow\infty}\frac{1}{\vol(A_n)}\tau_{T_i}(v\sci(A_n+x_n))=r_i,
\end{align*}
for arbitrary $x_n\in\mathbb{R}^d$, and the convergence is uniform for the choice of
$(x_n)_{n=1,2,\ldots}$.
\end{prop}

\begin{proof}
 We first observe that
\begin{align*}
 v\sci (A_n+x_n)&=\sum_{T\in\SP(v)}\xi^{km}(T)\sci(A_n+x_n)\\
&=\sum_{T\in\mathcal{P}(m,\subset A_n+x_n)}\xi^{km}(T)\sci(A_n+x_n)
+\sum_{T\in\mathcal{P}(m,\not\subset A_n+x_n)}\xi^{km}(T)\sci(A_n+x_n).
\end{align*}

(I) We have a convergence
\begin{align*}
& \frac{1}{\vol(A_n)}\tau_{T_i}\left(\sum_{T\in\mathcal{P}(m,\not\subset(A_n+x_n))}\xi^{km}(T)\sci (A_n+x_n)\right)\\
&\leqq\frac{1}{\vol(A_n)}\tau_{T_i}(v\sci\partial^{L_m}(A_n+x_n))\\
&\leqq \frac{1}{\vol(A_n)}\frac{\vol(\partial^{L_m}(A_n+x_n))}{\vol(T_i)}\\
&\rightarrow 0,
\end{align*}
as $n\rightarrow\infty$,
where the first inequality is due to Lemmma \ref{lem1_appendix}. Here, the convergence is uniform
for $(x_n)_n$.

(II) We also have a convergence
\begin{align*}
& \left|\frac{1}{\vol(A_n)}-\frac{1}{\sum_{S\in\mathcal{P}(m,\subset A_n+x_n)}\vol(\xi^{km}(S))}\right|\sum_{T\in\mathcal{P}(m,\subset A_n+x_n)}\tau_{T_i}(\xi^{km}(T))\\
&\leqq \frac{\vol(\partial^{L_m}(A_n+x_n))}{\vol(A_n)\vol(A_n+x_n)^{-L_m}}\frac{\vol(A_n+x_n)}{\vol T_i}\\
&\leqq\frac{\vol\partial^{L_m}A_n}{\vol(T_i)\vol(A_n^{-L_m})}\\
&\rightarrow 0,
\end{align*}
as $n\rightarrow \infty$, where the first inequality is due to Lemma \ref{lem2_appendix}.
Here, the convergence is uniform for $(x_n)_n$.

(III) 
We have
\begin{align}
 &\frac{\sum_{T\in\mathcal{P}(m,\subset A_n+x_n)}\vol(\tau_{T_i}(\xi^{km}(T)))}{\sum_{S\in\mathcal{P}(m,\subset A_n+x_n)}\vol(\xi^{km}(S))}-r_i\nonumber\\
=&\sum_{T\in\mathcal{P}(m,\subset A_n+x_n)}\frac{\vol(\xi^{km}(T))}{\sum_{S\in\mathcal{P}(m,\subset(A_n+x_n))}\vol(\xi^{km}(S))}\left(\frac{\tau_{T_i}(\xi^{km}(T))}{\vol(\xi^{km}(T))}-r_i\right).
\label{eq1_appendix}
\end{align}
If $T$ is a translate of $T_j$, we have
\begin{align*}
   \left|\frac{\tau_{T_i}(\xi^{km}(T))}{\vol(\xi^{km}(T))}-r_i\right|
&=\left|\frac{(M^{km})_{ij}}{\sum_{l=1}^{\kappa}\vol(T_l)(M^{km})_{lj}}-r_i\right|\\
&=\left|\frac{(M^{km})_{ij}}{\lambda^{km}\vol(T_j)}-r_i\right|
\end{align*}
and since $\vol(T_j)=l_j$, for any $\varepsilon>0$
this is smaller than $\varepsilon$ if $m$
is large enough.
The absolute value of \eqref{eq1_appendix} is smaller than arbitrary $\varepsilon>0$ if
$m$ is large enough.

By (I),(II) and (III), for any $\varepsilon$, if $m$ is large enough and $n$ is also large enough
(depending on the value of $m$), we have
\begin{align*}
& \left|\frac{1}{\vol(A_n)}\tau_{T_i}(v\sci (A_n+x_n))-r_i\right|\\
&\leqq\left |\frac{1}{\vol(A_n)}\tau_{T_i}\left(\sum_{T\in\mathcal{P}(m,\not\subset A_n+x_n)}\xi^{km}(T)\sci(A_n+x_n)\right)\right|\\
&+\left|\frac{1}{\vol(A_n)}-\frac{1}{\sum_{S\in\mathcal{P}(m,A_n+x_n)}\vol(\xi^{km}(S))}\right|
\tau_{T_i}\left(\sum_{T\in\mathcal{P}(m,\subset A_n+x_n)}\xi^{km}(T)\right)\\
&+\left|\frac{\sum_{T\in\mathcal{P}(m,\subset A_n+x_n)}\tau_{T_i}(\xi^{km}(T))}{\sum_{S\in\mathcal{P}(m,A_n+x_n)}\vol(\xi^{km}(S))}-r_i\right|\\
&<3\varepsilon
\end{align*}
\end{proof}

\section*{Acknowledgments}
SA and YN are supported by JSPS grants (17K05159, 20K03528, 21H00989, 23K12985). S-Q Zhang is supported by National Natural Science
Foundation of China (Grant No. 12101566).


\begin{thebibliography}{1}

\bibitem{Akiyama-Loridant:11}
S.~Akiyama and B.~Loridant, \emph{Boundary parametrization of self-affine
  tiles}, J. Math. Soc. Japan \textbf{63} (2011), no.~2, 525--579.

\bibitem{Akiyama-Arnoux:20}
S.~Akiyama and P.~Arnoux (ed.),
\emph{{S}ubstitution and tiling dynamics: introduction to
              self-inducing structures},
Lecture Notes in Mathematics, \textbf{2273},
Springer, Cham, 2020, xix+453

\bibitem{Baake-Grimm:13}
M.~Baake and U.~Grimm, \emph{Aperiodic {O}rder. {V}ol. 1}, Encyclopedia of
  Mathematics and its Applications, vol. 149, Cambridge University Press,
  Cambridge, 2013.

\bibitem{Manibo}
M.~ Baake, F.~ G\"ahler and N.~ Ma\~nibo,
Renormalisation of Pair Correlation Measures for Primitive Inflation Rules
and Absence of Absolutely Continuous Diffraction,
Communications in Mathematical Physics \textbf{370}(2019), 591-635.

\bibitem{Bandt-Mekhontsev:18}
C.~Bandt, D.~Mekhontsev, and A.~Tetenov, \emph{A single fractal pinwheel tile},
  Proc. Amer. Math. Soc. \textbf{146} (2018), no.~3, 1271--1285.
\bibitem{Clark-Sadun}
A.~Clark, L.~Sadun,
\emph{When size matters: subshifts and their related tiling spaces},
Ergod. Th. \& Dynam. Sys. \textbf{23} (2003), 1043--1057.

\bibitem{Falconer:97}
K.~J. Falconer, \emph{Techniques in fractal geometry}, John Wiley and Sons,
  Chichester, New York, Weinheim, Brisbane, Singapore, Toronto, 1997.

\bibitem{Feng-Wen:02}
D.~J. Feng and Z.Y.Wen, 
\emph{A property of Pisot numbers}, J. Number Theory, \textbf{97} (2002), no.2, 305--316.

\bibitem{Fogg:02}
 N. Pytheas Fogg,V.~Berthé, S.~Ferenczi, C.~Mauduit, and A.~Siegel,
\emph{Substitutions in {D}ynamics, {A}rithmetics and {C}ombinatorics}, Lecture Notes in Mathematics \textbf{1794}, Springer-Verlag, Berlin 2002

\bibitem{Frettloeh:11}
D. Frettl\"oh, 
A fractal fundamental domain with 12-fold symmetry, Symmetry: Culture and Science \textbf{22} (2011) 237-246.

\bibitem{FGM}
D. Frettlöh, A. Garber and N. Ma\~nibo,
Substitution tilings with transcendental inflation factor,
arXiv:2208.01327

\bibitem{Gruenbaum-Shephard:87}
B.~Gr{\"u}nbaum and G.~C. Shephard, \emph{Tilings and patterns}, W. H. Freeman
  and Company, New York, 1987.

\bibitem{Gummelt:96}
P.~Gummelt, 
Penrose Tilings as Coverings of Congruent
Decagons, Geometriae Dedicata 
\textbf{62} volume 1 (1996) 1--17.

\bibitem{Kamae:05}
T.~Kamae, \emph{Numeration systems, fractals and stochastic processes},
Israel J. Math. \textbf{149} (2005), 87--135.

\bibitem{Keane:71}
M.~Keane, \emph{The structure of substitution minimal sets}, Trans. Amer. Math. Soc. \textbf{162} (1971), 89--100.


\bibitem{Hata:85}
M.~Hata,
\emph{On the Structure of Self-Similar Sets},
Japan J. Appl. Math. \textbf{2} (1985), 381--414.

\bibitem{Lagarias:99}
J.~Lagarias, \emph{Geometric Models for Quasicrystals {I}. {D}elone Sets of Finite Type},
Discrete Comput. Geom., \textbf{21-2} (1999), 161--191.

\bibitem{Lagarias-Wang:03}
J.~C. Lagarias and Y.~Wang, \emph{Substitution {D}elone {S}ets}, Discrete
  Comput. Geom. \textbf{29} (2003), 175--209.

\bibitem{Luo-Akiyama-Thuswaldner:04}
J.~Luo, S.~Akiyama and J.~M.~Thuswaldner, \emph{On the boundary connectedness of connected tiles},
Math. Proc. Cambridge Phil. Soc. \textbf{137} (2004), no. 2, 397--410.

\bibitem{Ledrappier:92}
F.~Ledrappier, \textit{Some properties of the spectrum of dynamical systems arising from substitutions}, Number Theory and Physics, 1992, 305--316.

\bibitem{NZD24}
J. Nakakura, P.~Ziherl, T.~Dotera, \emph{Fourfold metallic-mean quasicrystals as aperiodic approximants of the square lattice},Physical Review B \textbf{110} (2024), 10 pages.
\bibitem{MKD25}
T. Matsubara, A. Koga, T.~Dotera, \emph{Triangular and dice quasicrystals modulated by generic one-dimensional aperiodic sequences}, Physical Review B \textbf{111} (2025), 12 pages.

\bibitem{MKTMD24}
T. Matsubara, A. Koga, A. Takano, Y. Matsushita, T.~Dotera, \emph{Aperiodic approximants bridging quasicrystals and modulated structures}, Nature Communications \textbf{15} (2024),5742.



\bibitem{Shechtman:84}
D.~Shechtman, I.~Blech, D.~Gratias, and J.W.~Cahn, 
Metallic phase with long-range orientational order and no
translational symmetry, Phys. Rev. Lett. \textbf{53} (1984), no. 20, 1951--1953.

\bibitem{Solomyak:97}
B.~Solomyak, Dynamics of self-similar tilings, Ergodic Theory Dynam. Systems \textbf{17}
(1997), no. 3, 695--738.

\bibitem{Solomyak_pseudo-self-similar}
B.~Solomyak, Pseudo-self-affine tilings in $\mathbb{R}^d$, J. Mathematical Sciences, \textbf{140}
(2007), 452--460.

\bibitem{Sirvent-Solomyak:02}
V.~F.~Sirvent and B.~Solomyak, \emph{Pure discrete spectrum for one-dimensional substitution systems of Pisot type}, Canadian Mathematical Bulletin, \textbf{45} (2002), no.~4, 697--710.


\bibitem{DBZ}
P.~Ziherl, T.~Dotera, S.~Bekku, \emph{Bronze-mean hexagonal quasicrystal},
  Nature Materials \textbf{16} (2017), 987--992.
\bibitem{NZD}
J. Nakakura, P.~Ziherl, T.~Dotera, \emph{Fourfold metallic-mean quasicrystals as aperiodic approximants of the square lattice},Physical Review B \textbf{110} (2024), 10 pages.
\bibitem{MKD}
T. Matsubara, A. Koga, T.~Dotera, \emph{Triangular and dice quasicrystals modulated by generic one-dimensional aperiodic sequences}, Physical Review B \textbf{111} (2025), 12 pages.

\bibitem{MKTMD}
T. Matsubara, A. Koga, A. Takano, Y. Matsushita, T.~Dotera, \emph{Aperiodic approximants bridging quasicrystals and modulated structures}, Nature Communications \textbf{15} (2024),5742.

\bibitem{PSS}
Y.~Peres, W.~Schlag, B.~Solomyak,
\emph{
Sixty Years of Bernoulli Convolutions},
Progress in Probability, Vol. 46 (2000)
Birkh{\"a}user Verlag Basel/Swizerland.

\bibitem{Queffelec:87}
M.~Queff{\'e}lec,
\emph{Substitution {D}ynamical {S}ystems---{S}pectral analysis},
Lecture Notes in Mathematics, \textbf{1294}, Springer-Verlag, Berlin 1987


\end{thebibliography}
\end{document}